\documentclass[a4paper,reqno,11pt]{amsart}

\usepackage[left=1 in, right=1 in,top=1 in, bottom=1 in]{geometry}

\usepackage{amsfonts,amsmath,amsthm,amssymb,stmaryrd}
\usepackage{mathrsfs}
\usepackage{color}
\allowdisplaybreaks

\newcommand{\beq}{\begin{equation}}
\newcommand{\eeq}{\end{equation}}
\newcommand{\bal}{\begin{align}}
\newcommand{\eal}{\end{align}}

\renewcommand{\(}{\left(}         
 \renewcommand{\)}{\right)}
\renewcommand{\[}{\left[}         
 \renewcommand{\]}{\right]}

\newcommand{\D}{\partial^{\varphi}}
\newcommand{\De}{\partial^{\varphi^\eep}}
\newcommand{\Dn}{\nabla^\varphi}
\newcommand{\V}{\mathcal{V}_{t}}

\newcommand{\N}{\mathcal{N}} 
\newcommand{\F}{\mathcal F} 

\DeclareMathOperator{\diverge}{div}

\providecommand{\ns}[1]{\norm{#1}^2}
\providecommand{\as}[1]{\abs{#1}^2}
\providecommand{\abs}[1]{\left\vert#1\right\vert}
\providecommand{\norm}[1]{\left\Vert#1\right\Vert}
\providecommand{\Rn}[1]{\mathbb{R}^{#1}}

\providecommand{\se}[1]{\mathcal{E}_{#1}}
\providecommand{\fe}[1]{\mathfrak{E}_{#1}}
\providecommand{\fd}[1]{\mathfrak{D}_{#1}}  
\providecommand{\sdb}[1]{\bar{\mathcal{D}}_{#1}}
\providecommand{\seb}[1]{\bar{\mathcal{E}}_{#1}}

\providecommand{\fdb}[1]{\bar{\mathfrak{D}}_{#1}}

\providecommand{\br}[1]{\left\langle #1 \right\rangle}

\providecommand{\jump}[1]{\left\llbracket #1 \right\rrbracket }
 \providecommand{\jumps}[1]{\llbracket #1 \rrbracket }

\def\pa{\partial} 
\def\nabb{\bar B\cdot \nabla}
\def\hee{\hat E}
\def\hb{\hat b}
\def\hv{\hat v}
\def\hf{\hat f}
\def\bb{{\bf b}}
\def\hbb{{\hat{ B}}}
\def\hhbb{{\hat{ {\bf b}}}}
\def\he{\hat E}

\def\dt{\partial_t}
\def\dtt{ \frac{d}{dt}}
\def\hal{\frac{1}{2}}

\def\ls{\lesssim} 

\def\eep{\epsilon}
\def\epp{\epsilon}

\def\x{\mathcal{X}}
\def\y{\mathcal{Y}}

\def\i{\mathcal{I}}
  
\def\al{\alpha}  
\def\p{\partial}
\def\dis{\displaystyle}  
\def\jj{\mathcal{J}} 
\def\mt{\mathbb{T}^2}  
\def\pav{\pa^\varphi}
\def\nav{\nabla^\varphi}

\def\be{{  b^\eep}} 
\def\hbe{{\hat b^\eep}}  
\def\bb{{\bf b}} 
\def\BB{{\bf B}}

\def\nabb{\bar B\cdot \nabla}  

\def\sdn{\mathcal{D}_{2N}}

\def\sen{\mathcal{E}_{2N}}

\def\sdn{\mathbb{D}_{2N}}

\def\diva{\diverge^{\varphi}} 
\def\divaz{\diverge^{\varphi_0}} 
\def\divae{{\rm div}^{\varphi^\eep}} 
\def\divaez{{\rm div}^{\varphi_0^\eep}} 
\def\divam{\diverge^{\varphi^m}} 
\def\divaml{\diverge^{\varphi^{m+1}}} 
\def\divamll{\diverge^{\varphi^{m-1}}} 
\def\curlv{{\rm curl}^\varphi}
\def\curlvz{{\rm curl}^{\varphi_0}}
\def\curlve{{\rm curl}^{\varphi^\eep}}
\def\curlvez{{\rm curl}^{\varphi_0^\eep}}
\def\curlvm{{\rm curl}^{\varphi^m}}

\def\curlvmll{{\rm curl}^{\varphi^{m-1}}}
\def\n{\mathcal{N}} 
\def\ne{\mathcal{N}^\eep} 
 
\def\nm{\mathcal{N}^m} 
\def\nml{\mathcal{N}^{m+1}}

\newcommand{\Div}{\operatorname{div}}
\newcommand{\curl}{\operatorname{curl}}

\newtheorem{prop}{Proposition}
\newtheorem{thm}[prop]{Theorem}
\newtheorem{lem}[prop]{Lemma}

\newtheorem{rem}[prop]{Remark}

\numberwithin{equation}{section}
\numberwithin{prop}{section}

\newcommand{\Dm}{\partial^{\varphi^m}}
\newcommand{\Dml}{\partial^{\varphi^{m+1}}}
\newcommand{\Dmll}{\partial^{\varphi^{m-1}}}
\newcommand{\Dnm}{\nabla^{\varphi^m}}
\newcommand{\Dnml}{\nabla^{\varphi^{m+1}}}

\title[Inviscid resistive plasma interface problems]{Global Well-posedness of Free Interface Problems for the incompressible Inviscid Resistive MHD}

\author{Yanjin Wang}
\address{School of Mathematical Sciences\\
	Xiamen University\\
	Xiamen, Fujian 361005, China
	\newline \indent and
	\newline \indent The Institute of Mathematical Sciences\\
	The Chinese University of Hong Kong\\
	Shatin, NT, Hong Kong}
\email[Y. J. Wang]{yanjin$\_$wang@xmu.edu.cn}

\author{Zhouping Xin}
\address{The Institute of Mathematical Sciences\\
	The Chinese University of Hong Kong\\
	Shatin, NT, Hong Kong}
\email[Z. P. Xin]{zpxin@ims.cuhk.edu.hk} 

\subjclass[2010]{35Q35, 35R35, 76B03, 76B15, 76W05.}

\keywords{Free boundary problems; Plasma interface; MHD; Inviscid fluids; Magnetic diffusion.}

\date{\today}

\begin{document}

\begin{abstract}
We consider the plasma-vacuum interface problem in a horizontally periodic slab impressed by a uniform non-horizontal magnetic field. The lower plasma region is governed by the incompressible inviscid and resistive MHD, the upper vacuum region is governed by the pre-Maxwell equations, and the effect of surface tension is taken into account on the free interface. The global well-posedness of the problem, supplemented with physical boundary conditions, around the equilibrium is established, and the solution is shown to decay to the equilibrium almost exponentially. Our results reveal the strong stabilizing effect of the magnetic field as the global well-posedness of the free-boundary incompressible Euler equations, without the irrotational assumption, around the equilibrium is unknown. One of the key observations here is an induced damping structure for the fluid vorticity due to the resistivity and transversal magnetic field. A similar global well-posedness for the plasma-plasma interface problem is obtained, where the vacuum is replaced by another plasma.
\end{abstract}

\maketitle

%\tableofcontents
%%%%%%%%%%%%%%%%%%%%%%%%%%%%%%%%%%%%%%%%%%%%%%%%%%%%%%
\section{Introduction}
%%%%%%%%%%%%%%%%%%%%%%%%%%%%%%%%%%%%%%%%%%%%%%%%%%%%%%

%%%%%%%%%%%%%%%%%%%%%%%%%%%%%%%%%%%%%%%%%%%%%%%%%%%%%%
\subsection{Formulation in Eulerian coordinates}
%%%%%%%%%%%%%%%%%%%%%%%%%%%%%%%%%%%%%%%%%%%%%%%%%%%%%%

We consider the plasma-vacuum interface problem  in the slab $\Omega=\mt \times(-1,1)$ impressed by a uniform transversal magnetic field $\bar B$, $i.e.$, $\bar B_3\neq0$, where  the slab  is assumed to be horizontally periodic for $ \mathbb{T} =\mathbb{R}/\mathbb{Z}$. Let $\Sigma_\pm:=\mt \times\{\pm 1\}$ be the upper and lower fixed boundaries, respectively.   The plasma moves in the lower domain
\beq
\Omega_-(t)= \left\{ y=(y_h,y_3):=(y_1,y_2,y_3) \in \mt \times\mathbb{R} \mid  -1< y_{3} < \eta(t, y_h) \right\},
\eeq
 the vacuum occupies the upper domain
\beq
\Omega_+(t)= \left\{ y \in \mt \times\mathbb{R} \mid   \eta(t, y_h)< y_{3} <1  \right\},
\eeq
and the interface $
\Sigma(t):= \left\{ y \in \mt \times\mathbb{R} \mid   y_{3}= \eta(t, y_h) \right\}
$ is  free to move, where  the
graph function $ \eta:\mathbb{R}_+\times\mt \to \Rn{}$ is  unknown.
We assume that the velocity $u$, the pressure $p$ and the magnetic field $B$ of the plasma satisfy the incompressible inviscid and resistive magnetohydrodynamic equations (MHD):
\beq\label{MHD00}
\begin{cases}
\partial_t {u}  +    {u}\cdot \nabla  {u} +\nabla p =  \curl B\times B& \text{in } \Omega_-(t)
\\ \Div u=0 &\text{in }\Omega_-(t)
\\ \dt B        =  \curl   E,\quad  E= u\times B-\kappa\curl  B& \text{in } \Omega_-(t)
\\ \Div B=0 &\text{in }\Omega_-(t),
\end{cases}
\eeq
where $E$  is the electric field of the plasma and  $\kappa>0$ is the magnetic diffusion coefficient, the inverse of the electric conductivity. The magnetic field $\hbb$ and the electric field $\hee$ in vacuum are assumed to satisfy the {\it pre}-Maxwell equations:
\beq\label{MHD2}
\begin{cases}
\curl\hbb=0,\quad\Div \hbb=0 &\text{in }\Omega_+(t)\\
\dt \hbb=\curl\hee,\quad \Div \hee=0&\text{in }\Omega_+(t).
\end{cases}
\eeq
 The interface $\Sigma(t)$ is adverted with the plasma through the kinematic boundary condition:
\beq\label{MHDb0}
\dt \eta= u\cdot \n\quad \text{on } \Sigma(t),
\eeq
where  
$\n =  (-\nabla_h \eta ,1) 
$ is the upward non-unit normal vector to $\Sigma(t)$ with $\nabla_h=(\pa_1,\pa_2)$  the horizontal gradient.   
Note that the equations \eqref{MHD00} and \eqref{MHD2} are derived by neglecting the displacement current in the  {\it full}-Maxwell equations, and one may refer to the books \cite{R,C,F,D,GP} for the physical backgrounds and applications.

To solve \eqref{MHD00} and \eqref{MHD2}, one needs to impose certain physical boundary conditions. First, due to 
$\curl B\times B=-\diverge (\hal \abs{B}^2 I-B\otimes B),$   the dynamic boundary condition of the balance of the normal stresses on the free interface    reads as
\beq\label{bdc1}
\( p I +\hal \abs{B}^2I-B\otimes B \)\n =\(\hal |\hbb|^2I-\hbb\otimes \hbb\)\n -\sigma H\n\quad \text{on } \Sigma(t),
\eeq
where $I$ is the $3\times 3$ identity matrix,   $\sigma>0$ is the surface tension coefficient and $H$ is the mean curvature  of $\Sigma(t)$ given by
\beq\label{HHdef}
H=\Div_h\(\frac{\nabla_h\eta}{\sqrt{1+|\nabla_h\eta|^2}}\).
\eeq
Here $\Div_h$ is the horizontal divergence, and also the notation $\Delta_h=\Div_h\nabla_h$ will be used later.
Next, the classical jump conditions for the magnetic and electric fields, which follow from the Maxwell equations (see \cite{GP,F}), are
\beq\label{bdc2}
  B\cdot \n = \hbb\cdot \n, \quad (E-\hee)\times \n =u\cdot\n (B-\hbb)\quad\text{on }\Sigma(t).
  \eeq
Due to the consideration in this paper that the problem is around the uniform traversal magnetic field $\bar B$, $B\cdot \n= \hbb\cdot \n\neq0$  on $\Sigma(t)$,  and hence the tangential components of  the jump condition \eqref{bdc1} imply in particular that $B\times \n=\hbb\times \n$  on $\Sigma(t)$, and one then finds  that \eqref{bdc1} and \eqref{bdc2} are equivalent to the following  boundary conditions:
\beq\label{MHDb1}
p =   - \sigma H,\quad B=\hbb,\quad E\times \n=\hee\times \n      \quad \text{on } \Sigma(t).
\eeq
Finally, we impose  the impermeable condition on the lower fixed boundary:
\beq
u\cdot e_3=0\quad\text{on }\Sigma_-,
\eeq
with $e_3=(0,0,1)$. It should be emphasized that the boundary conditions of the magnetic and electric fields on a fixed boundary  depend on the nature of the boundary, see \cite{R,C,F,D,GP};  we assume that the lower fixed boundary  is a perfect conducting wall, so
\beq\label{MHDb3}
B\cdot e_3=\bar B\cdot e_3,\quad E \times e_3=0 \quad \text{on }\Sigma_{-},
\eeq
while  the upper fixed boundary is a perfect insulating wall, then
\beq\label{MHDb4}
\hbb\times e_3=\bar B\times e_3,\quad  \hee\cdot e_3=0 \quad \text{on }\Sigma_{+}.
\eeq
One  may refer to \cite{St} for more physical and mathematical discussions on the boundary conditions of the magnetic and electric fields and related literature.

Mathematically, the electric field in vacuum $\he$ could be regarded as a second variable, see Ladyzhenskaya and Solonnikov \cite{LS1,LS}. Indeed,  one can eliminate $\he$ from the problem under consideration, $i.e.$, \eqref{MHD00}--\eqref{MHDb0} and \eqref{MHDb1}--\eqref{MHDb4},  and  arrive at  the following system for $(u,p,\eta,b,\hb)$, with $b=B-\bar B$ and $\hb=\hbb-\bar B$,
\beq\label{MHD}
\begin{cases}
\partial_t {u}  +    {u}\cdot \nabla  {u} +\nabla p =  \curl b\times (\bar B+b)& \text{in } \Omega_-(t)
\\ \Div u=0 &\text{in }\Omega_-(t)
\\ \dt b        =  \curl   E,\quad  E= u\times (\bar B+b)-\kappa\curl  b& \text{in } \Omega_-(t)
\\ \Div b=0 &\text{in }\Omega_-(t)
\\\curl\hb =0,\quad\Div \hb =0 &\text{in }\Omega_+(t)
\\ {\dt \eta} = u\cdot \n  & \text{on } \Sigma(t)
\\p =   - \sigma H,\quad b=\hb     & \text{on } \Sigma(t)
\\ u_3=0,\quad b_3=0,\quad E \times e_3=0 & \text{on }\Sigma_{-}
\\\hb \times e_3=0 & \text{on }\Sigma_{+}.
\end{cases}
\eeq
Once \eqref{MHD} is solved, then $\he$ can be recovered from the following problem:
\beq\label{MHDvk02}
\begin{cases}
\curl\he=\dt\hb ,\quad\Div \he =0 &\text{in }\Omega_+(t)
\\\he\times \n=E\times \n  &\text{on }\Sigma(t)
\\\hee_3=0 &\text{on }\Sigma_{+}.
\end{cases}
\eeq
Note also that the magnetic field in vacuum $\hb$ is completely determined by $b\cdot\n$ on $\Sigma(t)$  via the following  problem:
\beq\label{MHDv"}
\begin{cases}
\curl \hb=0,\quad\diverge \hb =0 &\text{in }\Omega_+(t)  
\\ \hb \cdot\n= b\cdot\n & \text{on } \Sigma(t)
\\  \hb \times e_3=0   & \text{on } \Sigma_{+}.
\end{cases}
\eeq
Then the jump condition $b=\hb $  on $\Sigma(t)$   could be regarded as a nonlocal boundary condition   for $b$ (see \cite{LS1,LS}):
\beq
b\times \n =\mathcal{B}^t(b\cdot\n)\times\n\quad \text{on } \Sigma(t),
\eeq
where $\mathcal{B}^t(b\cdot\n)$ is the solution to  \eqref{MHDv"}. Thus one could further formally suppress  $\hb $ in \eqref{MHD}.

To complete the statement of the problem \eqref{MHD}, one must specify the initial conditions. Suppose that the initial interface $\Sigma(0)$ is given by the graph of the function $\eta(0)=\eta_0: \mt \rightarrow\mathbb{R}$, which yields the initial lower domain $\Omega_-(0)$ on which the initial velocity $u(0)=u_0: \Omega_-(0) \rightarrow \mathbb{R}^3$, and the initial magnetic field  $b(0)=b_0: \Omega_-(0)\rightarrow \mathbb{R}^3$ are specified.  %It is assumed that $-1<\eta_0<1$ and that $(u_0,b_0,\eta_0)$ satisfy certain necessary compatibility conditions, which will be described later.

%%%%%%%%%%%%%%%%%%%%%%%%%%%%%%%%%%%%%%%%%%%%%%%%%%%%%%
\subsection{Physical energy-dissipation law}
%%%%%%%%%%%%%%%%%%%%%%%%%%%%%%%%%%%%%%%%%%%%%%%%%%%%%%

The problem \eqref{MHD} possesses a natural physical energy-dissipation law. First, as for the free-surface incompressible Euler equations, one has
\begin{align}\label{en_iden0}
 & \hal\dtt \(  \int_{\Omega_-(t)}  |u|^2\, dy   +   \int_{\mt } 2\sigma\(\sqrt{1+|\nabla_h \eta|^2}-1\) dy_h  \)  \nonumber
 \\&\quad=    \int_{\Omega_-(t)} \curl    b\times (\bar B+b) \cdot u\, dy
 =- \int_{\Omega_-(t)}u\times (\bar B+b) \cdot  \curl    b \, dy.
\end{align}
Next, to handle the magnetic system in \eqref{MHD}, making use of the electric field in vacuum $\he$ satisfying \eqref{MHDvk02}, one has
\begin{align}\label{en_iden1}
& \hal\dtt  \( \int_{\Omega_-(t)} |b|^2 \,dy  +\int_{\Omega_+(t)}   |\hb|^2   \,dy\) =      \int_{\Omega_-(t)} \dt b\cdot b \, dy+ \int_{\Omega_+(t)} \dt \hb\cdot\hb\, dy \nonumber
 \\&\quad=     \int_{\Omega_-(t)} \curl   E\cdot b\, dy+ \int_{\Omega_+(t)}  \curl \he \cdot \hb\, dy=\int_{\Omega_-(t)}    E\cdot \curl b \,dy
  \nonumber
 \\&\quad=  \int_{\Omega_-(t)} u\times (\bar B+b) \cdot  \curl   b\,dy- \kappa  \int_{\Omega_-(t)} |\curl b|^2\, dy.
\end{align}
It then follows from \eqref{en_iden0} and \eqref{en_iden1} that
\begin{align}\label{en_idenphy}
& \hal\dtt \(  \int_{\Omega_-(t)} \(|u|^2+|b|^2 \) dy +\int_{\Omega_+(t)}   |\hb|^2   \,dy +    \int_{\mt } 2\sigma\(\sqrt{1+|\nabla_h \eta|^2}-1\)  dy_h  \)  \nonumber
\\&\quad+ \kappa  \int_{\Omega_-(t)} |\curl b|^2\, dy =0.
\end{align}
This structure of the energy evolution equation is the basis of the energy method we will use to analyze the problem \eqref{MHD}. 

Note that \eqref{en_idenphy} can be derived in an alternative way,  motivated by  Ladyzhenskaya and Solonnikov \cite{LS1,LS}, that does not involve the electric field in vacuum $\he$. The idea is to introduce instead a virtual magnetic field $  \mathfrak{b}$ in $\Omega_-(t)$ as the solution to
 \beq\label{MHDv'}
\begin{cases} 
\curl \mathfrak{b}=0,\quad\Div  \mathfrak{b}=0 &\text{in }\Omega_-(t) 
\\ \mathfrak{b}\times \n= \hb\times \n   & \text{on } \Sigma(t)
\\\mathfrak{b}_3 =0  & \text{on } \Sigma_{-}.
\end{cases}
\eeq
 Then one may write  $\hb=\nabla  \hat\phi$   and $\mathfrak{b}=\nabla  {\bf \phi}$    with  $\hat\phi$ solving 
\beq
\Delta \hat\phi=0 \text{ in }\Omega_+(t),\quad 
\nabla \hat\phi\cdot \n=b\cdot \n  \text{ on }\Sigma(t),\quad \hat\phi=0  \text{ on }\Sigma_+
\eeq
and $\phi$ satisfying
\beq
\Delta \phi=0 \text{ in }\Omega_-(t),\quad 
  \phi=\hat\phi  \text{ on }\Sigma(t),\quad \p_3\phi=0  \text{ on }\Sigma_-.
\eeq
Note that  $ b\times \n=\mathfrak{b}\times \n $ on $\Sigma(t)$. Then one has
 \begin{align}\label{en_iden12}
& \hal\dtt  \( \int_{\Omega_-(t)} |b|^2 dy  +\int_{\Omega_+(t)}   |\hb|^2   \,dy\) =           \int_{\Omega_-(t)} \dt b\cdot (b-\mathfrak{b})\, dy+    \int_{\Omega_-(t)} \dt b\cdot \mathfrak{b}\, dy+ \int_{\Omega_+(t)} \dt \hb\cdot\hb\, dy \nonumber
\\&\quad=     \int_{\Omega_-(t)} \curl   E\cdot  (b-\mathfrak{b})\, d\V+    \int_{\Omega_-(t)} \dt b\cdot \nabla  \phi\, dy+ \int_{\Omega_+(t)} \dt \hb\cdot\nabla \hat \phi\, dy \nonumber
 \\&\quad=   \int_{\Omega_-(t)}    E\cdot \curl  b\, dy .
\end{align}
This yields again \eqref{en_iden1} and hence \eqref{en_idenphy}.

%%%%%%%%%%%%%%%%%%%%%%%%%%%%%%%%%%%%%%%%%%%%%%%%%%%%%%
\subsection{Related works}
%%%%%%%%%%%%%%%%%%%%%%%%%%%%%%%%%%%%%%%%%%%%%%%%%%%%%%

Free boundary problems in fluid mechanics have attracted huge attention in the mathematical community.  Unlike the Euler equations (see Nalimov \cite{N},  Wu \cite{W1,W2}) and the Navier-Stokes equations (see Solonnikov \cite{S1}, Beale \cite{B1}), the free boundary problems for MHD have been studied only more recently. The free boundary problems for MHD arise typically when a plasmas is surrounded by the vacuum and when two plasmas are separated by a free interface, which are known as the plasma-vacuum interface problem and the plasma-plasma interface problem, respectively. In this paper we focus us on the incompressible MHD.

For  the ideal (inviscid and non-resistive) MHD,  the magnetic field is required to be tangential on the free interface, which is transformed to be the constraint on the initial magnetic field, and in this case the dynamic boundary condition on the free interface is then reduced to the balance of the total pressure of the hydrodynamic part and the magnetic part. For the incompressible plasma-vacuum interface problem, under the non-colinearity condition of the magnetic fields on the free interface, which yields a regularizing effect for the free interface,  Morando, Trakhinin and Trebeschi \cite{MTT} showed the well-posedness of the linearized problem and Sun, Wang and Zhang \cite{SWZ2} proved the local well-posedness of the nonlinear problem; under the Taylor condition of the total pressure on the free interface,  when the magnetic field in the vacuum is trivial Hao and Luo \cite{HL} established an a priori estimate and Gu and the first author \cite{GW} proved the local well-posedness, while when the  magnetic field in the vacuum is nontrivial the local well-posedness is still unknown.  For the incompressible plasma-plasma interface problem, Sun, Wang and Zhang \cite{SWZ1} proved the local well-posedness under the Syrovatskij stability condition, and previously, Coulombel, Morando, Secchi and Trebeschi \cite{CMST} showed an a priori estimate under a stronger condition. For the incompressible Euler equations, it is known that either the Taylor condition of the pressure or the effect of surface tension on the free surface is required for the local well-posedness of the one-phase problem (see Christodoulou and Lindblad  \cite{CL}, Lindblad  \cite{Li}, Coutand and Shkoller \cite{CS1}, Shatah and Zeng \cite{SZ} and Zhang and Zhang \cite{ZZ}), and the effect of surface tension is necessary for the local well-posedness of the two-phase problem (see Cheng, Coutand and Shkoller \cite{CCS1} and Shatah and Zeng \cite{SZ1,SZ2}); otherwise, one has the ill-posedness of the problem (see Ebin  \cite{E1,E} and Caflisch and Orellana \cite{CO}). Thus the works \cite{CMST, MTT, SWZ1,SWZ2} show the stabilizing effect  of the magnetic field on the local well-posedness for inviscid fluids.

It is natural to consider the question whether there is a global well-posedness for free boundary problems or not. The recent works, Castro,  C\'ordoba, Fefferman, Gancedo and G\'omez-Serrano \cite{CCFGG1,CCFGG2}, Fefferman, Ionescu and Lie \cite{FIL} and Coutand \cite{Cou},  imply the development of singularities in finite time of free boundary problems for some large initial data. For the irrotational incompressible Euler equations in the horizontally nonperiodic setting,  certain dispersive effects can be used to establish the global well-posedness for the small initial data; we refer to  Wu \cite{W3,W4}, Germain, Masmoudi and Shatah \cite{GMS1,GMS2}, Ionescu and Pusateri \cite{IP1,IP2}, Alazard and Delort \cite{AD} and  Deng,  Ionescu,  Pausader and  Pusateri \cite{DIPP}. We refer to Beale \cite{B2}, Solonnikov \cite{S2},  Hataya \cite{H}, Guo and Tice \cite{GT1,GT2,GT3} and the first author, Tice and Kim \cite{WTK} for the global well-posedness of the incompressible Navier-Stokes equations. Despite these, it is still not clear whether the  free-boundary incompressible Euler equations for the general small initial data admits a global unique solution or not.  It is then interesting and important to study the effect of the magnetic field on the global well-posedness for inviscid fluids.

Note that the global well-posedness of free boundary problems for the ideal MHD is unknown, and  it is reasonable to expect the global well-posedness of the viscous and resistive MHD. 
We may refer to Padula and Solonnikov \cite{PS} and Solonnikov \cite{So2,So} for the local well-posedness of  the incompressible plasma-vacuum interface problem and Solonnikov and Frolova \cite{SF} and Solonnikov \cite{So} for the global well-posedness around the zero magnetic field. In \cite{W}, the first author proved the global well-posedness of the incompressible viscous and non-resistive plasma-plasma interface problem around a traversal uniform magnetic field. These results \cite{SF,So,W} of the global well-posedness rely heavily on the dissipation and regularizing effects of the viscosity for the velocity field. In this paper, we will prove the global well-posedness of free interface problems for the incompressible inviscid and resistive MHD around a traversal uniform magnetic field. It seems  more subtle and difficult to prove the global well-posedness for the inviscid and resistive MHD since the flow is transported by the velocity. Indeed, even the local well-posedness theory is much involved and technically difficult (see Section \ref{sec_lwp}) and the global existence of classical solutions to the Cauchy problem in 2D is unknown. Our analysis here depends on the finite depth of the fluid in our setting, which allows the use of the Poincar\'e-type inequality.

There are a huge amount of mathematical works for free boundary problems in fluid mechanics, and it is impossible to provide a thorough survey of the literature here. We may refer to the references cited in these works above for more proper survey of the literature.

%%%%%%%%%%%%%%%%%%%%%%%%%%%%%%%%%%%%%%%%%%%%%%%%%%%%%%
\section{Main results}\label{sec_2}
%%%%%%%%%%%%%%%%%%%%%%%%%%%%%%%%%%%%%%%%%%%%%%%%%%%%%%

%%%%%%%%%%%%%%%%%%%%%%%%%%%%%%%%%%%%%%%%%%%%%%%%%%%%%%
\subsection{Reformulation in flattening coordinates}
%%%%%%%%%%%%%%%%%%%%%%%%%%%%%%%%%%%%%%%%%%%%%%%%%%%%%%

As usual for free boundary problems in fluid mechanics, we  use a  coordinate transformation in which  the interface stays fixed in time. Set
\beq
\Omega_-:=\mt  \times (-1, 0)\text{ and }\Omega_+:=\mt  \times ( 0,1),
\eeq
and denote by $\Sigma :=\mt  \times \{0\} $ for the interface. The domains can be  flattened by the mapping
\beq\label{diff_def}
\Omega_\pm\ni x\mapsto(x_h, \varphi(t, x ): =x_3+\bar \eta(t, x ))=:\Phi (t,x)=y\in\Omega_\pm(t),
\eeq
where $\bar\eta=\chi\mathcal{P}\eta$ for $\chi=\chi(x_3)$ a smooth function  in $\mathbb{R}$ that satisfies  $\chi(0)=1$ and $\chi(\pm1)=0$  and $\mathcal{P}\eta$ the specialized harmonic extension of $\eta$ onto $\mathbb{R}^3$ with $\mathcal{P}$ defined by \eqref{poisson_def}. 

If $\eta$  is sufficiently small and regular,  then the mapping $\Phi$ is a diffeomorphism. This allows one to transform the problem in  $\Omega_\pm(t)$ to one  in $\Omega_\pm$ for each $t\ge 0$.
Set
\beq\label{pav_def}
\partial^\varphi_{i} =  \partial_{i}  - { \partial_{i} \bar\eta }\D_{3}, \quad i=t, \, 1, \, 2, \quad \D_{3}=  \frac{1}{  \partial_{3} \varphi} \partial_{3}.
\eeq
For the jump conditions on $\Sigma$, define the interfacial jump as
\begin{equation}
\jump{b} := \hb \vert_{\Sigma} - b \vert_{\Sigma}.
\end{equation}
etc.
 Then the problem \eqref{MHD}  is equivalent to the following problem in new coordinates:
\beq\label{MHDv}
\begin{cases}
\D_t   {u}  +    {u}\cdot \Dn  {u} +\Dn   p =   \curlv   b\times (\bar B+b)  & \text{in } \Omega_-
\\ \diva    u=0  &\text{in }\Omega_-
\\ \D_t  {b}        =  \curlv   E ,\quad   E=  u\times (\bar B+b) -\kappa\curlv  b & \text{in } \Omega_-
\\ \diva  b = 0  &\text{in }\Omega_-
\\\curlv\hb=0,\quad\diva \hb=0 &\text{in }\Omega_+ 
\\ \partial_t \eta  = u\cdot\N & \text{on }\Sigma
\\  p=-\sigma H,\quad \jump{b}=0  & \text{on } \Sigma
\\ u_3 =0,\quad b_3 =0,\quad E\times e_3 =0 & \text{on } \Sigma_{-}
\\  \hb\times e_3=0   & \text{on } \Sigma_{+}
\\ (u,b,\eta)\mid_{t=0}= (u_0,b_{0},\eta_0).
\end{cases}
\eeq
Here $ \(\nabla^\varphi  \)_i =  \D_{i},  \ i=1,2,3,\    \diva      = \nabla^\varphi\cdot$ and $\curlv    = \nabla^\varphi\times   $. Also the notation $ \Delta^\varphi=  \diva \nabla^\varphi$ will be used later.

The energy-dissipation law \eqref{en_idenphy} in the new coordinates reads as
\begin{align}\label{en_iden}
& \hal\dtt \(  \int_{\Omega_-} \(|u|^2+|b|^2 \) d\V+\int_{\Omega_+}   |\hb|^2   d\V +    \int_{\mt } 2\sigma\(\sqrt{1+|\nabla_h \eta|^2}-1\)     \)  \nonumber
\\&\quad+ \kappa  \int_{\Omega_-} |\curlv b|^2\,d\V =0.
\end{align}
Here   $d\V:=\partial_{3}\varphi \, dx$ is the volume element induced by the change of variables  \eqref{diff_def}.

%%%%%%%%%%%%%%%%%%%%%%%%%%%%%%%%%%%%%%%%%%%%%%%%%%%%%%
\subsection{Statement of the results}\label{sec21}
%%%%%%%%%%%%%%%%%%%%%%%%%%%%%%%%%%%%%%%%%%%%%%%%%%%%%%

One of the aims of this paper is to show the global well-posedness of the problem \eqref{MHDv} around the trivial equilibrium state when $\bar B_3\neq 0$.

Before stating the main results, we first mention the issue of compatibility conditions for the initial data $(u_0,b_0, \eta_0)$ since the problem \eqref{MHDv} is considered in a domain with boundary. We will work in a high-regularity context, essentially with regularity up to $2N$ temporal derivatives. This requires one to use $(u_0,b_0,\eta_0)$ to construct the initial data $\partial_t^j \eta(0)$ for $j=1,\dotsc,2N+1$, $\partial_t^j u(0)$ and $\partial_t^j b(0)$ for $j=1,\dotsc,2N$,  $\partial_t^j  p(0)$ for $j =0,\dotsc, 2N-1$ and $\partial_t^j  \hb(0)$ for $j =0,\dotsc, 2N$. These data need to satisfy various conditions, which in turn require $(u_0,b_0,\eta_0)$ to satisfy the necessary compatibility conditions that are natural for the local well-posedness of \eqref{MHDv} in the functional framework below. The construction of these data is technically quite involved and will be given in details in Section \ref{sec71}, and these compatibility conditions will be described explicitly as the $2N$-th order compatibility conditions \eqref{compatibility}.
We will also show in Section \ref{sec71} that the set of the initial data $(u_0,b_0,\eta_0)$ satisfying the  compatibility conditions \eqref{compatibility}  is not empty.  For the global well-posedenss of \eqref{MHDv}, it is assumed further that 
\beq\label{zero_0v}
\int_{\mt } \eta_0 =0.
\eeq
For sufficiently regular solutions,  the condition \eqref{zero_0v} persists in time, $i.e.$,
\beq\label{zero_0vt}
\int_{\mt } \eta  =0.
\eeq
Indeed, one has
\beq
\dtt \int_{\mt }\eta=  \int_{\mt }\dt \eta=  \int_{ \mt}u\cdot \n =\int_{\Omega_-} \diva u\, d\V=0.
\eeq

Let $H^k(\Omega_\pm)$, $k\ge 0$ and $H^s(\mt )$, $s \in \Rn{}$ be the usual Sobolev spaces with norms denoted by $\norm{\cdot}_m  $ and $ \abs{\cdot}_s$, respectively. For an integer $N\ge 4$, we define the high-order energy as
\begin{align}\label{high_en}
\se{2N}:=&\sum_{j=0}^{2N}\norm{  \dt^j{u}}_{2N-j}^2 +\sum_{j=0}^{2N-1}\norm{ \dt^j {b}}_{2N-j+1 }^2+\norm{ \dt^{2N} {b}}_{0}^2
 +\sum_{j=0}^{2N-1}\norm{ \dt^j {\hb}}_{2N-j +1}^2+\norm{ \dt^{2N} {\hb}}_{0}^2 \nonumber
\\&+\sum_{j=0}^{2N-1}\norm{  \dt^j p}_{2N-j}^2+ \sum_{j=0}^{{2N}-1}\abs{  \dt^j\eta}_{{2N}-j+3/2}^2+\abs{  \dt^{2N}\eta}_{1}^2+\abs{  \dt^{{2N}+1}\eta}_{-1/2}^2.
\end{align}

\begin{rem}
Note that in the definition \eqref{high_en}, the magnetic field $b$ does not have the usual parabolic regularity, which results from the coupling with the free-surface incompressible Euler equations. Due to the regularizing effect of the magnetic diffusion,  $b$  and $\hb$ enjoy one order of regularity higher than the velocity $u$, up to $2N-1$ temporal derivatives.
\end{rem}

The key part in proving the global well-posedness of \eqref{MHDv} is to show that the high-order energy $\se{2N}(t)$ for $N\ge8$ is bounded for all  $t\ge 0$. To this end,  we need to derive a sufficiently fast time-decay rate of certain lower-order Sobolev norms of the solution, which will follow from the dissipation estimates. For  $n=N+4,\dots,2N$, define a set of dissipations as
\begin{align} \label{low_diss}
\fd{n}:=&  \sum_{j=0}^{n-1} \norm{\dt^j u}_{n-j-1}^2+  \sum_{j=0}^{n-2}\norm{\dt^j b}_{n-j }^2 +\sum_{j=0}^{n}\norm{\dt^j b}_{1,n-j}^2   +\sum_{j=0}^{n }\norm{ \dt^j {\hb}}_{n-j+1}^2  \nonumber
\\&+\sum_{j=0}^{n-2}\norm{  \dt^{j} p }_{n-j-1}^2 +\sum_{j=0}^{n-2}\abs{  \dt^{j} \eta}_{n-j+1/2}^2+\abs{\dt^{n-1}\eta}_{1}^2 +\abs{\dt^{n}\eta}_{0}^2.
\end{align}
Here the anisotropic Sobolev norm $\norm{\cdot}_{m,\ell}$ is defined as
\beq\label{ani_sob_norm}
\norm{f}_{m,\ell}:= \sum_{\al\in\mathbb{N}^2, |\al|\le \ell}\norm{\pa^\al f}_m.
\eeq
Note that the dissipation $\fd{2N}$ can not control the energy $\se{2N}$. Furthermore,  it is the  following energy which is involved in the derivation of the dissipation estimates of $\fd{n}$:
\begin{align} \label{low_en}
\fe{n}:=&\norm{   u}_{n-1}^2+\norm{   u}_{0,n}^2 +\sum_{j=1}^{n}\norm{\dt^j u}_{n-j}^2 +\norm{ b}_{n}^2+ \sum_{j=1}^{n-1}\norm{ \dt^j {b}}_{n-j+1 }^2+ \norm{ \dt^n{b}}_{0}^2  \nonumber
\\&	+\norm{\hb}_{n}^2  +\sum_{j=1}^{n-1}\norm{\dt^j \hb}_{n-j+1 }^2  +\norm{ \dt^n {\hb}}_{0}^2  +\sum_{j=0}^{n-1}\norm{  \dt^j p}_{n-j}^2\nonumber
\\&	+ \sum_{j=0}^{{n}-1}\abs{  \dt^j\eta}_{{n}-j+3/2}^2+\abs{  \dt^{n}\eta}_{1}^2+\abs{  \dt^{{n}+1}\eta}_{-1/2}^2.
\end{align}

Now the main results of this paper are stated as follows.
\begin{thm}\label{main_thm}
Assume that $\kappa>0,\ \bar B_3\neq 0$ and $\sigma>0$ and let $N\ge 8$ be an integer. Assume that $u_0\in H^{2N}(\Omega_-)$, $b_0\in H^{2N+1}(\Omega_-)$  and $\eta_0\in H^{2N+3/2}(\Sigma )$ are given such that $\se{2N} (0)<+\infty$ and that  the $2N$-th order compatibility conditions \eqref{compatibility}  as well as the zero average condition \eqref{zero_0v} are satisfied.
There exists a universal constant $\varepsilon_0>0$ such that if $\se{2N} (0) \le \varepsilon_0$, then there exists a global unique solution $(u,p,\eta,b,\hb)$ to  \eqref{MHDv}. Moreover,   for all $t\ge 0$,
\beq\label{thm_en1}
 \se{2N} (t) + \int_0^t \fd{2N} (s)\,ds \ls \se{2N} (0)
\eeq
and
\beq\label{thm_en2}
\sum_{j=0}^{N-6} (1+t)^{N-5-j}\fe{N+4+j}(t) + \sum_{j=0}^{N-6}\int_0^t(1+s)^{N-5-j}\fd{N+4+j}(s)\,ds\ls \se{2N} (0)  .
\eeq
\end{thm}

\begin{rem}
Theorem \ref{main_thm} implies in particular that $\sqrt{\fe{N+4} (t)} \ls (1+t)^{- {(N-5)}/{2}} $, which is integrable in time for $N\ge 8 $. Since $N$ may be taken to be arbitrarily large, this decay result can be regarded as an ``almost exponential" decay rate. Since $\eta$ is such that the mapping $\Phi(t,\cdot)$, defined in \eqref{diff_def}, is a diffeomorphism for each $t\ge 0$, one may change coordinates to $y\in\Omega_\pm(t)$ to produce a global-in-time, decaying solution to \eqref{MHD}.
\end{rem}

\begin{rem}
In contrast with the works \cite{CMST, MTT, SWZ1,SWZ2} which show the stabilizing effect  of the tangential magnetic field on the local well-posedness of the ideal MHD, the global well-posedness in Theorem \ref{main_thm} relies crucially on that the  magnetic field is traversal (see also \cite{W}). Indeed, the analysis here cannot be applied to the case when $\bar B$ is horizontal; for example, if $\bar B=e_1:=$ $(1,0,0)$, $B=\bar B$, $ \hbb=\bar B$ and $u_1=0$, then the problem under consideration  reduces to the free-boundary incompressible Euler equations in 2D for which the global well-posedness is unknown.
\end{rem}

\begin{rem}
It turns out that to solve \eqref{MHDv} with the desired regularity of $b$ (and $\hb$) in  \eqref{high_en}, even local in time, one needs $\eta\in H^{2N+1/2}$ due to the magnetic diffusion term $\curlv\curlv b$. For the case without surface tension, $i.e.$, $\sigma=0$, it seems that only $H^{2N}$ regularity for $\eta$ is available.  Hence $\sigma>0$ is necessary here even for the local well-posedness. This is different from the viscous and resistive problem \cite{So}, where the viscosity has a regularizing effect of $1/2$ order for $\eta$ and so $\sigma>0$ is not necessary. 
\end{rem}

\begin{rem}
The local well-posedness of \eqref{MHDv}, which is of independent interest, will be proved in Section \ref{sec_lwp}. It should be noted that the ideas in the works \cite{ GW,MTT, SWZ2} for the local well-posedness of the ideal MHD do not work in our case. Indeed, even though the magnetic diffusion here has a regularized effect for the magnetic field, one of the main difficulties in constructing solutions to \eqref{MHDv} lies in solving the magnetic system due to the nonlocal boundary condition for the magnetic field. For the viscous and resistive MHD, Padula and Solonnikov \cite{PS}  solved the magnetic system in the framework of full parabolic regularity theory,  which unfortunately can not be applied to the inviscid problem here due to the less regularity of the velocity. Our way is to solve the magnetic system in the framework  of energy method, which is naturally consistent with the Euler equations, and the solution is constructed as the limit of  approximate solutions to an appropriate regularization as described in the next subsection.
\end{rem}

\begin{rem}
The main ideas and strategies for the the plasma-vacuum interface problem can be modified to study the plasma-plasma interface problem to obtain its global well-posedness. This will be given in Section \ref{sec_pp}. To our best knowledge, the results in this paper are the first ones on the global well-posedness of free boundary problems for the incompressible inviscid rotational fluids around the equilibrium. This is due to the strong coupling between the fluid and the diffusive magnetic field. 
\end{rem}

%%%%%%%%%%%%%%%%%%%%%%%%%%%%%%%%%%%%%%%%%%%%%%%%%%%%%%
\subsection{Strategy of the proof}
%%%%%%%%%%%%%%%%%%%%%%%%%%%%%%%%%%%%%%%%%%%%%%%%%%%%%%

Theorem \ref{main_thm} will be proved in Section \ref{sec_gwp} by combining  the local well-posedness of  \eqref{MHDv}, Theorem \ref{main_thm2}, and the global-in-time a priori estimates, Theorem \ref{apriorith}, with a standard continuity argument.

Note that for free boundary problems in fluid mechanics, even with the necessary a priori estimates of the solutions ready, it is often still highly nontrivial to construct such solutions, especially for inviscid fluids. So we consider first the construction of local solutions to  \eqref{MHDv}. As the Lorentz force is of  lower order regularity compared to the magnetic diffusion term, one may decompose \eqref{MHDv}  into the hydrodynamic part in $\Omega_-$ and the magnetic part in $\Omega$ and then construct the solutions by an iteration. For the force $ F =\curl^{\tilde \varphi}  b\times (\bar B+  b)$ with $\tilde \eta$ and $  b$ given, the hydrodynamic part (cf. \eqref{MHDvk1}) is the free-surface incompressible Euler equations with surface tension, which can be solved in spirit of Coutand and Shkoller \cite{CS1}.  It then remains to handle  the magnetic part  for $ G =  u\times (\bar B+\tilde b)$ with $  u$,  $\tilde b$ and $\eta$ given (cf. \eqref{MHDvk01}). The magnetic part \eqref{MHDvk01} was solved in Padula and Solonnikov \cite{PS} in a different setting  by treating \eqref{MHDvk01} with  $\eta$ small as a perturbation of the ``flat interface" problem, the one obtained by setting $\eta=0$ in \eqref{MHDvk01}. The flat interface problem can be  solved  by employing the Galerkin method, see  Ladyzhenskaya and Solonnikov \cite{LS1,LS}. The solution to  \eqref{MHDvk01} is then produced from solutions to the flat interface problem and an iteration argument in \cite{PS} by employing the full parabolic regularity, which works in the  anisotropic space-time Sobolev spaces (cf. \eqref{parabolicspace}). Such spaces were used extensively in studying the nonhomogeneous boundary value problem for parabolic systems, see  Lions and Magenes \cite{LM}.  A subtle point of using such spaces  in \cite{PS}  is that it allows for the control of the resulting forcing terms when one adjusts the inhomogeneous  terms, $e.g.$ $\curl \hb\neq 0$ in $\Omega_+$ here, see also \cite{B1,S2} for the study of the free-surface incompressible Navier-Stokes equations. However,  such full parabolic regularity of solving \eqref{MHDvk01} is not consistent in the iteration scheme of constructing solutions to \eqref{MHDv} as the {\it hyperbolic} Euler equations could not provide such higher regularity of $u$ and $\eta$. To get around this, we will solve \eqref{MHDvk01} in the functional framework  of using the energy structure of the problem (cf. \eqref{en_iden}). The main strategy is to first construct approximate solutions to an appropriately regularized problem by following the arguments of \cite{PS} and then derive the uniform estimates  independent of the smoothing parameter of the solutions in the framework of energy method  to pass to the limit to produce a solution to \eqref{MHDvk01}. More precisely, we will consider first the following regularized problem:  
\beq\label{MHDv22} 
\begin{cases}
\De_t  {\be}   +\kappa\curlve\curlve  \be     =  \curlve \( G^\eep -\Psi^\eep\)& \text{in } \Omega_-
\\ \divae  \be = 0  &\text{in }\Omega_-
\\\curlve\hbe=0,\quad\divae \hbe=0 &\text{in }\Omega_+
\\ \jump{\be}=0  & \text{on } \Sigma
\\  b^\eep_3 =0,\quad\kappa \curlve  \be \times e_3 = { G^\eep }\times e_3 & \text{on } \Sigma_{-}
\\  \hbe\times e_3=0  & \text{on } \Sigma_{+}
\\ \be\mid_{t=0}= b_{0}^\eep.
\end{cases}
\eeq
Here $\varphi^\eep=\varphi(\eta^\eep)$ as in \eqref{diff_def}, $\eta^\epp$ and $G^\eep$ are the smooth approximations of $\eta$ and $G$, respectively, where $\epp>0$ is the smoothing parameter.  It should be pointed out that the introduction of both the so-called corrector $\Psi^\eep$ in \eqref{MHDv22} and a sequence of  correctors $\phi_j^\epp$ in \eqref{MHDvk21100} is crucial here, which allows one to construct the smooth approximation  $b_{0}^\eep$ of $b_0$ satisfying the corresponding compatibility conditions for \eqref{MHDv22}. Such idea could be applicable for general parabolic problems when one needs to smooth out the initial data. We then follow  the arguments of \cite{PS} to solve  \eqref{MHDv22}, in a higher order regularity context.  To derive the uniform estimates independent of $\epp>0$ of  the solutions to \eqref{MHDv22}, with the desired regularity in our functional framework, we will make an important use of the corresponding regularized electric field in vacuum, $\hee^\eep$, which solves 
\beq\label{MHDv33}
\begin{cases}
\curlve\hee^\eep=\De_t\hb^\eep,\quad\divae \hee^\eep=0 &\text{in }\Omega_+
\\
\hee^\eep\times \ne=\(-\kappa\curlve  b^\epp+G^\epp-\Psi^\epp\)\times \ne  &\text{on }\Sigma
\\
 \hee_3^\eep=0 &\text{on }\Sigma_{+}.
\end{cases}
\eeq
The solvability of \eqref{MHDv33} is classical, see  Cheng and Shkoller \cite{CS11} and the references therein for instance. Indeed, the first, second, fourth and fifth equations in \eqref{MHDv22} provide the necessary conditions for solving \eqref{MHDv33}. The solution to  the original problem \eqref{MHDvk01} is then obtained as the limit of the sequence of solutions to \eqref{MHDv22} as $\epsilon\rightarrow 0$ after deriving the uniform estimates for the approximate solutions on a time interval independent of $\epsilon$, by  a slight variant of the derivation of the estimates of \eqref{MHDv} as outlined below. Note that in our way of solving  the  magnetic system \eqref{MHDvk01}, the electric field in vacuum $\he$ could be viewed as an auxiliary variable, rather than the secondary variable  as in \cite{PS,So,SWZ1,SWZ2}. We remark  that one could  also use the virtual magnetic field $\mathfrak{b}$ as an auxiliary variable instead of $\he$; yet we choose  to work with $\he$  here as it is a physical variable. Finally, one can then construct solutions to   \eqref{MHDv} by the method of successive approximations, based on the solvability of  the  problems \eqref{MHDvk1} and \eqref{MHDvk01}.

Now we turn to the derivation of the a priori estimates for the solutions to \eqref{MHDv}.  Our derivation  involves the  electric field  in vacuum $\hee$ which solves \eqref{MHDv3e}, and the estimates of  $\hee$ are provided in Section \ref{sec_ee}. The basic ingredient in our analysis is to use the energy-dissipation structure \eqref{en_iden}. Since the higher order energy functionals are needed to control the nonlinear terms, one applies the temporal and horizontal spatial derivatives $\pa^\al$ for $\al\in \mathbb{N}^{1+2}$ with $\abs{\al}\le 2N$ to \eqref{MHDv} and  \eqref{MHDv3e} to derive the tangential energy evolution
\begin{align}\label{intro_1}
& \hal\dtt \left(\int_{\Omega_-}\( |\pa^\al u|^2+|\pa^\al b|^2 \) d\V+\int_{\Omega_+}|\pa^\al \hb|^2 d\V+\int_{\Sigma } \sigma  { |    \nabla  \pa^\alpha \eta|^2  }  \)  +   \kappa  \int_{\Omega_-} |\curlv \pa^\al b|^2\, d\V\nonumber
\\&\quad =\int_{\Omega_-}   \pa^\al p  F^{2,\al} \, d\V-\int_{\Sigma }   \sigma \pa^\alpha H   F^{5,\al}+ \int_{\Omega_+}   \pa^\al \hat E \cdot   \hat F^{4,\al}  d\V  +{\sum}_{\mathcal{R}}.
\end{align}
Here  the nonlinear terms $F^{2,\al}$, $F^{5,\al}$ and $\hat F^{4,\al}$ are defined by \eqref{f2a}, \eqref{f5a} and \eqref{f4aa}, respectively,  and ${\sum}_{\mathcal{R}}$ denotes terms involving other nonlinearities, which, after some delicate arguments, can be controlled well in the sense that a term in ${\sum}_{\mathcal{R}}$ is either bounded by $\sqrt{\fe{N+4}} \(\se{ 2N}+\fd{2N} \)$, or its time integration is bounded by  $(\se{ 2N})^{3/2}$,   as $\se{2N}$ is small. When $\al_0\le 2N-1$,  the first  three terms  in the right hand side of \eqref{intro_1} can be shown to be also of ${\sum}_{\\jump{\pa_3 q}{R}}$. When $\alpha_0=2N$, the difficulty is that  $\dt^{2N}p$, $\dt^{2N}H$  and $\dt^{2N}\he$ seem to be out of control. Integrating by parts in time shows that  the third term is of ${\sum}_{\mathcal{R}}$, while integrating by parts in both time and space in an appropriate order and then employing a crucial cancelation between $\dt^{2N}p$ and $\sigma\dt^{2N}H$ on  $\Sigma$ by using the dynamic boundary condition as what we have done in \cite{WX}, one can show that  the first two terms are of ${\sum}_{\mathcal{R}}$. These lead to the following tangential energy evolution estimates:
\beq\label{intro_3}
\seb{2N}(t)+ \int_0^t  \sdb{2N} \ls   \se{2N}(0)+ (\se{ 2N}(t))^{3/2} +\int_0^t\sqrt{\fe{N+4}} \(\se{ 2N}+\fd{2N} \),
\eeq
where $\seb{n}$ and $\sdb{n}$ represent the tangential energy and dissipation functionals to be defined by \eqref{bar_en_def} and \eqref{bar_dn_def} respectively.

Note that, as already seen from the estimate \eqref{intro_3},  to close the estimates, since the energy can not be dominated by the dissipation, one needs to show that $\sqrt{\fe{N+4}(t)}$ is integrable in time. The key here is to show that $\fe{N+4}(t)$ decays sufficiently fast in time. To this end, employing an elaborate argument, we are able to derive a related set of tangential energy evolution estimates different from \eqref{intro_3}:
\beq\label{intro_4}
\frac{d}{dt}\left(\bar{\mathcal{E}}_{n}+\mathcal{B}_{n}\right)+ \bar{\mathcal{D}}_{n} \ls   \sqrt{\se{2N}   } \fd{ n} ,\  n=N+4,\dots,2N-2,
\eeq
with $\mathcal{B}_{n}$ satisfying $\abs{\mathcal{B}_{n}}\ls \sqrt{\se{2N}   } { \fe{n} }.$  

With the control of the tangential energy  estimates, one proceeds to derive the full energy estimates by exploiting further the structures of \eqref{MHDv}. First, one 
may improve the $\curl$-estimates of $b$ in the tangential dissipation $\bar{\mathcal{D}}_{n}$ to be the  $H^1$-estimates of $b$ and derive the  desired  dissipation estimates of $\hb$ 
by applying  the Hodge-type estimates. 
Note that such dissipation estimates  control only  $b$ and $\hb$. To get the tangential dissipation estimates for $u$, it is crucial for us to derive the dissipation estimates of the following $\nabb-$terms from the control of $\sdb{n}$ for $\kappa>0$,
\beq\label{Bidiss}
  \sum_{j=0}^{n-1} \norm{\bar B \cdot \nabla\dt^j u_3}_{0,n-j-1}^2+ \sum_{j=0}^{n-1}\norm{\bar B \cdot \nabla\dt^j \(\kappa\p_3b_h+\bar B_3 u_h\)}_{0,n-j-1}^2.
\eeq
Here we have used the notation that $v_h=(v_1,v_2)$ for any vector $v$.
This follows by projecting the magnetic equations onto the vertical and horizontal components, respectively. Then one uses the Poincar\'e-type inequalities related to $\nabb$ for $\bar B_3\neq0$  together with the boundary conditions on $\Sigma_-$ to derive the tangential dissipation estimates of $u$.

Now the heart of the analysis is to derive the estimates involving the normal derivatives of $u$ and $b$.  The natural way of estimating the normal derivatives of $u$, as for the incompressible Euler equations, is to consider the equations for the vorticity $\curlv{u} $:
\beq\label{intro_5}
\D_t \curlv{u}  +    {u}\cdot \Dn  \curlv {u}  =   \bar{B}   \cdot \Dn  \curlv {b}  + \cdots.
\eeq
Here $+\cdots$ means plus some nonlinear terms.
One of the key observations here is the treatment of the linear term in the right hand side of \eqref{intro_5}: by using the third and second equations in \eqref{MHDv}, one finds
\begin{align}\label{intro_51}
&\bar{B}   \cdot \Dn  (\curlv  b)_1 \equiv\bar{B}_h   \cdot \Dn_h  (\curlv  b)_1+\bar{B}_3     \D_1  (\curlv  b)_3+\bar{B}_3      (\curlv\curlv  b)_2 
\nonumber
\\&\quad=\bar{B}_h   \cdot \Dn_h  (\curlv  b)_1  +\bar{B}_3     \D_1  (\curlv  b)_3+\frac{\bar B_3}{\kappa} ( -\D_t b_2+\bar{B}   \cdot \Dn u_2+\cdots).
\end{align}
On the other hand, one has
\begin{align}\label{intro_52}
\bar{B}   \cdot \Dn u_2
\equiv \bar{B}_h   \cdot \Dn_h u_2-\bar B_3 (\curlv  u)_1+\bar B_3 \pav_2 u_3.
\end{align}
Carrying out the similar computations for $\bar{B}   \cdot \Dn  (\curlv  b)_2$, one then deduces from \eqref{intro_5}   the following equation of $(\curlv{u})_h$: for $i=1,2,$
\begin{align} \label{intro_55}
&\D_t (\curlv{u})_i  +    {u}\cdot \Dn  (\curlv{u})_i+\frac{\bar{B}_3^2}{\kappa}(\curlv{u})_i
\\&\quad=\bar{B}_h   \cdot \nabla_h  (\curl  b)_i  +\bar{B}_3     \p_i  (\curl  b)_3+(-1)^{i+1}\frac{\bar B_3}{\kappa} ( -\p_t b_{3-i}+\bar{B}_h   \cdot \nabla_h u_{3-i} +\bar B_3 \p_{3-i} u_3 ) +\cdots.\nonumber
\end{align}
One thus sees again the key roles of the positivity of the magnetic diffusion coefficient $\kappa>0$ and the non-vanishing of $\bar B_3\neq0$; they induce the damping term in \eqref{intro_55}, which provides the mechanism for the global-in-time estimates of $(\curlv{u})_h$. Note that for the estimates of $u$ in the energy $\se{2N}$,  one can estimate the linear $\nabla_h u $  terms in the right hand of  \eqref{intro_55} by the control of $\seb{2N}$; for the estimates of $u$ in the dissipation $\fd{n}$, one has to estimate these terms by  using instead the tangential dissipation estimates of $u$ derived from the $\nabb$-estimate \eqref{Bidiss}. Making use of these estimates, the transport-damping structure of $(\curlv{u})_h$ in \eqref{intro_55} and the Hodge-type estimates, one can derive the desired estimates of $u$ in a recursive way in terms of the number of normal derivatives of $u$;  the desired estimates of $b$ and $\hb$ can be derived along by employing  the elliptic estimates.

With these  estimates above, one may then derive the desired estimates for $p$ and $\eta$ by using directly the equations  \eqref{MHDv}.  The conclusion is that one can thus improve \eqref{intro_3}  and \eqref{intro_4}  to be, respectively, since $\se{2N}$ is small,
\begin{align}\label{intro_6}
&\se{2N}(t)+ \int_0^t  \fd{2N}    \ls  \se{2N}(0)+ \int_0^t\sqrt{\fe{N+4} }  \se{ 2N}  
\end{align}
and
\beq\label{intro_7}
\frac{d}{dt}\fe{n}	+ \fd{n} \le 0,\quad n=N+4,\dots,2N-2.
\eeq 	
	
Note that if $\fe{N+4}(t)$ decays at a sufficiently fast rate, then the estimate \eqref{intro_6}  can lead   to   \eqref{thm_en1}. This will be achieved by using \eqref{intro_7}.  One does not have that $  \fe{n}\ls \fd{n}$, which rules out the exponential decay; also, $\fd{n}$ can not control $\fe{n}$ with respect to not only the spatial regularity but also the temporal regularity, which prevents one from using the spatial regularity Sobolev interpolation argument to bound $\fe{n} \ls \se{2N}^{1-\theta} \fd{n}^{\theta},  0<\theta<1$ so as to derive the algebraic decay.  Our key ingredient to get around this here is to observe that $\fe{\ell}\le \fd{\ell+1}$ and   employ  a time weighted inductive argument to \eqref{intro_7}. The conclusion is  
\begin{align}
&\sum_{j=0}^{N-6} (1+t)^{N-5-j}\fe{N+4+j}(t) + \sum_{j=0}^{N-6}\int_0^t(1+s)^{N-5-j}\fd{N+4+j}(s)\,ds\nonumber
\\&\quad\ls \se{2N}(0)+ \int_0^t\fd{2N-1}(s)\,ds .
\end{align}
This together with \eqref{thm_en1} yields \eqref{thm_en2} and hence a decay of $\fe{N+4}$ with the rate $(1+t)^{-N+5}$. Consequently, this scheme of the a priori estimates can be closed by requiring $N\ge 8$.

%%%%%%%%%%%%%%%%%%%%%%%%%%%%%%%%%%%%%%%%%%%%%%%%%%%%%%
\subsection{Notation}
%%%%%%%%%%%%%%%%%%%%%%%%%%%%%%%%%%%%%%%%%%%%%%%%%%%%%%

We now list the conventions for notations.  $C>0$ denotes a generic constant independent of the data and time,  but may depend on the  parameters of the problem, $\kappa,\sigma, \bar B$ and  $N$, which is referred  to as ``universal''.  Such constants are allowed to change from line to line. To indicate some constants in some places so that they can be referred to later, they will be denoted  in particular by $C_1,C_2$, etc.   $A_1 \lesssim A_2$ means that $A_1 \le C A_2$, and 
$A_1\ls   A_2+A_3$ means that $A_1\le  A_2+C A_3$, for a universal constant $C>0$.  To avoid the constants in various time differential inequalities, we use the following two conventions:
\begin{align} \label{ccon}
&\dt A_1+A_2\ls A_3 \text{ means }\dt \widetilde A_1+A_2\ls A_3 \text{ for }A_1\ls \widetilde A_1\ls A_1,
\\ 
&\dt \(A_1+A_2\)+A_3\ls A_4 \text{ means }\dt \(C_1A_1+C_2A_2\)+A_3\ls A_4 \text{ for constants }C_1,C_2>0.
\end{align}

Also, $\mathbb{N} = \{ 0,1,2,\dotsc\}$ denotes for the collection of non-negative integers. When using space-time differential multi-indices, we write $\mathbb{N}^{1+d} = \{ \alpha = (\alpha_0,\alpha_1,\dotsc,\alpha_d) \}$ to emphasize that the $0-$index term is related to temporal derivatives. For just spatial derivatives, we write $\mathbb{N}^d$.  For $\alpha \in \mathbb{N}^{1+d}$,  $\pa^\alpha = \dt^{\alpha_0}  \pa_1^{\alpha_1}\cdots \pa_d^{\alpha_d}.$  We define the standard commutator
\beq  \label{cconm1}
\[\pa^\al, f\]g=\pa^\al (fg)-f\pa^\al g
\eeq
and the symmetric commutator
\beq  \label{cconm2}
\[\pa^\al, f,g \]=\pa^\al (fg)-f\pa^\al g-\pa^\al f g.
\eeq

We omit the differential elements $dx$ and $dx_h$ of the integrals over $\Omega_\pm$ and $\Sigma$   and also sometimes the  differential elements $ds$ of the time integrals.

%%%%%%%%%%%%%%%%%%%%%%%%%%%%%%%%%%%%%%%%%%%%%%%%%%%%%%%
\subsection{Organization of the paper}
%%%%%%%%%%%%%%%%%%%%%%%%%%%%%%%%%%%%%%%%%%%%%%%%%%%%%%%
%

The rest of the paper is organized as follows. Section \ref{sec_3} concerns several Hodge-type elliptic problems to be used often later. Section \ref{sec_4} contains some preliminary results for the a priori estimates. We derive first the tangential energy evolution estimates in Section  \ref{sec_5}, then the full energy and dissipation estimates in Section \ref{sec_6}, and finally the global a priori estimates in Section \ref{sec_priori}. Section \ref{sec_lwp} contains the proof of  the local well-posedness. The global well-posedness is proved in Section \ref{sec_gwp}. Section \ref{sec_pp}  considers  the plasma-plasma interface problem.  Some analytic tools are collected in Appendix \ref{section_app}.

%%%%%%%%%%%%%%%%%%%%%%%%%%%%%%%%%%%%%%%%%%%%%%%%%%%%%%
 \section{Hodge-type Elliptic systems}\label{sec_3}
 %%%%%%%%%%%%%%%%%%%%%%%%%%%%%%%%%%%%%%%%%%%%%%%%%%%%%%

In this section we will consider the solvability and regularity of several Hodge-type elliptic problems to be used later. 

First, we consider the following one-phase Hodge-type elliptic problem:
  \beq \label{elpp1}
\begin{cases}
\curlv v =f^1,\quad\diva v=f^2 &\text{in }\tilde \Omega
\\v\times \n =f^3  &\text{on }\tilde\Sigma_1
\\ v\cdot \n=f^4 &\text{on }\tilde\Sigma_{2}.
\end{cases}
\eeq 
Here  $\tilde \Omega$ is either $\Omega_-$ or $\Omega+$ or $\Omega$, $\tilde\Sigma_1$,  $ \tilde\Sigma_2$ are the two boundaries of $\tilde \Omega$ and we have extended $\n$ to be $(-\nabla_h\bar\eta,1)$, which reads as  $(-\nabla_h \eta,1)$ on $\Sigma $ and  $e_3$  on $\Sigma_\pm$, respectively.
 \begin{prop}\label{propel1}
Assume  $\eta\in H^{k+1/2}$ for an integer $k> 3/2$ with $\p_3\varphi>0$.  Let $r=1,\dots,k$ and $f^1\in H^{r-1}(\tilde\Omega )$, $f^2\in H^{r-1}(\tilde\Omega )$, $f^3\in H^{r-1/2}(\tilde\Sigma_1)$ and $f^4\in H^{r-1/2}(\tilde\Sigma_2)$   be given such that 
 \beq\label{nececon}
\diva f^1=0\text{ in }\tilde\Omega ,\ f^3\cdot \n =0 \text{ and }f^1\cdot \n =\diverge_h f^3_h\text{ on }\tilde\Sigma_1.
\eeq
There   exists a unique solution $v$ to \eqref{elpp1} satisfying
\beq\label{v_ell_th01}
\norm{v}_r\ls_\eta \norm{f^1}_{r-1}+\norm{f^2}_{r-1} +\abs{f^3}_{r-1/2}+\abs{f^4}_{r-1/2}.
\eeq
Hereafter $\ls_\eta$  stands for $\le C_\eta$ for a constant $C_\eta$ depending on $\abs{\eta}_{k+1/2}$ and $1/\norm{ \p_3\varphi}_{L^\infty}$.
\end{prop}
\begin{proof}
This can be proved similarly as for Theorem 1.1 in Cheng and Shkoller \cite{CS11}, which established the solvability and regularity for a Hodge-type elliptic system  in a Sobolev-class bounded domain, where the curl and divergence inside the domain and either the normal component  or tangential components on the boundary of a vector field are prescribed.  In order to understand the compatibility conditions in \eqref{nececon} and for reader's convenience, we still sketch the construction of the solution to \eqref{elpp1}, for the case that $\tilde\Omega=\Omega_+$, $\tilde \Sigma_1=\Sigma$ and $\tilde \Sigma_2=\Sigma_+$ for instance.

First, by the first condition in \eqref{nececon}, according to the assertion (1) of Theorem 1.1 in  \cite{CS11}, one can define $\tilde v$ as the solution to
\beq \label{elpp10}
\begin{cases}
\curlv \tilde v =f^1,\quad\diva \tilde v=f^2 &\text{in }  \Omega_+
\\ \tilde v\cdot\n =\int_{\Sigma_+} f^4- \int_{\Omega_+} f^2 \p_3\varphi  &\text{on } \Sigma
\\ \tilde v_3=f^4 &\text{on } \Sigma_+.
\end{cases}
\eeq 
Then the solution to \eqref{elpp1} can be constructed as   $v=\tilde v+\bar v$ with $\bar v$ the solution to
\beq \label{elpp100}
\begin{cases}
\curlv \bar v =0,\quad\diva \bar v=0 &\text{in }  \Omega_+
\\ \bar v\times \n=\bar f^3:= f^3-\tilde v\times\n &\text{on } \Sigma
\\ \bar v_3=0 &\text{on } \Sigma_+.
\end{cases}
\eeq 
Next, note that the last two conditions in \eqref{nececon} and the first equation in \eqref{elpp10} yield
 \beq\label{nececon'}
\bar f^3\cdot \n = f^3\cdot \n =0 \text{ and } \diverge_h \bar f^3_h=\diverge_h   f^3_h-\curlv \tilde v\cdot \n=\diverge_h   f^3_h-f^1\cdot \n=0\text{ on }\Sigma.
\eeq
This implies that there exists a $\psi=\psi(x_1,x_2)$ such that 
\beq
\bar f^3=(\p_2\psi, -\p_1\psi, \p_1\eta\p_2\psi-\p_2\eta\p_1\psi).
\eeq
Then  the solution to \eqref{elpp100} can be constructed as $\bar v=\Dn \phi$ with $\phi$ the solution to 
\beq \label{elpp10000}
\begin{cases}
\Delta^\varphi  \phi =0 &\text{in }  \Omega_+
\\  \phi  =\psi &\text{on } \Sigma
\\ \p_3 \phi\cdot =0 &\text{on }\Sigma_+.
\end{cases}
\eeq 
The construction of the solution to \eqref{elpp1} is thus completed.
\end{proof}

Next, we consider the following two-phase Hodge-type elliptic problem:
\beq\label{elpp2}
\begin{cases}
    \curlv   v=  f^1, \quad \diva v= f^2 & \text{in } \Omega_-
 \\\curlv \hv= \hat f^1 ,\quad\diva \hv=  \hat f^2 &\text{in }\Omega_+
\\\jump{v}=0  & \text{on } \Sigma
\\  v_3=0 & \text{on } \Sigma_{-}
\\   \hv\times e_3=0  & \text{on } \Sigma_{+}.
\end{cases}
\eeq

 \begin{prop} \label{propel2}
Assume  $\eta\in H^{k+1/2}$ for an integer $k> 3/2$ with $\p_3\varphi>0$.  Let $r=1,\dots,k+1$ and $f^1,f^2\in H^{r-1}(\Omega_-)$, $\hf^1,\hf^2\in H^{r-1}(\Omega_+)$  be given such that 
 \beq\label{nececon0}
\diva f^1=0\text{ in }\Omega_-, \ \diva \hf^1=0\text{ in }\Omega_+,\ \jump{f^1}\cdot \n=0\text{ on }\Sigma  \text{ and }\hf^1\cdot e_3=0\text{ on }\Sigma_+.
\eeq
There exists   a unique solution $(v,\hv)$ to \eqref{elpp2} satisfying
\beq\label{v_ell_th02}
\norm{v}_r+\norm{\hv}_r\ls_\eta  \norm{f^1}_{r-1}+\norm{f^2}_{r-1} +\norm{\hf^1}_{r-1}+\norm{\hf^2}_{r-1}.
\eeq
\end{prop}
\begin{proof}
This can be proved similarly as for Theorem 2 in  Padula and Solonnikov \cite{PS}, which concerns the bounded domain case  when $\Omega_-$ is surrounded by $\Omega_+$ and the normal component of $\hv$ was prescribed on $\p\Omega$.  Here we sketch an alternative proof based on  Proposition \ref{propel1}.

To this end, one may define $F^1$ in $\Omega$ with $F^1=f^1$ in $\Omega_-$ and $F^1=\hf^1$ in $\Omega_+$ and $F^2 $ in $\Omega$ with $F^2=f^2$ in $\Omega_-$ and $F^2=\hf^2$ in $\Omega_+$. By the first three conditions in \eqref{nececon0}, one has that $
\diva F^1$ exists and vanishes in $\Omega$, and the last condition implies $F^1\cdot e_3=0$ on $\Sigma_+$. Hence, by Proposition \ref{propel1}, there exists a unique solution $V\in H^1(\Omega)$ to 
  \beq \label{elpp1op}
\begin{cases}
\curlv V =F^1,\quad\diva V=F^2 &\text{in }  \Omega
\\  V_3=0 & \text{on } \Sigma_{-}
\\   V\times e_3=0  & \text{on } \Sigma_{+}.
\end{cases}
\eeq 
Taking $v=V$ in $\Omega_-$ and $\hv=V$ in $\Omega_+$ yields the conclusion for $r=1$.

Now we derive the higher regularity estimates \eqref{v_ell_th02} for $2\le r\le k+1$ by an induction argument. Suppose that $\ell\in[1,r-1]$ and \eqref{v_ell_th02} holds for $\ell$. One then applies $\p^\al$ for $\al\in\mathbb{N}^2$ with $\abs{\al}\le \ell$ to \eqref{elpp2} to find that
\beq\label{elpp231}
\begin{cases}
    \curlv  \p^\al v=f^{1,\al} , \quad \diva \p^\al v= f^{2,\al}  & \text{in } \Omega_-
 \\    \curlv  \p^\al \hv=  \hf^{1,\al}, \quad \diva \p^\al \hv= \hf^{2,\al} &\text{in }\Omega_+
\\\jump{ \p^\al v}=0  & \text{on } \Sigma
\\   \p^\al v_3=0 & \text{on } \Sigma_{-}
\\    \p^\al\hv\times e_3=0  & \text{on } \Sigma_{+},
\end{cases}
\eeq
where
\beq 
\begin{split}
f^{1,\al} := \p^\al f^1-\[\p^\al,\curlv\] v, \quad  f^{2,\al} := \p^\al f^2-\[\p^\al,\diva\] v,
 \\    \hf^{1,\al} := \p^\al f^1-\[\p^\al,\curlv\] v, \quad  \hf^{2,\al} := \p^\al \hf^2-\[\p^\al,\diva\] v.
\end{split}
\eeq
It is routine to check that
 \beq\label{nececon0a}
\diva f^{1,\al}=0\text{ in }\Omega_-, \ \diva \hf^{1,\al}=0\text{ in }\Omega_+,\ \jump{f^{1,\al}}\cdot \n=0\text{ on }\Sigma  \text{ and }\hf^{1,\al}\cdot e_3=0\text{ on }\Sigma_+.
\eeq
The conclusion for $r=1$ then yields that 
\beq 
\norm{v}_{1,\ell}+\norm{\hv}_{1,\ell}\ls_\eta  \norm{f^1}_{\ell}+\norm{f^2}_{\ell} +\norm{\hf^1}_{\ell}+\norm{\hf^2}_{\ell}+C_\eta\(\norm{v}_{\ell}+\norm{\hv}_{\ell}\).
\eeq
This together with the trace theory and   the induction assumption  imply
\beq 
\abs{v}_{ \ell+1/2}+\abs{\hv}_{ \ell+1/2}\ls \norm{v}_{1,\ell}+\norm{\hv}_{1,\ell}\ls_\eta   \norm{f^1}_{\ell}+\norm{f^2}_{\ell} +\norm{\hf^1}_{\ell}+\norm{\hf^2}_{\ell}.
\eeq
Hence, by Proposition \ref{propel1}, one has
\beq 
\norm{v}_{ \ell+1}+\norm{\hv}_{ \ell+1}\ls_\eta  \norm{f^1}_{\ell}+\norm{f^2}_{\ell} +\abs{v}_{ \ell+1/2}+\abs{\hv}_{ \ell+1/2}
\ls_\eta \norm{f^1}_{\ell}+\norm{f^2}_{\ell} + \norm{\hf^1}_{\ell}+\norm{\hf^2}_{\ell}.
\eeq
This implies that \eqref{v_ell_th02} holds for $\ell+1$. \eqref{v_ell_th02} is thus proved.
\end{proof}

Finally, we consider the following mixed-phase Hodge-type elliptic problem:
  \beq \label{elpp3}
\begin{cases}
    \curlv\curlv  v=f^1 & \text{in } \Omega_-
\\ \diva v= f^2 & \text{in } \Omega_-
 \\\curlv \hv= \hat f^1 ,\quad\diva \hv=  \hat f^2 &\text{in }\Omega_+
\\\jump{v}=0  & \text{on } \Sigma
\\  v_3=0,\quad   { \curlv v }\times e_3=f^3 & \text{on } \Sigma_{-}
\\   \hv\times e_3=0  & \text{on } \Sigma_{+}.
\end{cases}
\eeq

 \begin{prop}\label{propel3}
Assume  $\eta\in H^{k+1/2}$ for an integer $k> 3/2$ with $\p_3\varphi>0$.  Let $r=2,\dots,k+1$ and  $f^1\in H^{r-2}(\Omega_-)$, $f^2\in H^{r-1}(\Omega_-)$, $\hf^1,\hf^2\in H^{r-1}(\Omega_+)$ and $f^3\in H^{r-3/2}(\Sigma_-)$  be given such that 
 \beq\label{nececon1}
\diva f^1=0\text{ in }\Omega_-, \ f^3\cdot  e_3 =0 \text{ on }\Sigma_- \text{ and }f^1\cdot e_3=\diverge_h f^3_h\text{ on }\Sigma_-
\eeq
and
 \beq\label{nececon2}
\diva \hf^1=0\text{ in }\Omega_+\text{ and } \hf^1\cdot e_3 =0\text{ on }\Sigma_+.
\eeq
There exists a unique solution $(v,\hv)$ to \eqref{elpp3} satisfying
\beq\label{v_ell_th03}
\norm{v}_r+\norm{\hv}_r\ls_\eta  \norm{f^1}_{r-2}+\norm{f^2}_{r-1} +\norm{\hf^1}_{r-1}+\norm{\hf^2}_{r-1} +\abs{f^3}_{r-3/2}.
\eeq\end{prop}
\begin{proof}
Motivated by the works of Ladyzhenskaya and Solonnikov \cite{LS1,LS}, the solution to \eqref{elpp3} can be constructed as follows. First,  it follows from  \eqref{nececon1} and Proposition \ref{propel1} that
\beq \label{eeelp21} 
\begin{cases}
\curlv w =f^1,\quad\diva w=0 &\text{in }\Omega_+
\\ w\cdot \n = \hf^1\cdot \n  &\text{on }\Sigma
\\ w\times e_3=f^3 &\text{on }\Sigma_{+}
\end{cases}
\eeq 
has a unique solution $w$ which satisfies, by the trace theory,
\beq \label{vvv1}
\norm{w}_{r-1}\ls_\eta  \norm{f^1}_{r-2} +\abs{\hf^1\cdot \n}_{r-3/2}+\abs{f^3}_{r-3/2}\ls_\eta  \norm{f^1}_{r-2} +\norm{\hf^1  }_{r-1}+\abs{f^3}_{r-3/2}.
\eeq
Then by the second and third equations in \eqref{eeelp21}, \eqref{nececon2} and Proposition \ref{propel2}, one can define $(v,\hv)$ as the  solution to
 \beq \label{eeelp2} 
\begin{cases}
    \curlv   v=  w, \quad \diva v= f^2 & \text{in } \Omega_-
 \\\curlv \hv= \hf^1 ,\quad\diva \hv=  \hf^2 &\text{in }\Omega_+
\\\jump{v}=0  & \text{on } \Sigma
\\  v_3=0 & \text{on } \Sigma_{-}
\\   \hv\times e_3=0  & \text{on } \Sigma_{+},
\end{cases}
\eeq
which satisfies
\beq \label{vvv2}
\norm{v}_r+\norm{\hv}_r\ls_\eta  \norm{w}_{r-1}+\norm{f^2}_{r-1} +\norm{\hf^1}_{r-1}+\norm{\hf^2}_{r-1}.
\eeq
It is easy to check that $(v,\hv)$ solves \eqref{elpp3}, and \eqref{v_ell_th03} follows from \eqref{vvv1} and \eqref{vvv2}.
\end{proof}

%%%%%%%%%%%%%%%%%%%%%%%%%%%%%%%%%%%%%%%%%%%%%%%%%%%%%%
\section{Preliminaries for the a priori estimates}\label{sec_4}
%%%%%%%%%%%%%%%%%%%%%%%%%%%%%%%%%%%%%%%%%%%%%%%%%%%%%%

In this section we give some preliminary results to be used in the derivation of the a priori estimates for solutions to \eqref{MHDv}. It will be assumed throughout Sections \ref{sec_4}--\ref{sec_6} that the solution  is given on the interval $[0,T]$ and obeys the a priori assumption
\beq\label{apriori_1}
\se{2N}(t)\le \delta,\quad\forall t\in[0,T]
\eeq
for $N\ge 4$ and a sufficiently small constant $\delta>0.$ This implies in particular that
\beq\label{apriori_2}
\hal\le \pa_3\varphi(t,x)\le \frac{3}{2},\quad\forall (t,x)\in[0,T]\times \bar\Omega.
\eeq
We remark that \eqref{apriori_1} and \eqref{apriori_2} are always used; in particular, the smallness \eqref{apriori_1} is used in many nonlinear estimates so that the various polynomials of $\se{2N}$ are bounded by $C\se{2N}$.

%%%%%%%%%%%%%%%%%%%%%%%%%%%%%%%%%%%%%%%%%%%%%%%%%%%%%%
\subsection{Estimates of $\he$}\label{sec_ee}
%%%%%%%%%%%%%%%%%%%%%%%%%%%%%%%%%%%%%%%%%%%%%%%%%%%%%%

Our derivation of the  estimates of the solutions to \eqref{MHDv} will involve the  electric field in vacuum $\hee$, which solves
\beq\label{MHDv3e}
\begin{cases}
\curlv \hee =\D_t\hb,\quad\diva \hee=0 &\text{in }\Omega_+
\\\hee\times \n=E\times \n  &\text{on }\Sigma
\\ \hee_3=0 &\text{on }\Sigma_{+}.
\end{cases}
\eeq
\begin{rem}\label{ssrem}
Note that by the seventh, fifth,  tenth and third equations in \eqref{MHDv}, one has
\beq
\diva\D_t\hb=0\text{ in }\Omega_+\text{ and }\D_t\hb\cdot \n=\D_tb\cdot \n=\curlv E\cdot\n\equiv\diverge_h(E\times \n)_h \text{ on }\Sigma.
\eeq
Thus Proposition \ref{propel1} guarantees the existence of a unique solution $\he$ to \eqref{MHDv3e}.
\end{rem}

Now we estimate $\he.$ For $N\ge 4$, define
\beq\label{high_ene}
\se{2N}(\he):= \sum_{j=0}^{2N-1} \ns{ \dt^j    \he }_{2N-j} 
\eeq
and that for $n=N+4,\dots,2N$, 
\beq\label{low_ene}
\fe{n}(\he):= \ns{   \he }_{n-1} + \sum_{j=1}^{n-1} \ns{ \dt^j    \he }_{n-j}
\eeq
and
\beq \label{low_disse}
\fd{n}(\he):= \sum_{j=0}^{n-2} \ns{ \dt^j    \he }_{n-j-1}.  
\eeq 

The  equations \eqref{MHDv3e} can be rewritten as
\beq\label{eeequ}
\begin{cases}
 \curl \he=P^1,\quad \diverge \he=P^2&\text{in }\Omega_+
 \\ \hee\times e_3 =P^{3} & \text{on } \Sigma
 \\   \he_3=0  & \text{on } \Sigma_{+},
\end{cases}
\eeq
where
\begin{align}
& P^1=\D_t \hb+\nabla\bar\eta\times\pav_3  \he,
\\& P^2=\nabla\bar\eta\cdot\pav_3 \he,
\\& P^3=  E\times\n+\he\times (e_3-\n) .
\end{align}
Here one has used the fact that $\pav_i-\pa_i=-\pa_i\bar\eta\pav_3$ for $i=t,1,2,3$ due to \eqref{pav_def}.

The   terms $P^i$ are estimated as follows.
\begin{lem}\label{le_G_2Ne}
It holds that
\begin{align}\label{p_G_e_02e1}
&\sum_{j=0}^{2N-1}\ns{    \dt^j  P^1}_{2N-j-1}+\sum_{j=0}^{2N-1}\ns{    \dt^j  P^2}_{2N-j-1}+ \sum_{j=0}^{2N-1}\as{    \dt^j  P^3}_{2N-j-1/2} \nonumber
\\&\quad \ls \se{ 2N}+ \se{ 2N}\se{ 2N}(\he)
\end{align} 
and that for $n=N+4,\dots,2N$,
\begin{align}\label{p_G_e_02e2}
&\ns{      P^1}_{n-2}+\sum_{j=1}^{n-1}\ns{    \dt^j  P^1}_{n-j-1}+  \ns{    P^2}_{n-2}+\sum_{j=1}^{n-1}\ns{    \dt^j  P^2}_{n-j-1}+  \as{    P^3}_{n-3/2}+ \sum_{j=1}^{n-1}\as{    \dt^j  P^3}_{n-j-1/2} \nonumber
\\&\quad \ls \fe{n}+ \se{ 2N}\fe{ n}(\he)
\end{align} 
and
\begin{align}\label{p_G_e_02e3}
& \sum_{j=0}^{n-2}\ns{    \dt^j  P^1}_{n-j-2}+\sum_{j=0}^{n-2}\ns{    \dt^j  P^2}_{n-j-2}+ \sum_{j=0}^{n-2}\as{    \dt^j  P^3}_{n-j-3/2} \nonumber
\\&\quad \ls \fd{n}+ \se{ 2N}\fd{ n}(\he).
\end{align} 
\end{lem}
\begin{proof}
These can be checked similarly as for Lemmas  \ref{le_F_2N} and \ref{p_F_eN2} below. 
\end{proof}

The estimates of $\he$ are given as follows.
\begin{prop}\label{eprop}
It holds that
\begin{align}\label{e_es1}
\se{2N}(\he)\ls    \se{2N}
\end{align}
and that for $n=N+4,\dots,2N$,
\begin{align}\label{e_es3}
\fe{n}(\he) \ls   \fe{n} 
\end{align}
and
\begin{align}\label{e_es4}
\fd{n}(\he) \ls   \fd{n}.
\end{align}
\end{prop}
\begin{proof}
For $j=0,\dots,2N-1$, it follows from the Hodge-type estimates \eqref{v_ell_th01} of Proposition \ref{propel1}  (setting $\eta=0$) with $r=2N-j\ge 1$  that, by \eqref{eeequ} and \eqref{p_G_e_02e1},
\begin{align}\label{eeee1}
 \ns{ \dt^j    \he }_{2N-j}&\ls    \norm{ \dt^j  \curl {\he} }_{2N-j-1}^2+  \ns{   \dt^j  \diverge {\he}}_{2N-j-1 } +  \abs{ \dt^j    {\he}\times e_3 }_{2N-j-1/2}^2\nonumber\\&\ls   
  \ns{   \dt^j    P^1}_{2N-j-1}+  \ns{   \dt^j P^2 }_{2N-j-1} + \as{   \dt^j   P^3  }_{2N-j-1/2}  \nonumber
 \\&\ls  \se{ 2N}+ \se{ 2N}\se{ 2N}(\he).
\end{align}
Summing \eqref{eeee1} over $j=0,\dots,2N-1$  yields the estimate \eqref{e_es1} since $\se{2N}$ is small. 
The estimates \eqref{e_es3} and  \eqref{e_es4} follow similarly by using instead \eqref{p_G_e_02e2} and \eqref{p_G_e_02e3}, respectively.
\end{proof}
 
In light of Proposition \ref{eprop}, in the nonlinear estimates to be carried out  in Sections \ref{sec_5}--\ref{sec_6}, we can simply use $\se{2N}$ to bound terms which are controlled by $\se{2N}(\he)$, etc.

%%%%%%%%%%%%%%%%%%%%%%%%%%%%%%%%%%%%%%%%%%%%%%%%%%%%%%
\subsection{Geometric perturbed formulation}
%%%%%%%%%%%%%%%%%%%%%%%%%%%%%%%%%%%%%%%%%%%%%%%%%%%%%%

In order to  use the energy-dissipation structure \eqref{en_iden} to  derive the tangential energy evolution estimates for solutions to \eqref{MHDv}, as usual for free boundary problems in fluid mechanics, it is natural to utilize the geometric structure given in \eqref{MHDv}. For this, one applies the temporal and horizontal spatial derivatives $\pa^\al$ for $\al\in \mathbb{N}^{1+2}$  with $|\al|\ge 1$   to \eqref{MHDv} and \eqref{MHDv3e} to find that
\beq\label{MHDva} 
\begin{cases}
\D_t \pa^\al {u}  +    {u}\cdot \Dn \pa^\al {u} +\Dn \pa^\al p =   \curlv \pa^\al b\times (\bar B+b) +F^{1,\al}& \text{in } \Omega_-
\\ \diva  \pa^\al u=F^{2,\al}  &\text{in }\Omega_-
\\ \D_t \pa^\al {b}        =  \curlv \pa^\al E +F^{3,\al},\quad \pa^\al E= \pa^\al u\times (\bar B+b)-\kappa\curlv \pa^\al b+F^{4,\al} & \text{in } \Omega_-
\\ \D_t \pa^\al {\hb}        =  \curlv \pa^\al \hat E +\hat F^{3,\al} ,\quad  \curlv \pa^\al \hb= \hat F^{4,\al} & \text{in } \Omega_+
\\ \partial_t \pa^\al \eta  =\pa^\al u\cdot\N +F^{5,\al}&\text{on }\Sigma
\\ \pa^\al p =  -  \sigma \pa^\alpha H  ,\quad \pa^\alpha H    =      \diverge_h  \(\! \dis\frac{ \nabla_h \pa^\alpha \eta}{\sqrt{1+|\nabla_h \eta|^2}}- \frac{\nabla_h \eta \cdot \nabla_h \pa^\alpha \eta}{\sqrt{1+|\nabla_h \eta|^2}^3}\nabla_h \eta 
\! \) +F^{6,\al}   & \text{on } \Sigma
\\ \jump{\pa^\al b} =  0,\quad \jump{\pa^\al E}\times \n = F^{7,\al}   & \text{on } \Sigma
\\ \pa^\al u_3 =0,\quad 	\pa^\al E\times e_3=0  & \text{on } \Sigma_{-}
\\  \pa^\al \hb\times e_3=0  & \text{on } \Sigma_{+} ,
\end{cases}
\eeq
where, recalling the commutator notations \eqref{cconm1} and \eqref{cconm2},
\begin{align}
&F^{1,\al}=-\[\pa^\al,  (\bar B+b) \times  \curlv   \]  b -\[\pa^\al,\D_t  +    {u}\cdot \Dn  \]{u} -\[\pa^\al, \Dn  \]p ,\label{f1a}
\\&F^{2,\al}=-\[\pa^\al, \diva  \] {u} , \label{f2a}
\\&F^{3,\al}=\[\pa^\al,     \curlv   \] E-\[\pa^\al,\D_t  \]{b}    ,\label{f3a}
\\& F^{4,\al}=-\[\pa^\al,  b \] \times u-\kappa \[\pa^\al, \curlv\]  b,   \label{f4a}
\\&  \hat F^{3,\al}=\[\pa^\al,     \curlv   \] \hat E -\[\pa^\al,\D_t  \]{\hb}  ,\label{f3aa}
\\& \hat F^{4,\al}=-  \[\pa^\al, \curlv\]  \hb ,   \label{f4aa}
\\&F^{5,\al}= \[\pa^\al, \n  \] \cdot {u},\label{f5a}
\\&  F^{6,\al}=\diverge_h \(-\[\pa^{\alpha-\alpha'},\frac{\nabla_h \eta }{\sqrt{1+|\nabla_h \eta|^2}^3}\]\cdot \nabla_h \pa^{\alpha'} \eta
 \nabla_h \eta +\[\pa^\alpha, \frac{1}{\sqrt{1+|\nabla_h \eta|^2}}, \nabla_h \eta \]\),\label{f6a}
\\&F^{7,\al}= \[\pa^\al, \n  \] \times \jump{E} .\label{f7a}
\end{align}
Here $\al'$ in \eqref{f6a} is  any $\al'<\al$ with $|\alpha'|=1$. Note that the rest of equations in \eqref{MHDv} and \eqref{MHDv3e} that are not considered in \eqref{MHDva} will be not needed in the tangential energy evolution estimates.

These nonlinear terms $F^{i,\al}$ with $\abs{\al}\le 2N$ are estimated as follows.
\begin{lem}\label{le_F_2N}
It holds that for $\abs{\al}\le 2N$,
\begin{align}\label{p_F_e_01}
&\ns{F^{1,\al} }_{0}+ \ns{F^{2,\al}}_{0}+ \ns{ F^{3,\al}}_{0}+ \ns{ \hat {F}^{3,\al}}_{0} + \ns{ F^{4,\al}}_{0}+ \ns{ \hat {F}^{4,\al}}_{0}\nonumber
\\&\quad+ \abs{F^{5,\al}}_{0}^2 + \abs{F^{6,\al}}_{1/2}^2
+ \abs{F^{7,\al}}_{0}^2    \ls \fe{N+4} \se{2N} ,
\end{align}
and for $\abs{\al}\le 2N,\al_0\le 2N-1$,
\beq\label{p_F_e_02}
\abs{F^{5,\al}}_{1/2}^2    \ls \fe{N+4} \se{2N}
\eeq
and 
\begin{align}\label{p_F_e_01'}
 \ns{\dt \hat {F}^{4,(2N,0)}}_{-1}   \ls \fe{N+4}\( \se{2N} +\fd{2N}\),
\end{align}
where $\norm{\cdot}_{-1}$ denotes the norm of $(H^1(\Omega))^\ast$.
\end{lem}
\begin{proof}
To estimate $F^{1,\al}$, it follows from the Leibniz rule and the Sobolev embeddings that, by the definition \eqref{pav_def} of $\D_{t}$ and Lemma  \ref{sobolev},
\begin{align}\label{ftes1}
&\ns{\[\pa^\al, \D_t \] {u}}_0\equiv \ns{\[\pa^\al,   \frac{\dt\bar\eta}{\p_3\varphi}  \]\p_3 {u}}_0 \ls \sum_{0\neq \al'\le \al} \ns{ \pa^{\al'}\( \frac{\dt\bar\eta}{\p_3\varphi} \)\pa^{\al-\al'} \p_3 u  }_0\nonumber
\\& \quad\ls \sum_{  \al'\le \al \atop 0<\abs{\al'}\le N} \ns{ \pa^{\al'}\(\frac{\dt\bar\eta}{\p_3\varphi}\)}_{2} \ns{\pa^{\al-\al'} \p_3 u  }_0
+ \sum_{  \al'\le \al \atop  \abs{\al'}> N} \ns{\pa^{\al-\al'} \p_3 u  }_{2}\ns{ \pa^{\al'}\(\frac{\dt\bar\eta}{\p_3\varphi}\)}_{0}.
\end{align}
To bound the $H^0$ norms in the right hand side of \eqref{ftes1},   one may check directly that the terms of the highest order derivatives involved are $\pa^{\al-\al'}\pa_3 u$ for $\al'\in \mathbb{N}^{1+2}$ with $\abs{\al'}=1$ and $\pa^{\al+\beta'}\bar\eta$ for $\beta'\in \mathbb{N}^{1+3}$ with $\abs{\beta'}=1$. Noting that the term $\as{\dt^{2N+1}\eta}_{-1/2} $ is included in $\se{2N}$  so that  when  $\pa^{\al+\beta'}=\dt^{2N+1}$, one gets by  Lemma \ref{p_poisson} that
\beq
\ns{\dt^{2N+1}\bar\eta}_0\ls \ns{\dt^{2N+1}\mathcal{P}\eta}_{0}\ls \as{\dt^{2N+1}\eta}_{-1/2}\le \se{2N}.
\eeq
Then the $H^0$ norms in the right hand side of \eqref{ftes1} are bounded by $\se{2N}$, due to  Lemmas \ref{p_poisson} and \ref{sobolev}. On the other hand, by Lemmas \ref{p_poisson} and \ref{sobolev} again along with the definition \eqref{low_en} of $\fe{n}$, one notes that the extra $4$ derivatives in $\fe{N+4} $ have been chosen so that those $H^2$ norms in the right hand side of \eqref{ftes1} can be bounded by $\fe{N+4} $.  Hence $\ns{\[\pa^\al, \D_t \] {u}}_0\ls \fe{N+4} \se{2N}$. Estimating the other terms in $F^{1,\al}$ in the same way, one may conclude that
\beq
\ns{F^{1,\al} }_{0}\ls \fe{N+4} \se{2N}.
\eeq
Similarly, one has that
\beq
\ns{F^{2,\al} }_{0}+\ns{\hat F^{3,\al} }_{0}+\ns{  F^{4,\al} }_{0}+\ns{\hat F^{4,\al} }_{0}\ls \fe{N+4} \se{2N}.
\eeq

To estimate $ F^{3,\al} $, one may argue again as $F^{1,\al}$ to derive the desired estimates except the term $\kappa \[\pa^\al,\curlv\]  \curlv b$, which involves new types of terms of the highest order derivatives: $\pa^{\al-\al'}\pa_3 \nabla \bar\eta$ and $\pa^{\al-\al'}\pa_3 \nabla b$  for $\al'\in \mathbb{N}^{1+2}$ with $\abs{\al'}=1$. By  Lemma \ref{p_poisson}, one has
\beq
\ns{\pa^{\al-\al'}\pa_3 \nabla  \bar\eta}_0 \ls \ns{\pa^{\al-\al'} \mathcal{P}\eta}_2\ls \as{\pa^{\al-\al'} \eta}_{3/2}\le \se{2N}
\eeq
and since $\al_0-\al_0'\le 2N-1$,
\beq
\ns{\pa^{\al-\al'}\pa_3 \nabla b}_0\le \ns{\pa^{\al-\al'} b}_2 \le \se{2N}.
\eeq
Hence, one can get
\beq
\ns{F^{3,\al} }_{0}\ls \fe{N+4}  \se{2N}.
\eeq

To estimate $F^{5,\al}$,  since the terms of the highest order derivatives are $\pa^{\al-\al'}u_h$ for $\al'\in \mathbb{N}^{1+2}$ with $\abs{\al'}=1$ and $\pa^{\al }\nabla_h \eta$, one then separates the cases $\al_0=2N$ and $\al_0\le 2N-1$. It follows from Lemma \ref{sobolev} and the trace theory  that  
\beq
\as{F^{5,\al}}_0\ls \fe{N+4} \se{2N}\text{ for } \al_0=2N\text{ and }\as{F^{5,\al}}_{1/2}\ls \fe{N+4} \se{2N}\text{ for }\al_0\le 2N-1.
\eeq

To estimate $F^{6,\al}$, since the terms of the highest order derivatives are $\pa^{\al-\al'}\nabla_h^2\eta$ for $\al'\in \mathbb{N}^{1+2}$ with $\abs{\al'}=1$, similarly, one obtains  
\beq
\as{F^{6,\al}}_{1/2}\ls \fe{N+4} \se{2N}.
\eeq

Next, for $F^{7,\al}$, the  new  terms of the highest order derivatives here are $\pa^{\al-\al'}\nabla b$  for $\al'\in \mathbb{N}^{1+2}$ with $\abs{\al'}=1$. It follows from the trace theory that
\beq
\as{\pa^{\al-\al'}  \nabla b}_0\le \ns{\pa^{\al-\al'} b}_2 \le \se{2N}
\text{ and }
\as{\pa^{\al-\al'}  \he}_0\le \ns{\pa^{\al-\al'} \he}_1 \le \se{2N}.
\eeq
Hence, one can get
\beq
\as{F^{7,\al} }_{0}\ls \fe{N+4} \se{2N}.
\eeq

Finally, in $\dt \hat {F}^{4,(2N,0)}$, the new terms of the highest order derivatives are $\dt^{2N+1}\nabla_h \eta$ and $\dt^{2N}\p_3 \hb$. Note that
\beq
\ns{\dt^{2N+1}\nabla \bar\eta}_{-1} \ls \as{\dt^{2N+1}  \eta}_{-1/2}\le \se{2N}\text{ and }\ns{\dt^{2N}\p_3 \hb}_{0}\le \fd{2N}.
\eeq
Hence, by Lemma \ref{sobolev-1}  one can get
\beq 
  \ns{\dt \hat {F}^{4,(2N,0)}}_{-1}   \ls \fe{N+4}\( \se{2N} +\fd{2N}\)
\eeq

Consequently, the estimates \eqref{p_F_e_01}--\eqref{p_F_e_01'} follow.
\end{proof}

We now present some specialized estimates of $F^{i,\al}$ when $\abs{\al}\le 2N-2$.
\begin{lem}\label{p_F_eN2}
For $\abs{\al}\le 2N-2$, it holds that
\begin{align}\label{p_F_e_03}
&\ns{F^{1,\al} }_{1}+ \ns{\dt F^{1,\al} }_{0}+\ns{F^{2,\al}}_{1} +\ns{\dt F^{2,\al}}_{0}+\ns{\dt^2 F^{2,\al}}_{0}+\ns{ F^{3,\al}}_{0}+\ns{ \hat F^{3,\al}}_{0} \nonumber
\\&\quad+\ns{ F^{4,\al}}_{0} +\ns{\hat F^{4,\al}}_{1} +\ns{\dt \hat F^{4,\al}}_{0}+\ns{\dt^2 \hat F^{4,\al}}_{0}+  \abs{F^{5,\al}}_{3/2}^2 +\abs{\dt F^{5,\al}}_{1/2}^2 +\abs{\dt^2F^{5,\al}}_{0}^2
\nonumber
\\&\quad + \abs{F^{6,\al}}_{1}^2 + \abs{\dt F^{6,\al}}_{0}^2+ \abs{F^{7,\al}}_{0}^2    \ls  \fd{N+4}\se{2N} 
\end{align}
and
\begin{align}\label{p_F_e_04}
&\ns{  F^{1,\al} }_{0} +\ns{ F^{2,\al}}_{0}+\ns{\dt  F^{2,\al}}_{0}+\ns{ \hat F^{4,\al}}_{0}+\ns{\dt  \hat F^{4,\al}}_{0}\nonumber
\\&\quad+ \abs{  F^{5,\al}}_{0}^2 + \abs{\dt F^{5,\al}}_{0}^2 + \abs{  F^{6,\al}}_{0}^2  \ls  \fe{N+4} \se{2N}.
\end{align}
\end{lem}
\begin{proof}
This follows similarly as for Lemma  \ref{le_F_2N}, by checking the terms of the highest order derivatives involved.
\end{proof}

%%%%%%%%%%%%%%%%%%%%%%%%%%%%%%%%%%%%%%%%%%%%%%%%%%%%%%
\subsection{Linear perturbed formulation}
%%%%%%%%%%%%%%%%%%%%%%%%%%%%%%%%%%%%%%%%%%%%%%%%%%%%%%

 In order to use the linear structure of  \eqref{MHDv},  it is more convenient to write it as a perturbation of the linearized equations:
\beq\label{MHDv_perb}
\begin{cases}
\dt {u}   +\nabla  p =  \curl    b\times  \bar B +G^1  & \text{in } \Omega_-
\\ \diverge  u= G^2  &\text{in }\Omega_-
\\ \dt {b}  +\kappa \curl\curl  b=\curl (u\times \bar B)+G^3 & \text{in } \Omega_-
\\ \diverge  b = G^4 &\text{in }\Omega_-
\\\curl \hb=  \hat{G}^3,\quad\diverge \hb=  \hat{G}^4 &\text{in }\Omega_+
\\ \partial_t \eta  = u_3+G^5 &\text{on }\Sigma
\\ p=   - \sigma \Delta_h \eta +G^6 ,\quad \jump{b}=0 & \text{on } \Sigma
\\ u_3 =0, \quad b_3=0,\quad  {\kappa\curl b }\times e_3=(u\times \bar B)\times e_3+G^7 & \text{on } \Sigma_{-}
\\  \hb\times e_3=0 & \text{on } \Sigma_{+},
\end{cases}
\eeq
where
\begin{align}
&G^1=\dt\bar\eta \pav_3 {u}-  {u}\cdot \Dn  {u}+\nabla\bar\eta\pav_3 p -(\nabla\bar\eta\times\pav_3  {b})\times\bar B+\curlv   b\times b   ,
\\ &G^2=\nabla\bar\eta\cdot\pav_3 u,
\\ &G^3= \dt\bar\eta\pav_3b- \kappa (\curlv\curlv  b -\curl\curl  b)+(\curl-\curl )(u\times \bar B)+\curlv(u\times b),
\\& G^4=\nabla\bar\eta\cdot\pav_3 b,
\\& \hat{G}^3=\nabla\bar\eta\times\pav_3  \hb,
\\& \hat{G}^4=\nabla\bar\eta\cdot\pav_3 \hb,
\\ &G^5=- u_h\cdot\nabla_h\eta ,
\\ &G^6=- \sigma \diverge_h\(\(  ({1+\abs{\nabla_h\eta}^2})^{-1/2}-1\)\nabla_h\eta\) ,
\\& G^7=    \kappa(\curl b -\curlv b) \times e_3  .
\end{align} 

The nonlinear terms $G^i$ are estimated as follows.
\begin{lem}\label{le_G_2N}
It holds that
\begin{align}\label{p_G_e_02}
&\sum_{j=0}^{2N-1}\ns{    \dt^j  G^1}_{2N-j-1}+\sum_{j=0}^{2N-1}\ns{    \dt^j  G^2}_{2N-j-1}+\sum_{j=0}^{2N-1}\ns{    \dt^j  G^3}_{2N-j-1}\nonumber
\\&\quad+\sum_{j=0}^{2N-1}\ns{    \dt^j  G^4}_{2N-j}
+\sum_{j=0}^{2N-1}\ns{    \dt^j  \hat{G}^3}_{2N-j}+\sum_{j=0}^{2N-1}\ns{    \dt^j  \hat{G}^4}_{2N-j}\nonumber
\\&\quad  +\sum_{j=0}^{2N}\as{    \dt^j  G^5}_{2N-j-1/2}+\sum_{j=0}^{2N-1}\as{    \dt^j  G^6}_{2N-j-1/2} 
+ \sum_{j=0}^{2N-1}\as{    \dt^j  G^7}_{2N-j-1/2} \nonumber
\\&\quad \ls \min\{\fe{N+4},\fd{N+4} \}  \se{ 2N} 
\end{align} 
and   
\begin{align}\label{p_G_e_03}
\sum_{j=0}^{2N}\ns{    \dt^j  G^4}_{2N-j}+\sum_{j=0}^{2N}\ns{    \dt^j  \hat{G}^3}_{2N-j} +\sum_{j=0}^{2N}\ns{    \dt^j  \hat{G}^4}_{2N-j}  \ls \fe{N+4}\( \se{ 2N}+\fd{2N}\) .
\end{align}
\end{lem}
\begin{proof}
\eqref{p_G_e_02} can be proved similarly as Lemmas  \ref{le_F_2N} and \ref{p_F_eN2}, and \eqref{p_G_e_03}  follows similarly by noting that   the  new  term of the highest order derivatives is $ \dt^{2N} \p_3 b$ and estimating $\ns{ \dt^{2N} \p_3 b }_0\le \fd{2N}$. 
\end{proof}

The following estimates on the difference between $\D_i $ and $\pa_i$ will be used later.
\begin{lem}
It holds that for  $\abs{\al}\le 2N-1$,  
\beq\label{curl_e_1}
\ns{ \pa^\al \(\D_i u-\pa_i  u\)  }_{0}  \ls  \min\{\fe{N+4},\fd{N+4} \} \se{2N}
\eeq 
and that
for  $\abs{\al}\le n$ with $n=N+4,\dots,2N$,  
\beq\label{curl_e_2}
\ns{ \D_i \pa^\al b-\pa_i  \pa^\al b  }_{0}  \ls  \fe{N+4}\fd{n} .
\eeq 
\end{lem}
\begin{proof}
The proof   follows in the same way as for Lemmas \ref{le_F_2N} and \ref{p_F_eN2}.
\end{proof}

%%%%%%%%%%%%%%%%%%%%%%%%%%%%%%%%%%%%%%%%%%%%%%%%%%%%%%
\subsection{Vorticity equations}
%%%%%%%%%%%%%%%%%%%%%%%%%%%%%%%%%%%%%%%%%%%%%%%%%%%%%%

For the estimates of the normal derivatives of $u$, as for the incompressible Euler equations, a natural way is to estimate first the vorticity $\curlv u$ (to get rid of the pressure term $\nav p$ and avoid the loss of derivatives) and then  to use the Hodge-type  estimates. Applying $\curlv$ to the first equation in \eqref{MHDv} yields that
\beq\label{curlv_eq}
\D_t \curlv{u}  +    {u}\cdot \Dn  \curlv {u} =    \bar{B}   \cdot \Dn  \curlv {b}  + \curlv{u}\cdot \Dn  {u}+\curlv(\curlv{b}\times b ).
\eeq
The difficulty is that there is a linear forcing term $\bar{B}   \cdot \Dn  \curlv {b}$ on the right hand side of \eqref{curlv_eq}, and one cannot use the equations of $\curlv b$ to balance this term as  for the tangential energy evolution estimates, due to the usual difficulties caused by the diffusion term $\kappa\curlv\curlv b$. This is harmful for  the global-in-time uniform estimate of $\curlv {u} $. But, on the other hand, if there were without this term, then it would be difficult to derive the global-in-time uniform estimates  of  the higher order derivatives for $\curlv {u} $ just as for the incompressible Euler equations. The crucial observation here is that there is a new damping structure for the vorticity $\curlv u$.  Indeed, it follows from the second component of the third equation and the second equation in \eqref{MHDv} that
\begin{align}\label{curl_1}
&\bar{B}   \cdot \Dn  (\curlv  b)_1 \equiv\bar{B}_h   \cdot \Dn_h  (\curlv  b)_1+\bar{B}_3      (\curlv\curlv  b)_2 +\bar{B}_3     \D_1  (\curlv  b)_3\nonumber
 \\&\quad=\bar{B}_h   \cdot \Dn_h  (\curlv  b)_1  +\bar{B}_3     \D_1  (\curlv  b)_3+\frac{\bar B_3}{\kappa} ( -\D_t b_2+\bar{B}   \cdot \Dn u_2+(\curlv(u\times b ))_2) .
\end{align}
On the other hand, one can write
\beq\label{curl_3}
\bar{B}   \cdot \Dn u_2
\equiv \bar{B}_h   \cdot \Dn_h u_2-\bar B_3 (\curlv  u)_1+\bar B_3 \pav_2 u_3.
\eeq
Hence, as a consequence of \eqref{curl_1} and \eqref{curl_3}, the first component of \eqref{curlv_eq} can be rewritten as
\begin{align} \label{MHDcurlv1}
 &\D_t (\curlv{u})_1  +    {u}\cdot \Dn  (\curlv {u})_1+\frac{\bar{B} _3^2}{\kappa}(\curlv u)_1 \nonumber
 \\&\quad=  \bar{B}_h   \cdot \Dn_h  (\curlv  b)_1  +\bar{B}_3     \D_1  (\curlv  b)_3+\frac{\bar B_3}{\kappa} ( -\D_t b_2+\bar{B}_h   \cdot \Dn_h u_2 +\bar B_3 \pav_2 u_3 ) \nonumber
  \\&\qquad +\frac{\bar B_3}{\kappa}  (\curlv(u\times b ))_2 + \curlv{u}\cdot \Dn  {u}_1+( \curlv(\curlv{b}\times b ))_1  .
\end{align}
Similarly, one has
\begin{align} \label{MHDcurlv2}
 &\D_t (\curlv{u})_2  +    {u}\cdot \Dn  (\curlv {u})_2+\frac{\bar{B} _3^2}{\kappa}(\curlv u)_2 \nonumber
 \\&\quad=  \bar{B}_h   \cdot \Dn_h  (\curlv  b)_2  +\bar{B}_3     \D_2  (\curlv  b)_3-\frac{\bar B_3}{\kappa} ( -\D_t b_1+\bar{B}_h   \cdot \Dn_h u_1 +\bar B_3 \pav_1 u_3 ) \nonumber
  \\&\qquad -\frac{\bar B_3}{\kappa}  (\curlv(u\times b ))_1 + \curlv{u}\cdot \Dn  {u}_2+( \curlv(\curlv{b}\times b ))_2 .
\end{align}
The equations \eqref{MHDcurlv1} and \eqref{MHDcurlv2} yield a transport-damping evolution structure for $(\curlv{u})_h $, and one then sees the key roles of the positivity of the magnetic diffusion coefficient $\kappa>0$ and the non-vanishing of $\bar B_3\neq0$.

Applying $\pa^\al$ for $\al\in \mathbb{N}^{1+3}$ with $\abs{\al}\ge 1$ to  \eqref{MHDcurlv1} and \eqref{MHDcurlv2} gives that
\beq\label{MHDcurlva}
 \D_t \pa^\al (\curlv{u})_h  +    {u}\cdot \Dn  \pa^\al (\curlv{u})_h+\frac{\bar{B}_3^2}{\kappa} \pa^\al (\curlv{u})_h=    \pa^\al L_h +\Phi^{\al}_h ,
\eeq
where for $i=1,2,$
\beq\label{lheq}
L_i=  \bar{B}_h   \cdot \nabla_h  (\curl  b)_i  +\bar{B}_3     \p_i  (\curl  b)_3+(-1)^{i+1}\frac{\bar B_3}{\kappa} ( -\p_t b_{3-i}+\bar{B}_h   \cdot \nabla_h u_{3-i} +\bar B_3 \p_{3-i} u_3 )
\eeq
and 
\beq
\Phi^{\al}_h=\[\pa^\al, \D_t +    {u}\cdot \Dn\]  (\curlv {u})_h+ \pa^\al\Phi_h
\eeq
with that for $i=1,2,$
\begin{align}
\Phi_i=& -\bar{B}_h   \cdot \nabla_h \bar\eta  \pav_3 (\curl  b)_i  -\bar{B}_3     \p_i   \bar\eta  \pav_3 (\curl  b)_3- \bar{B}_h   \cdot \Dn_h  (\nabla\bar \eta\times \pav_3 b)_i  -\bar{B}_3     \D_i  (\nabla\bar \eta\times \pav_3  b)_3 \nonumber
  \\&+(-1)^{i+1}\frac{\bar B_3}{\kappa} (  \dt\bar \eta   \pav_3  b_{3-i}-\bar{B}_h   \cdot \nabla\bar\eta \pav_3  u_{3-i} -\bar B_3 \p_{3-i} \bar\eta \pav_3 u_3 ) \nonumber
  \\&  +(-1)^{i+1}\frac{\bar B_3}{\kappa}  (\curlv(u\times b ))_{3-i} + \curlv{u}\cdot \Dn  {u}_i+( \curlv(\curlv{b}\times b ))_i .
\end{align}

The nonlinear term $\Phi^{\al}_h$ can be estimated as follows.
\begin{lem}
It holds that for  $\abs{\al}\le 2N-1$,
\beq\label{p_Phi_e_1}
\ns{\Phi^{\al}_h  }_{0}   \ls   \min\{\fe{N+4},\fd{N+4} \}\se{2N}.
\eeq 
\end{lem}
\begin{proof}
The proof follows in the same way as for Lemmas \ref{le_F_2N} and \ref{p_F_eN2}.
\end{proof}

%%%%%%%%%%%%%%%%%%%%%%%%%%%%%%%%%%%%%%%%%%%%%%%%%%%%%%
\section{Tangential energy evolution}\label{sec_5}
%%%%%%%%%%%%%%%%%%%%%%%%%%%%%%%%%%%%%%%%%%%%%%%%%%%%%%

In this section we will derive the tangential energy evolution estimates for solutions to \eqref{MHDv}. For a generic integer $n\ge 3$, we define the tangential energy that involves the temporal and horizontal spatial derivatives by, employing the anisotropic Sobolev norm \eqref{ani_sob_norm},
\beq\label{bar_en_def}
\seb{n}:=\sum_{j=0}^n\norm{\dt^j u}_{0,n-j}^2+\sum_{j=0}^n\norm{\dt^jb}_{0,n-j}^2+\sum_{j=0}^n\norm{\dt^j\hb}_{0,n-j}^2+ \sum_{j=0}^n\abs{\dt^j\eta }_{n-j+1}^2
\eeq
and the corresponding dissipation by
\beq\label{bar_dn_def}
\sdb{n}:= \sum_{j=0}^n\norm{\curl\dt^jb}_{0,n-j}^2.
\eeq

%%%%%%%%%%%%%%%%%%%%%%%%%%%%%%%%%%%%%%%%%%%%%%%%%%%%%%
\subsection{Energy evolution at the $2N$ level}
%%%%%%%%%%%%%%%%%%%%%%%%%%%%%%%%%%%%%%%%%%%%%%%%%%%%%%

We start with the following time-integrated tangential energy evolution estimate at the $2N$ level.
\begin{prop}\label{evolution_2N}
It holds that
\beq \label{en_ev_2N}
\seb{2N}(t)+ \int_0^t  \sdb{2N} (s)\,ds \ls   \se{2N}(0)+(\se{ 2N}(t))^{3/2} +\int_0^t\sqrt{\fe{N+4} } \(\se{ 2N} +\fd{2N} \) .
\eeq
\end{prop}
\begin{proof}
Let $\al\in \mathbb{N}^{1+2}$ such that  $1\le|\al|\le 2N$. Taking the inner product of the first equation in \eqref{MHDva} with $\pa^\al u$ and the third equation with $\pa^\al b$, respectively, integrating by parts over $\Omega_-$ by using the second, eighth and eleventh equations in \eqref{MHDv}, and then adding the resulting equations together, one has
\begin{align}\label{en_iden_al_0}
& \hal\dtt \int_{\Omega_-} \(\abs{\pa^\al u}^2+\abs{\pa^\al b}^2 \) d\V - \hal  \int_{\Sigma}  \dt \eta\abs{\pa^\al b}^2   +\int_{\Omega_-}   \nabla^\varphi \pa^\al p   \cdot \pa^\al u \, d\V
\\&\quad=\int_{\Omega_-}\left(\curlv \pa^\al b\times (\bar B+b)\cdot \pa^\al u+ \curlv \pa^\al E \cdot \pa^\al b \) d\V+\int_{\Omega_-} \(F^{1,\al}\cdot \pa^\al u + F^{3,\al}  \cdot \pa^\al b\)d\V.\nonumber
\end{align}
The integration by parts over $\Omega_-$  shows that, by using the thirteenth, fourth,    eleventh  and tenth  equations in \eqref{MHDva},
\begin{align}\label{en_iden_al_1}
&  \int_{\Omega_-}\left(\curlv \pa^\al b\times (\bar B+b)\cdot \pa^\al u+ \curlv \pa^\al E \cdot \pa^\al b \) d\V\nonumber
\\&\quad= \int_{\Omega_-}\left(  (\bar B+b)\times \pa^\al u+ \pa^\al E \) \cdot  \curlv \pa^\al b\,d\V+\int_\Sigma \n\times\pa^\al E\cdot    \pa^\al b
\nonumber
\\&\quad= \int_{\Omega_-}\left( -\kappa\curlv \pa^\al b +F^{4,\al} \) \cdot  \curlv \pa^\al b\,d\V+\int_\Sigma \n\times\pa^\al \hat E\cdot    \pa^\al \hb +\int_\Sigma   F^{7,\al}\cdot    \pa^\al b   .
\end{align}
The integration by parts over $\Omega_+$  yields, by using the fifth, sixth,  and fourteenth equations in \eqref{MHDva}, 
\begin{align}\label{en_iden_al_11}
  &\int_\Sigma \n\times\pa^\al \hat E\cdot    \pa^\al \hb =  -\int_{\Omega_+} \curlv \pa^\al \hat E \cdot    \pa^\al \hb \, d\V+ \int_{\Omega_+}  \pa^\al \hat E \cdot    \curlv\pa^\al \hb \, d\V+\int_{\Sigma_+} e_3 \times \pa^\al \hat E\cdot    \pa^\al \hb\nonumber
  \\& \quad=  -\int_{\Omega_+} \(\D_t \pa^\al {\hb}-\hat F^{3,\al}\) \cdot    \pa^\al \hb\, d\V+ \int_{\Omega_+}  \pa^\al \hat E \cdot   \hat F^{4,\al} \, d\V + \int_{\Sigma_+}     \pa^\al \hb \times e_3 \cdot \pa^\al \hat E\nonumber
  \\&\quad =  -\hal \dtt\int_{\Omega_+}   \abs{\pa^\al {\hb}}^2\, d\V-\hal\int_\Sigma \dt \eta \abs{\pa^\al {\hb}}^2+\int_{\Omega_+}  \(\hat F^{3,\al}  \cdot    \pa^\al \hb+ \pa^\al \hat E \cdot   \hat F^{4,\al} \) d\V.
\end{align}
By the twelfth, eighth, seventh and second equations in \eqref{MHDva}, one integrates by parts over $\Omega_-$ to obtain
\begin{align}\label{ffhh}
\int_{\Omega_-}   \nabla^\varphi \pa^\al p   \cdot \pa^\al u \, d\V &= \int_{\Sigma }   \pa^\al p    \pa^\al u\cdot \N    - \int_{\Omega_-}   \pa^\al p   \diva \pa^\al u \, d\V\nonumber
\\&= \int_{\Sigma }   ( -  \sigma \pa^\alpha H     )   \(\partial_{t} \pa^\alpha \eta- F^{5,\al}\)  - \int_{\Omega_-}   \pa^\al p   F^{2,\al} \, d\V.
\end{align}
By the ninth equation in \eqref{MHDva}, one may write
\begin{align}
 - \int_{\Sigma }   \sigma \pa^\alpha H          \dt \pa^\al \eta  =-\int_{\Sigma } \sigma\( \diverge_h  \( \dis\frac{ \nabla_h \pa^\alpha \eta}{\sqrt{1+|\nabla_h \eta|^2}}- \frac{\nabla_h \eta \cdot \nabla_h \pa^\alpha \eta}{\sqrt{1+|\nabla_h \eta|^2}^3}\nabla_h \eta \)+F^{6,\al}\)\dt \pa^\al \eta .
\end{align}
Integrating by parts in both $x_h$ and $t$ yields
\begin{align}\label{en_iden_al_2} 
&  -\int_{\Sigma } \sigma  \diverge_h  \( \dis\frac{ \nabla_h \pa^\alpha \eta}{\sqrt{1+|\nabla_h \eta|^2}}- \frac{\nabla_h \eta \cdot \nabla_h \pa^\alpha \eta}{\sqrt{1+|\nabla_h \eta|^2}^3}\nabla_h \eta \)
 \dt \pa^\al \eta \nonumber
 \\&\quad = \int_{\Sigma } \sigma    \( \dis\frac{ \nabla_h \pa^\alpha \eta}{\sqrt{1+|\nabla_h \eta|^2}}- \frac{\nabla_h \eta \cdot \nabla_h \pa^\alpha \eta}{\sqrt{1+|\nabla_h \eta|^2}^3}\nabla_h \eta \)\cdot
 \dt \nabla_h\pa^\al \eta\nonumber
 \\&\quad = \hal\dtt  \int_{\Sigma } \sigma \(\frac{ |    \nabla_h \pa^\al   \eta|^2  }{\sqrt{1+| \nabla_h \eta|^2}}-\frac{ | \nabla_h\eta  \cdot    \nabla_h  \pa^\al  \eta|^2  }{\sqrt{1+| \nabla_h\eta |^2}^3} \) -\i^\al,
\end{align}
where
\begin{align}
 \i^\al=& \frac{1}{2}  \int_{\Sigma}   \sigma \dt\(\frac{ 1  }{\sqrt{1+| \nabla_h \eta|^2} } \) \abs{   \nabla_h \pa^\al \eta}^2 - \frac{1}{2}  \int_{\Sigma}   \sigma\dt\( \frac{ 1  }{ \sqrt{1+| \nabla_h\eta |^2}^3}\)| \nabla_h\eta  \cdot    \nabla_h  \pa^\al  \eta|^2\nonumber
\\&-  \int_{\Sigma}      \frac{ \nabla_h\eta  \cdot \nabla_h  \pa^\al  \eta}{\sqrt{1+| \nabla_h\eta |^2}^3} \dt \nabla_h\eta \cdot   \nabla_h\pa^\al\eta .
\end{align}
Consequently, in light of \eqref{en_iden_al_1}--\eqref{en_iden_al_2}, \eqref{en_iden_al_0} yields the following energy identity:
\begin{align}\label{en_iden_al}
& \hal\dtt  \left(\int_{\Omega_-} \( \abs{\pa^\al u}^2   +\abs{\pa^\al b}^2\)   d\V+\int_{\Omega_+}  \abs{\pa^\al \hb}^2    d\V+\int_{\Sigma } \sigma \(\frac{ |    \nabla_h \pa^\al   \eta|^2  }{\sqrt{1+| \nabla_h \eta|^2}}-\frac{ | \nabla_h\eta  \cdot    \nabla_h  \pa^\al  \eta|^2  }{\sqrt{1+| \nabla_h\eta |^2}^3} \)\right)   \nonumber
\\&\quad+   \kappa  \int_{\Omega_-} |\curlv \pa^\al b|^2\, d\V  \nonumber
\\&\quad= \i^\al+ \int_{\Omega_-}\( F^{1,\al}\cdot \pa^\al u +\pa^\al p   F^{2,\al}+ F^{3,\al} \cdot \pa^\al b  + F^{4,\al} \cdot  \curlv \pa^\al b \)d\V
\\&\qquad+\int_{\Omega_+}  \(\hat F^{3,\al}  \cdot    \pa^\al \hb+ \pa^\al \hat E \cdot   \hat F^{4,\al} \) d\V+ \int_{\Sigma }   \(- F^{5,\al} \sigma \pa^\alpha H +\sigma   F^{6,\al}\dt \pa^\al \eta  +F^{7,\al}\cdot    \pa^\al b\right)     .\nonumber
\end{align}

We now estimate the right hand side of \eqref{en_iden_al}.  First, one has
\beq\label{hh00}
\i^\al\ls   \sqrt{\fe{N+4} }\abs{   \nabla_h \pa^\al \eta}_0^2\ls  \sqrt{\fe{N+4} }\se{2N}.
\eeq
 It follows from \eqref{p_F_e_01}  and \eqref{curl_e_2} with $n=2N$  that 
 \begin{align}
&\int_{\Omega_-} \(F^{1,\al}\cdot \pa^\al u  +F^{3,\al} \cdot \pa^\al b  + F^{4,\al} \cdot  \curlv \pa^\al b\) d\V +\int_{\Omega_+}  \hat F^{3,\al}  \cdot    \pa^\al \hb\,d\V\nonumber
 \\ &\quad\ls \norm{ F^{1,\al}}_0\norm{ \pa^\al u}_0+\norm{ F^{3,\al}}_0\norm{ \pa^\al b}_0+\norm{ F^{4,\al}}_0\norm{\curlv \pa^\al b}_0+\norm{\hat F^{3,\al}}_0\norm{ \pa^\al \hb}_0 \nonumber
 \\&\quad\ls\sqrt{\fe{N+4}\se{2N} }\(\sqrt{\se{2N}  }+\sqrt{\fd{2N}  }\),
\end{align}
\begin{align}
 \int_{\Sigma }   \sigma    F^{6,\al}\dt \pa^\alpha \eta      \ls \abs{  F^{6,\al}}_{1/2}\abs{\dt \pa^\alpha \eta }_{-1/2}     \ls  \sqrt{\fe{N+4} \se{2N}   } \sqrt{\se{2N}  }
\end{align}
and by the trace theory,
\begin{align}\label{hh11}
\int_{\Sigma }  F^{7,\al}\cdot \p^\al b  \le \abs{  F^{7,\al}}_{0}\abs{  \pa^\alpha b }_{0}
\ls  \abs{  F^{7,\al}}_{0}\norm{  \pa^\alpha b }_{1} 
  \ls  \sqrt{\fe{N+4}\se{2N} }\sqrt{\fd{2N}  }.
\end{align} 

Now we turn to estimate the most delicate three remaining terms. As explained in Section \ref{sec_2}, one needs to consider the cases $\alpha_0\le 2N-1$ and $\alpha_0=2N$ separately. For the case $\al_0\le 2N-1$,  one has that by \eqref{p_F_e_01} and \eqref{p_F_e_02},
\begin{align}
&\int_{\Omega_-}   \pa^\al p  F^{2,\al} \, d\V+ \int_{\Omega_+}   \pa^\al \hat E \cdot   \hat F^{4,\al}  d\V  -\int_{\Sigma }   \sigma \pa^\alpha H   F^{5,\al}\nonumber
 \\ &\quad \ls \norm{ \pa^\al p}_0\norm{ F^{2,\al}}_0+\norm{ \pa^\al \hat E}_0\norm{\hat F^{4,\al}}_0+\abs{ \pa^\alpha H }_{-1/2}\abs{ F^{5,\al}}_{1/2}\nonumber
 \\ &\quad\ls \sqrt{\se{2N}  } \sqrt{\fe{N+4}\se{2N} }.
\end{align}
For the case $\al_0=2N$, one integrates by parts in  $t$ to have, by \eqref{p_F_e_01} and \eqref{p_F_e_01'},
\begin{align}\label{es_2N_4}
& \int_{\Omega_+}   \dt^{2N} \hat E \cdot \hat F^{4,(2N,0)} \, d\V
=\dtt \int_{\Omega_+}   \dt^{2N-1} \hat E \cdot \hat F^{4,(2N,0)} \, d\V
-\int_{\Omega_+}   \dt^{2N-1} \hat E \cdot  \dt \(\hat F^{4,(2N,0)}\p_3\varphi \)\nonumber
\\&\quad
  \ls \dtt \int_{\Omega_+}   \dt^{2N-1} \hat E \cdot \hat F^{4,(2N,0)} \, d\V+ \norm{ \dt^{2N-1} \hat E}_1\(\norm{\dt \hat F^{4,(2N,0)}}_{-1}+\norm{  \hat F^{4,(2N,0)}}_{0}\)
 \nonumber
\\&\quad
  \ls \dtt \int_{\Omega_+}   \dt^{2N-1} \hat E \cdot \hat F^{4,(2N,0)} \, d\V+\sqrt{\se{2N}} \sqrt{\fe{N+4} \(\se{2N}+\fd{2N}\)  }.
\end{align}

The treatment of the remaining two terms is more involved. The difficulty is that there is no any estimate of $\dt^{2N} p$ and  so one needs to integrate by parts in $t$ for the pressure term, and also there is a $1/2$ regularity loss of $\dt^{2N} H$  so that it is insufficient to control the surface tension term. The crucial observation here is that these two terms will enjoy some cancellation by performing some careful computations. We start with the integration by parts in $t$ for the pressure term, and make use of a variant of the expression of $F^{2,(2N,0)}$ defined by \eqref{f2a}. Indeed, $\diva u=0$  yields
\beq\label{div_free}
\diva u\pa_3\varphi=\N\cdot \pa_3u +\pa_3\varphi  \diverge_h  u_h =0.
\eeq
Applying $\dt^{2N}$ to \eqref{div_free} and using the second equation in \eqref{MHDva}, one gets that
\beq
-\pa_3\varphi F^{2,(2N,0)}   = \[ \pa_t^{2N} ,  -\nabla_h\bar\eta  \] \cdot \partial_3u_h +  \[ \pa_t^{2N} ,  \partial_3 \bar\eta \] \diverge_h  u_h  .
\eeq
Moreover, one needs to single out  the highest $2N-1$ order time derivative terms of $ u_h$ and the highest $2N$ order time derivative terms of $\eta$ as
\begin{align}
-\pa_3\varphi F^{2,(2N,0)}= \sum_{i=1}^5 F^{2,(2N,0)}_i,
\end{align}
where
\begin{align}
& F^{2,(2N,0)}_1= -2N\pa_t  \nabla_h\bar\eta  \cdot   \pa_t^{2N-1} \pa_3u_h,
\\ & F^{2,(2N,0)}_2=  2N\pa_t \pa_3 \bar\eta \pa_t^{2N-1} \diverge_h  u_h ,
\\ & F^{2,(2N,0)}_3=  - \pa_t^{2N} \nabla_h\bar\eta  \cdot \pa_3u _h,
\\ &  F^{2,(2N,0)}_4= \pa_t^{2N} \pa_3 \bar\eta   \diverge_h  u_h ,
\\ &  F^{2,(2N,0)}_5=   \sum_{\ell=2 }^{2N-1} C_{2N}^\ell\(-\pa_t^\ell  \nabla_h\bar\eta \cdot   \pa_t^{2N-\ell}\pa_3u_h+\pa_t^\ell \pa_3 \bar\eta \pa_t^{2N-\ell} \diverge_h  u_h \).
\end{align}
 Accordingly,
\beq\label{sur_q_1}
\int_{\Omega_-}   \dt^{2N} p  F^{2,(2N,0)} \, d\V= -\sum_{i=1}^5 \int_{\Omega_-}  \dt^{2N} pF^{2,(2N,0)}_i   .
\eeq
Integrating by parts in  $t$ for the last four terms yields
\begin{align}
-\sum_{i=2}^5\int_{\Omega_-} \dt^{2N} p   F^{2,(2N,0)}_i = -\sum_{i=2}^5\dtt\int_{\Omega_-}  \pa_t^{2N-1} p F^{2,(2N,0)}_i 
 + \sum_{i=2}^5\int_{\Omega_-} \pa_t^{2N-1} p\dt   F^{2,(2N,0)}_i  .
\end{align}
One can check easily that $\norm{\dt F^{2,(2N,0)}_5}_0\ls \sqrt{\fe{N+4} \se{2N}}$ as in Lemma  \ref{le_F_2N}. Thus
\beq
\int_{\Omega_-}\pa_t^{2N-1} p \dt   F^{2,(2N,0)}_5  \le \norm{ \pa_t^{2N-1} p}_0\norm{\dt F^{2,(2N,0)}_5}_0\ls  \sqrt{  \se{2N}}\sqrt{\fe{N+4} \se{2N}} .
\eeq
Upon an integration by parts in $x_3$ and estimating as in Lemma \ref{le_F_2N}, one has
\begin{align}
&\int_{\Omega_-}  \pa_t^{2N-1} p\dt   F^{2,(2N,0)}_4  =\int_{\Omega_-}\pa_t^{2N-1} p  \(\pa_t^{2N+1} \pa_3 \bar\eta   \diverge_h  u_h +\pa_t^{2N} \pa_3 \bar\eta  \dt \diverge_h  u_h   \)\nonumber
\\&\quad=\int_{\Sigma}\pa_t^{2N-1} p   \pa_t^{2N+1}  \eta   \diverge_h  u_h   -\int_{\Omega_-}\(\!\pa_t^{2N+1} \bar \eta\pa_3\(\pa_t^{2N-1} p     \diverge_h  u_h \) +  \pa_t^{2N-1} p\pa_t^{2N} \pa_3 \bar\eta  \dt \diverge_h  u_h  \! \)\nonumber
\\&\quad\ls \abs{ \pa_t^{2N+1} \eta }_{-1/2} \abs{\pa_t^{2N-1} p   \diverge_h  u_h }_{1/2}  +\norm{\pa_t^{2N+1} \bar \eta}_0\norm{\pa_3 \(\pa_t^{2N-1} p     \diverge_h  u_h \)}_0
\nonumber\\&\qquad+\norm{\pa_t^{2N} \pa_3 \bar\eta   }_0\norm{\pa_t^{2N-1} p \dt \diverge_h  u_h }_0\nonumber
\\& \quad\ls  \sqrt{  \se{2N}}\sqrt{\fe{N+4}  }\sqrt{  \se{2N}}.
\end{align}
Similarly, by integrating by parts in $x_h$, one deduces
\beq
\int_{\Omega_-}\pa_t^{2N-1} p \(\dt   F^{2,(2N,0)}_2+ \dt   F^{2,(2N,0)}_3\)\ls  \sqrt{  \se{2N}}\sqrt{\fe{N+4} \se{2N}} .
\eeq

It remains to deal with the most difficult term, the first term involving $F^{2,(2N,0)}_1$ in \eqref{sur_q_1}.  One integrates by parts in $x_3$ first to get
\beq\label{sur_q_2}
-\int_{\Omega_-}  \dt^{2N} p F^{2,(2N,0)}_1  =\int_{\Sigma }   \dt^{2N} p 2N\pa_t \nabla_h\eta \cdot   \pa_t^{2N-1} u_h   -\int_{\Omega_-}    \pa_3\left( \dt^{2N} p2N\pa_t  \nabla_h\bar\eta \right) \cdot \pa_t^{2N-1} u_h   .
\eeq
Then integrating by parts in $t$ for the second term in the right hand side of \eqref{sur_q_2} yields
\begin{align}\label{sur_q_3}
&- \int_{\Omega_-}    \pa_3\left( \dt^{2N} p2N\pa_t  \nabla_h\bar\eta \right) \cdot \pa_t^{2N-1} u_h   \nonumber
\\ &\quad=-\dtt\int_{\Omega_-}  \pa_3\( \pa_t^{2N-1} p2N\pa_t  \nabla_h\bar\eta \)\cdot  \pa_t^{2N-1} u_h    \nonumber
\\&\qquad+\int_{\Omega_-}  \( \pa_3\( \pa_t^{2N-1} p2N\pa_t  \nabla_h\bar\eta \)\cdot  \pa_t^{2N} u_h +  \pa_3\( \pa_t^{2N-1} p2N\pa_t^2  \nabla_h\bar\eta \)\cdot  \pa_t^{2N-1} u_h \)\nonumber
\\ &\quad\le  -\dtt\int_{\Omega_-} \pa_3\( \pa_t^{2N-1} p2N\pa_t  \nabla_h\bar\eta \)\cdot  \pa_t^{2N-1} u_h     + C\sqrt{\fe{N+4} \se{2N}}\sqrt{  \se{2N}}.
\end{align}
Note carefully that we integrate by parts in $x_3$ here first rather than in $t$ since there are no estimates for $\dt^{2N}u_h$ on the boundary. This also indicates the difficulty in controlling the first term in the right hand side of \eqref{sur_q_2} since one can no longer integrate by parts in $t$. Recall here that there is also anther term out of control, which is the surface tension term in the right hand side of \eqref{en_iden_al} when $\al=(2N,0)$. Our crucial observation is that there is a cancellation between them since $\dt^{2N} p= -\sigma \dt^{2N} \eta  $ on $\Sigma $. Indeed, one has
\begin{align}\label{sur_q_4}
 & -\int_{\Sigma }   \sigma \dt^{2N}H      F^{5,(2N,0)}  +\int_{\Sigma }   \dt^{2N} p 2N\pa_t \nabla_h\eta \cdot   \pa_t^{2N-1} u_h  \nonumber
 \\& \quad=-\int_{\Sigma } \sigma \pa_t^{2N}  H  \( F^{5,(2N,0)} +  2N\pa_t  \nabla_h\eta \cdot   \pa_t^{2N-1} u_h\) +\int_{\Sigma }  (\sigma \pa_t^{2N}  H+\dt^{2N} p )  2N\pa_t \nabla_h\eta \cdot   \pa_t^{2N-1} u_h \nonumber
 \\&\quad=\int_{\Sigma } \sigma \pa_t^{2N}  H  \(u_h\cdot \nabla_h \pa_t^{2N} \eta +\tilde{F}^{5,(2N,0)} \)  
 \end{align}
 where
\beq
\tilde{F}^{5,(2N,0)}=\sum_{\ell=2}^{2N-1}C_{2N}^\ell \pa_t^\ell \nabla_h\eta \cdot    \pa_t^{2N-\ell} u_h.
\eeq
Note that $\abs{ \tilde{F}^{5,(2N,0)}}_1\ls \sqrt{\fe{N+4} \se{2N}}$ as in Lemma  \ref{le_F_2N}. So integrating by parts in $x_h$ yields
\begin{align}\label{sur_q_5}
& \int_{\Sigma } \sigma \pa_t^{2N}  H  \tilde{F}^{5,(2N,0)}  =-\int_{\Sigma } \sigma \dt^{2N}\( \frac{  \nabla_h  \eta}{\sqrt{1+| \nabla_h \eta|^2}} \) \cdot \nabla_h  \tilde{F}^{5,(2N,0)}   \nonumber
\\&\quad\ls \abs{  \dt^{2N}\( \frac{  \nabla_h\eta }{\sqrt{1+| \nabla_h\eta |^2}} \)}_{0}\abs{\tilde{F}^{5,(2N,0)}}_{1} \ls  \sqrt{  \se{2N}}\sqrt{\fe{N+4} \se{2N}}.
\end{align}
It follows from the ninth equation in \eqref{MHDva} that
\begin{align}
&\int_{\Sigma }  \sigma \pa_t^{2N}  H  u_h\cdot \nabla_h\dt^{2N} \eta \nonumber
\\&\quad= \int_{\Sigma }\sigma \(\diverge_h  \( \frac{  \nabla_h  \dt^{2N} \eta}{\sqrt{1+| \nabla_h\eta |^2}}- \frac{ \nabla_h\eta  \cdot \nabla_h \dt^{2N} \eta}{\sqrt{1+| \nabla_h\eta |^2}^3} \nabla_h\eta  \)+F^{6,(2N,0)} \) u_h\cdot \nabla_h\dt^{2N} \eta   .
\end{align}
Then \eqref{p_F_e_01} implies
\beq
\int_{\Sigma }\sigma   F^{6,(2N,0)}  u_h\cdot \nabla_h  \dt^{2N} \eta  \ls \abs{F^{6,(2N,0)} }_0\abs{\dt^{2N} \eta }_1\ls  \sqrt{\fe{N+4} \se{2N}}\sqrt{  \se{2N}}.
\eeq
Integrating by parts in  $x_h$, one can deduce that
\begin{align}
&\int_{\Sigma }\sigma  \diverge_h   \(\frac{     \nabla_h  \dt^{2N} \eta}{\sqrt{1+|  \nabla_h\eta |^2}} \) u_h\cdot \nabla_h   \dt^{2N} \eta =-\int_{\Sigma }\sigma  \frac{     \nabla_h  \dt^{2N} \eta}{\sqrt{1+|  \nabla_h\eta |^2}}  \cdot \nabla_h  (u_h\cdot \nabla_h   \dt^{2N} \eta )  \nonumber\\&\quad=-\int_{\Sigma }\sigma \( \frac{    \nabla_h  \dt^{2N} \eta}{\sqrt{1+| \nabla_h  h|^2} } \cdot \nabla_h u_h\cdot\nabla_h \dt^{2N} \eta  -\hal \diverge_h  \( \frac{ u_h  }{\sqrt{1+|  \nabla_h\eta |^2}} \) \abs{  \nabla_h   \dt^{2N} \eta}^2  \)
\nonumber
\\&\quad\ls \sqrt{\fe{N+4} }\abs{  \dt^{2N} \eta}_1^2\ls \sqrt{\fe{N+4} }\se{2N}.
\end{align}
Similarly,  one has
\begin{align}\label{sur_q_6}
 -\int_{\Sigma }\sigma \diverge_h  \(  \frac{ \nabla_h\eta  \cdot \nabla_h \dt^{2N} \eta}{\sqrt{1+| \nabla_h\eta |^2}^3} \nabla_h\eta   \) u_h\cdot \nabla_h\dt^{2N} \eta\ls \sqrt{\fe{N+4} }\se{2N}.
\end{align}
 Hence, it follows from \eqref{sur_q_5}--\eqref{sur_q_6} and \eqref{sur_q_4} that
\beq
 -\int_{\Sigma }   \sigma \dt^{2N} \eta     F^{5,(2N,0)}  +\int_{\Sigma }   \dt^{2N} p 2N\pa_t \nabla_h\eta \cdot   \pa_t^{2N-1} u_h   \ls\sqrt{\fe{N+4} }\se{2N}.
\eeq
This together with  \eqref{sur_q_1}--\eqref{sur_q_3}  implies that
\beq\label{es_2N_5}
- \int_{\Sigma }  \sigma  \dt^{2N}   \eta F^{5,(2N,0)}  +\int_{\Omega_-}   \dt^{2N} p  F^{2,(2N,0)} \, d\V \le-\dtt \mathcal{B}_{2N} +C\sqrt{\fe{N+4} }\se{2N}.
\eeq
where
\beq
\mathcal{B}_{2N}:= \sum_{i=2}^5 \int_{\Omega_-}  \pa_t^{2N-1} p F^{2,(2N,0)}_i
 + \int_\Omega\pa_3\( \pa_t^{2N-1} p2N\pa_t  \nabla_h\bar\eta \)  \pa_t^{2N-1} u_h    .
\eeq

As a consequence of the estimates \eqref{hh00}--\eqref{es_2N_4}  and  \eqref{es_2N_5},   one deduces from \eqref{en_iden_al}  with summing over such $\al$ and \eqref{en_iden} that,  by \eqref{curl_e_2} with $n=2N$ and Cauchy's inequality and then integrating in time from $0$ to $t$,
\beq
\seb{2N}(t)+ \int_0^t  \sdb{2N}(s)\,ds\ls   \seb{2N}(0)+\mathcal{B}_{2N}(0)-\mathcal{B}_{2N}(t) +\int_0^t\sqrt{\fe{N+4}(s)}  \(\se{ 2N}+\fd{2N} \)\,ds.
\eeq
Note that $\norm{F^{2,(2N,0)}_i}_0\ls \sqrt{\fe{N+4} \se{2N}},\ i=2,\dots,5,$ as Lemma  \ref{le_F_2N}. Thus
\beq
 \abs{\mathcal{B}_{2N}}\le \sum_{i=2}^5\norm{\pa_t^{2N-1} p}_0\norm{F^{2,(2N,0)}_i}_0+\sqrt{\fe{N+4} \se{2N}}\sqrt{  \se{2N}}\ls    (\se{ 2N} )^{3/2}  .
\eeq
Then the estimate \eqref{en_ev_2N} follows.
 \end{proof}

%%%%%%%%%%%%%%%%%%%%%%%%%%%%%%%%%%%%%%%%%%%%%%%%%%%%%%
\subsection{Energy evolution at $N+4,\dots, 2N-2$ levels}
%%%%%%%%%%%%%%%%%%%%%%%%%%%%%%%%%%%%%%%%%%%%%%%%%%%%%%

Now we present the following time-differential tangential energy evolution estimate, at $N+4,\dots, 2N-2$ levels.
\begin{prop}\label{evolution_n}
For $n=N+4,\dots,2N-2$, it holds that
\beq \label{en_ev_n}
\frac{d}{dt}\left(\bar{\mathcal{E}}_{n}+\mathcal{B}_{n}\right)+ \bar{\mathcal{D}}_{n} \ls   \sqrt{\se{2N}   } { \fd{ n}},
\eeq
where $\mathcal{B}_{n}$ is defined by \eqref{Bn_def} below and satisfies the estimate
\beq\label{Bn_es}
\abs{\mathcal{B}_{n}}\ls \sqrt{\se{2N}  } { \fe{n} }.
\eeq
\end{prop}
\begin{proof}
Let $n$ denote $N+4\dots, 2N-2$ throughout the proof and $\al\in \mathbb{N}^{1+2}$  such that  $1\le |\al|\le n$.  The equality \eqref{en_iden_al}  in the proof of Proposition \ref{evolution_2N} holds also here, and 
we will estimate the right hand side in a quite different way from  the arguments that lead to the estimates \eqref{hh00}--\eqref{es_2N_4} and  \eqref{es_2N_5}.

First,  one has
\begin{align}\label{es_n_1}
 \i^\al \ls \sqrt{\fd{N+4} }\abs{\nabla_h\pa^\al \eta}_1\abs{\nabla_h\pa^\al \eta}_{-1}\ls \sqrt{\fd{N+4} } \sqrt{ \se{2N}}\sqrt{  \fd{n}}.
\end{align} 
It follows from \eqref{p_F_e_03} and \eqref{curl_e_2}  that
\begin{align}\label{hh112}
& \int_{\Omega_-} \(F^{3,\al} \cdot \pa^\al b  + F^{4,\al} \cdot  \curlv \pa^\al b\) d\V +\int_{\Omega_+}  \hat F^{3,\al}  \cdot    \pa^\al \hb\,d\V
\nonumber
 \\&\quad \ls \norm{ F^{3,\al}}_0\norm{ \pa^\al b}_0+\norm{ F^{4,\al}}_0\norm{\curlv \pa^\al b}_0 +\norm{\hat F^{3,\al}}_0\norm{ \pa^\al \hb}_0\nonumber
 \\&\quad \ls  \sqrt{\fd{N+4}\se{2N} }\sqrt{\fd{n} } 
\end{align}
and by trace theory,
\begin{align}
 \int_{\Sigma } F^{7,\al}\cdot \p^\al b   \ls  \abs{  F^{7,\al}}_{0}\abs{  \pa^\alpha b }_{0}  \ls  \sqrt{\fd{N+4}\se{2N} }\sqrt{\fd{n} }.
\end{align} 

Next, we consider the  terms involving $F^{1,\al}$ and $F^{6,\al}$. If $|\al|\le n-1$, then by \eqref{p_F_e_03}, one has
\begin{align}
\int_{\Omega_-}F^{1,\al}\cdot \pa^\al u\,d\V+  \int_{\Sigma }  \sigma    F^{6,\al}\dt \pa^\al \eta & \ls  
\norm{ F^{1,\al}}_0\norm{ \pa^\al u}_0+\abs{F^{6,\al}}_0\abs{\dt \pa^\alpha \eta}_0\nonumber
\\&\ls  \sqrt{\fd{N+4} \se{2N}} \sqrt{  \fd{n}}.
\end{align}
If $\abs{\al}=n, \al_1+\al_2\ge 1$, then  by \eqref{p_F_e_03}, one obtains 
\begin{align}
 \int_{\Omega_-}F^{1,\al}\cdot \pa^\al u\,d\V+ \int_{\Sigma }  \sigma     F^{6,\al}\dt \pa^\al \eta & \ls  
 \norm{  F^{1,\al}}_1\norm{ \pa^{\al} u}_{-1}+\abs{F^{6,\al}}_1\abs{\dt \pa^{\al} \eta}_{-1}\nonumber
 \\&\ls  \sqrt{\fd{N+4}\se{2N} } \sqrt{  \fd{n}}.
\end{align}
The remaining case is  that when $\al_0=n$, and integrating by parts in $t$ and using \eqref{p_F_e_03} show that
\begin{align}
&\int_{\Omega_-}F^{1,(n,0)}\cdot  \dt^{n} u\,d\V=\dtt\int_{\Omega_-}F^{1,(n,0)}\cdot  \dt^{n-1} u\,d\V-\int_{\Omega_-}\dt\(F^{1,(n,0)}\pa_3\varphi\)\cdot   \dt^{n-1} u \nonumber
\\&\quad\ls \dtt\int_{\Omega_-}F^{1,(n,0)}\cdot  \dt^{n-1} u\,d\V+ \(\norm{   F^{1,(n,0)}}_0+\norm{ \dt F^{1,(n,0)}}_0\)\norm{ \dt^{n-1} u}_0 \nonumber
\\&\quad\ls \dtt\int_{\Omega_-}F^{1,(n,0)}\cdot  \dt^{n-1} u\,d\V+ \sqrt{\fd{N+4}\se{2N} }\sqrt{  \fd{n}}
\end{align}	
and
\begin{align}
&\int_{\Sigma }  \sigma    F^{6,(n,0)}\dt \dt^n \eta   =\dtt \int_{\Sigma }  \sigma      F^{6,(n,0)} \dt^n \eta  -\int_{\Sigma }  \sigma  \dt    F^{6,(n,0)}  \dt^n \eta \nonumber
\\&\quad\ls    \dtt \int_{\Sigma }  \sigma     F^{6,(n,0)} \dt^n \eta   + \abs{\dt F^{6,(n,0)}}_0\abs{\dt^n \eta}_0\nonumber
\\&\quad\ls    \dtt \int_{\Sigma }  \sigma     F^{6,(n,0)} \dt^n \eta   + \sqrt{\fd{N+4}\se{2N} } \sqrt{  \fd{n}} 
\end{align}

Next, we treat the terms involving $F^{2,\al}$ and $\hat F^{4,\al}$. If $|\al|\le n-1, \al_0\le n-2$, then \eqref{p_F_e_03} implies
\begin{align}
\int_{\Omega_-}   \pa^\al p   F^{2,\al} \, d\V+ \int_{\Omega_+}   \pa^\al \hat E \cdot   \hat F^{4,\al}  d\V   &\ls   \norm{ \pa^\al p }_0\norm{ F^{2,\al}}_0+\norm{ \pa^\al \he }_0\norm{  \hat F^{4,\al} }_0 \nonumber
\\&\ls \sqrt{  \fd{n}}\sqrt{\fd{N+4}\se{2N} }.
\end{align}
If $\al=(n-1,0)$, then one integrates by parts in $t$  to get 
\begin{align}
&\int_{\Omega_-}   \dt^{n-1} p   F^{2,(n-1,0)} \, d\V+ \int_{\Omega_+} \dt^{n-1}\hat E \cdot   \hat F^{4,(n-1,0)}  d\V  
\nonumber
\\&\quad=	\dtt\(\int_{\Omega_-}   \dt^{n-2} p   F^{2,(n-1,0)} \, d\V+ \int_{\Omega_+} \dt^{n-2}\hat E \cdot   \hat F^{4,(n-1,0)}  d\V  \)\nonumber
\\&\qquad-	\int_{\Omega_-}   \dt^{n-2} p  \dt \(F^{2,(n-1,0)} \pa_3\varphi\)- \int_{\Omega_+} \dt^{n-2}\hat E \cdot \dt\(  \hat F^{4,(n-1,0)}  \pa_3\varphi\)\nonumber
\\&\quad	\ls		\dtt\(\int_{\Omega_-}   \dt^{n-2} p   F^{2,(n-1,0)} \, d\V+ \int_{\Omega_+} \dt^{n-2}\hat E \cdot   \hat F^{4,(n-1,0)}  d\V  \)\nonumber
\\&\qquad+ \norm{ \dt^{n-2}  p }_0\(\norm{ F^{2,(n-1,0)}}_0+\norm{ \dt F^{2,(n-1,0)}}_0\)\nonumber
\\&\qquad+ \norm{ \dt^{n-2} \he }_0\(\norm{\hat F^{4,(n-1,0)}}_0+\norm{ \dt \hat F^{4,(n-1,0)}}_0\)\nonumber
\\&\quad\ls\dtt\(\int_{\Omega_-}   \dt^{n-2} p   F^{2,(n-1,0)} \, d\V+ \int_{\Omega_+} \dt^{n-2}\hat E \cdot   \hat F^{4,(n-1,0)}  d\V  \)+ \sqrt{  \fd{n}}\sqrt{\fd{N+4}\se{2N} }.
\end{align}
If $\abs{\al}=n $ and $\al_1+\al_2\ge 2$,  then 
\begin{align}
 \int_{\Omega_-}   \pa^\al p  F^{2,\al} \, d\V + \int_{\Omega_+}   \pa^\al \hat E \cdot   \hat F^{4,\al}  d\V 
 &\ls \norm{ \pa^{\al} p}_{-1}\norm{ F^{2,\al}}_1+\norm{ \pa^{\al} \he}_{-1}\norm{ \hat F^{4,\al}}_1 \nonumber
\\& \ls \sqrt{  \fd{n}}\sqrt{\fd{N+4}\se{2N} }.
\end{align}
If $\abs{\al}=n $ and $\al_1+\al_2= 1$,  then   one writes $\al=(n-1,0)+\al'$ for $\al'\in \mathbb{N}^2$ with $\al'\le\al$ and $\abs{\al'}=1$ and then integrates by parts in $t$  to have, by \eqref{p_F_e_03},
\begin{align}
&\int_{\Omega_-}   \dt^{n-1}\pa^{\al'} p  F^{2,\al} \, d\V+\int_{\Omega_+}   \dt^{n-1}\pa^{\al'} \he \cdot  \hat F^{4,\al} \, d\V \nonumber
\\&\quad=	\dtt \(\int_{\Omega_-}   \dt^{n-2}\pa^{\al'} p  F^{2,\al} \, d\V+\int_{\Omega_+}   \dt^{n-2}\pa^{\al'} \he \cdot  \hat F^{4,\al} \, d\V\)\nonumber
\\&\qquad-	\int_{\Omega_-}   \dt^{n-2}\pa^{\al'} p \dt \(F^{2,\al }\pa_3\varphi \)-\int_{\Omega_+}   \dt^{n-2}\pa^{\al'} \he \cdot \dt \(\hat F^{4,\al }\pa_3\varphi \)\nonumber
\\&\quad \ls	\dtt \(\int_{\Omega_-}   \dt^{n-2}\pa^{\al'} p  F^{2,\al} \, d\V+\int_{\Omega_+}   \dt^{n-2}\pa^{\al'} \he \cdot  \hat F^{4,\al} \, d\V\)\nonumber
\\&\qquad+ \norm{ \dt^{n-2}  p}_1\(\norm{  F^{2,\al }}_0+\norm{ \dt F^{2,\al }}_0\)+ \norm{ \dt^{n-2}  \he}_1\(\norm{ \hat F^{4,\al}}_0+\norm{ \dt \hat F^{4,\al}}_0\)\nonumber
\\&\quad\ls\dtt \(\int_{\Omega_-}   \dt^{n-2}\pa^{\al'} p  F^{2,\al} \, d\V+\int_{\Omega_+}   \dt^{n-2}\pa^{\al'} \he \cdot  \hat F^{4,\al} \, d\V\)+  \sqrt{  \fd{n}}\sqrt{\fd{N+4}\se{2N} }.
\end{align}
The remaining case, $\al_0=n$, can be handled by the integration by parts in $t$ twice and using \eqref{p_F_e_03} as
\begin{align}\label{es_n_2}
&\int_{\Omega_-}    \dt^{n} p   F^{2,(n,0)} \, d\V+\int_{\Omega_+}    \dt^{n} \he \cdot  \hat F^{4,(n,0)} \, d\V\nonumber
\\& \quad=\dtt \( \int_{\Omega_-}    \dt^{n-1} p   F^{2,(n,0)} \, d\V+\int_{\Omega_+}    \dt^{n-1} \he \cdot  \hat F^{4,(n,0)} \, d\V\)\nonumber
\\& \qquad- 	\int_{\Omega_-}    \dt^{n-1} p  \dt \(F^{2,(n,0)} \pa_3\varphi\) -\int_{\Omega_+}    \dt^{n-1} \he \cdot  \dt \( \hat F^{4,(n,0)} \pa_3\varphi\) \nonumber
\\& \quad =\dtt\(\int_{\Omega_-}    \dt^{n-1} p   F^{2,(n,0)} \, d\V- 	\int_{\Omega_-}   \dt^{n-2} p \dt \(F^{2,(n,0)} \pa_3\varphi\)+\int_{\Omega_+}    \dt^{n-1} \he \cdot  \hat F^{4,(n,0)} \, d\V\right.\nonumber
\\& \qquad\left.-\int_{\Omega_+}    \dt^{n-2} \he \cdot  \dt \( \hat F^{4,(n,0)} \pa_3\varphi\) \) +\int_{\Omega_-}   \dt^{n-2} p \dt^2 \(F^{2,(n,0)} \pa_3\varphi\)
\nonumber
\\&\qquad+\int_{\Omega_+}    \dt^{n-2} \he \cdot  \dt^2 \( \hat F^{4,(n,0)} \pa_3\varphi\) \nonumber
\\&\quad\ls \dtt\(\int_{\Omega_-}    \dt^{n-1} p   F^{2,(n,0)} \, d\V- 	\int_{\Omega_-}   \dt^{n-2} p \dt \(F^{2,(n,0)} \pa_3\varphi\)
+\int_{\Omega_+}    \dt^{n-1} \he \cdot  \hat F^{4,(n,0)} \, d\V\right.\nonumber
\\& \qquad\left.-\int_{\Omega_+}    \dt^{n-2} \he \cdot  \dt \( \hat F^{4,(n,0)} \pa_3\varphi\) \) + \norm{ \dt^{n-2}  p}_0\(\norm{   F^{2,(n,0)}}_0+\norm{ \dt  F^{2,(n,0)}}_0+\norm{ \dt^2 F^{2,(n,0)}}_0\)\nonumber
\\&\qquad + \norm{ \dt^{n-2}  \he}_0\(\norm{    \hat F^{4,(n,0)}}_0+\norm{ \dt   \hat F^{4,(n,0)}}_0+\norm{ \dt^2  \hat F^{4,(n,0)}}_0\)\nonumber
\\&\quad \ls\dtt\(\int_{\Omega_-}    \dt^{n-1} p   F^{2,(n,0)} \, d\V- 	\int_{\Omega_-}   \dt^{n-2} p \dt \(F^{2,(n,0)} \pa_3\varphi\)+\int_{\Omega_+}    \dt^{n-1} \he \cdot  \hat F^{4,(n,0)} \, d\V\right.\nonumber
\\& \qquad\left.-\int_{\Omega_+}    \dt^{n-2} \he \cdot  \dt \( \hat F^{4,(n,0)} \pa_3\varphi\) \) + \sqrt{  \fd{n}}\sqrt{\fd{N+4}\se{2N} } .
\end{align}

As a consequence of the estimates \eqref{es_n_1}--\eqref{es_n_2},   one deduces from \eqref{en_iden_al}  with summing over $1\le \abs{\al}\le n$ and \eqref{en_iden} that, by \eqref{curl_e_2} and since $n\ge N+4$,
\beq
\frac{d}{dt}\left(\bar{\mathcal{E}}_{n}+\mathcal{B}_{n}\right)+ \bar{\mathcal{D}}_{n} \ls   \sqrt{\fd{N+4}\se{2N}  } \sqrt{\fd{ n}}+ \se{2N}  \fd{ n}\ls \sqrt{\se{2N}  } { \fd{n} },
\eeq
where 
\begin{align}\label{Bn_def}
\mathcal{B}_{n}:=& -\int_{\Omega_-}F^{1,(n,0)}\cdot  \dt^{n-1} u\,d\V-\int_{\Sigma }  \sigma    F^{6,(n,0)} \dt^n \eta-\int_{\Omega_-}   \dt^{n-2} p  F^{2,(n-1,0)} \, d\V\nonumber
\\& -\int_{\Omega_-}   \dt^{n-2} \pa^{\al'} p  F^{2,\al } \, d\V-\int_{\Omega_-}    \dt^{n-1} p   F^{2,(n,0)} \, d\V+ 	\int_{\Omega_-}   \dt^{n-2} p \dt \(F^{2,(n,0)} \pa_3\varphi\)  \nonumber
\\&- \int_{\Omega_+} \dt^{n-2}\hat E \cdot   \hat F^{4,(n-1,0)}  d\V-\int_{\Omega_+}   \dt^{n-2}\pa^{\al'} \he \cdot  \hat F^{4,\al} \, d\V-\int_{\Omega_+}    \dt^{n-1} \he \cdot  \hat F^{4,(n,0)} \, d\V
 \nonumber
\\&+\int_{\Omega_+}    \dt^{n-2} \he \cdot  \dt \( \hat F^{4,(n,0)} \pa_3\varphi\) 
\end{align}
By  \eqref{p_F_e_04}, one has
\begin{align}
\abs{\mathcal{B}_{n}}\ls  \sqrt{\fe{N+4}\se{2N}  }\sqrt{\fe{n} }\ls \sqrt{\se{2N}  } { \fe{n} }.
\end{align}
Then the estimates \eqref{en_ev_n} and \eqref{Bn_es} follow.
\end{proof}

%%%%%%%%%%%%%%%%%%%%%%%%%%%%%%%%%%%%%%%%%%%%%%%%%%%%%%	
\section{Improved estimates}\label{sec_6}
%%%%%%%%%%%%%%%%%%%%%%%%%%%%%%%%%%%%%%%%%%%%%%%%%%%%%%

In this section, making use of the tangential energy evolution estimates derived in Section \ref{sec_5}, we established the full energy and dissipation estimates by exploiting the important damping structure of \eqref{MHDv}  and some elaborate elliptic analysis.

%%%%%%%%%%%%%%%%%%%%%%%%%%%%%%%%%%%%%%%%%%%%%%%%%%%%%%
\subsection{Primary improvement of the dissipation estimates}
%%%%%%%%%%%%%%%%%%%%%%%%%%%%%%%%%%%%%%%%%%%%%%%%%%%%%%

In this subsection we first give certain improvements of  the tangential dissipation $\sdb{n}$, defined by \eqref{bar_dn_def}.

%%%%%%%%%%%%%%%%%%%%%%%%%%%%%%%%%%%%%%%%%%%%%%%%%%%%%%
\subsubsection{$H^1$-dissipation estimates of $b$ and full dissipation estimates of $\hb$}
%%%%%%%%%%%%%%%%%%%%%%%%%%%%%%%%%%%%%%%%%%%%%%%%%%%%%%

One may first apply the Hodge-type estimates  to improve the dissipation estimates of the magnetic fields $b$ and $\hb$ from the assumed control of $\sdb{n}$.  Define
\beq \label{di_ev_n}
\fdb{n}:=\sum_{j=0}^{n }\norm{ \dt^j {b}}_{1,n-j}^2+\sum_{j=0}^{n }\norm{ \dt^j {\hb}}_{n-j+1}^2.
\eeq

\begin{prop}\label{evolution_1}
It holds that
\beq \label{di_ev_n1}
\fdb{2N}\ls \bar{\mathcal{D}}_{2N} +\fe{N+4} \(\se{2N}+\fd{2N}\)
\eeq
and for $n=N+4,\dots,2N-1$,  
\beq \label{di_ev_n2}
\fdb{n}\ls \bar{\mathcal{D}}_{n} + \fd{N+4} \se{2N}.
\eeq
\end{prop}
\begin{proof}
Assume that $n=N+4,\dots,2N$. First, the magnetic part of \eqref{MHDv_perb} yields that
\beq\label{elpp2123}
\begin{cases}
    \curl    b=   \curl    b, \quad \diverge  b= G^4 & \text{in } \Omega_-
 \\\curl  \hb= \hat G^3 ,\quad\diverge  \hb=  \hat G^4 &\text{in }\Omega_+
\\\jump{b}=0  & \text{on } \Sigma
\\  b_3=0 & \text{on } \Sigma_{-}
\\   \hb\times e_3=0  & \text{on } \Sigma_{+}.
\end{cases}
\eeq
It then follows from the Hodge-type estimates \eqref{v_ell_th02} of Proposition \ref{propel2} (setting $\eta=0$) with $r=1$ that for $j=0,\dots,n$,  
\begin{align}\label{deq00}
\norm{ \dt^j {  b}}_{1,n-j}^2+\norm{ \dt^j { \hb}}_{1,n-j}^2   \ls  \norm{ \dt^j \curl{ b}}_{0,n-j}^2+\ns{ \dt^j G^4}_{0,n-j}+\ns{ \dt^j \hat{G}^3}_{0,n-j}+ \ns{ \dt^j \hat{G}^4}_{0,n-j}.
\end{align}
On the other hand, employing the Hodge-type  estimates \eqref{v_ell_th} of Lemma \ref{ell_hodge} with $r= n-j+1\ge 1$ in $\Omega_+$, one deduces that, by   the third and fourth equations in \eqref{elpp2123}, 
\begin{align}\label{deq11}
 \ns{  \dt^j \hb  }_{n-j+1}&\ls \norm{ \dt^j { \hb}}_{0,n-j+1}^2 +\ns{\dt^j \curl \hb  }_{n-j}+ \ns{\dt^j \diverge \hb  }_{n-j}   \nonumber
\\ &=  \norm{ \dt^j { \hb}}_{0,n-j+1}^2 +\ns{\dt^j \hat{G}^3 }_{n-j}+ \ns{\dt^j \hat{G}^4 }_{n-j}   .
\end{align}
It follows from \eqref{deq00} and \eqref{deq11} that
\begin{align}\label{deq}
\fdb{n} &\ls  \sdb{n}+\sum_{j=0}^{n }\ns{ \dt^j G^4}_{n-j}+\sum_{j=0}^{n }\ns{ \dt^j \hat{G}^3}_{n-j}+\sum_{j=0}^{n }\ns{ \dt^j \hat{G}^4}_{n-j}.
\end{align}
By using \eqref{p_G_e_03} when $n=2N$ and \eqref{p_G_e_02} when $n=N+4,\dots, 2N-1$, one obtains \eqref{di_ev_n1} and \eqref{di_ev_n2}, respectively, from \eqref{deq}. 
\end{proof}

\begin{rem}\label{evolution_11}
Note that one can derive the desired boundary regularity of $b$ in the dissipation estimates. Indeed,
it follows from the trace theory that
\beq \label{di_ev_ntr}
\sum_{j=0}^{n }\abs{ \dt^j {b}}_{n-j+1/2}^2\ls \sum_{j=0}^{n }\ns{  \dt^j  b }_{1,n-j}\le\fdb{n}.
\eeq
\end{rem}

%%%%%%%%%%%%%%%%%%%%%%%%%%%%%%%%%%%%%%%%%%%%%%%%%%%%%%
\subsubsection{$\bar B\cdot\nabla$-dissipation estimates of $u$}
%%%%%%%%%%%%%%%%%%%%%%%%%%%%%%%%%%%%%%%%%%%%%%%%%%%%%%

Note that by now the dissipation estimates  only control the magnetic fields $b$ and $\hb$. The dissipation estimates for the velocity $u$ rely on the coupling between the fluid and the magnetic field and $\bar B_3\neq0$, and one first has the following.
\begin{prop}\label{evolution_12}
For $n=N+4,\dots,2N$, it holds that
\beq \label{low_diss2}
  \sum_{j=0}^{n-1}\(\ns{ \dt^j  u  }_{0,n-j-1}+\abs{  \dt^j  u_3}_{n-j-1}^2\) \ls \fdb{n}+ \fd{N+4} \se{2N}.
\eeq
\end{prop}
\begin{proof}
Assume that $n=N+4,\dots,2N$. First, it follows from   the third and fourth equations  in \eqref{MHDv_perb} that
\beq\label{bbequ}
\dt {b} -\kappa \Delta b=  \bar B\cdot \nabla u+G^3+\kappa \nabla G^4\quad  \text{in } \Omega_-
\eeq
By the vertical component of \eqref{bbequ} and the fourth equations  in \eqref{MHDv_perb}, one has
\begin{align}\label{disss1}
\bar B\cdot\nabla   u_3&= \dt b _3  -\kappa \Delta b_3-G^3_3- \kappa \p_3 G^4\nonumber
\\&= \dt b _3 -\kappa\Delta_hb_3+\kappa \pa_3\diverge_h  b_h -G^3_3-2\kappa \pa_3 G^4 .
\end{align}
It then follows from \eqref{disss1} and \eqref{p_G_e_02} that for $j=0,\dots,n-1$,
\begin{align}\label{v_bd_es_3}
\norm{\bar B\cdot \nabla \dt^j  u_3}_{0,n-j-1}^2 &\ls \norm{\dt^{j+1} b_3}_{0,n-j-1}^2+\norm{\dt^{j} b}_{1,n-j}^2+\norm{\dt^{j} G^3_3}_{n-j-1}^2+\norm{\dt^{j}   G^4}_{n-j}^2\nonumber
\\&\ls \fdb{n}+ \fd{N+4} \se{2N} .
\end{align}
This implies,  since $\bar B_3\neq 0$,  by the Poincare-type inequalities \eqref{poincare_1}  and \eqref{poincare_2} of Lemma \ref{lempoi} and the tenth equation  in \eqref{MHDv_perb}, that
\begin{align}\label{v_bd_es_32}
\norm{  \dt^j  u_3}_{0,n-j-1}^2+\abs{  \dt^j  u_3}_{n-j-1}^2\ls \norm{\bar B\cdot \nabla \dt^j  u_3}_{0,n-j-1}^2 \ls \fdb{n}+ \fd{N+4} \se{2N}.
\end{align}

Next,  the tenth to twelfth equations in \eqref{MHDv_perb}  imply
 \beq\label{bbdd2}
 \kappa\p_3b_h+\bar B_3 u_h =  G^{7}_h \quad\text{on }\Sigma_-.
 \eeq
This motivates one to consider the quantity $ \kappa\p_3b_h+\bar B_3 u_h$.
 It then follows from the horizontal components of the third equation  in \eqref{MHDv_perb} that
\begin{align}\label{disss122}
&\bar B\cdot \nabla (\kappa\p_3b_h+\bar B_3 u_h)\equiv\bar B_h\cdot \nabla_h (\kappa\p_3b_h)+\bar B_3  (\kappa\p_3^2b_h +\bar B\cdot \nabla  u_h)\nonumber
\\&\quad=\bar B_h\cdot \nabla_h (\kappa\p_3b_h)+\bar B_3  (-\kappa\Delta_hb_h -\dt b_h+G^3_h+\kappa \nabla_h G^4).
\end{align}
\eqref{disss122} and \eqref{p_G_e_02} imply that for $j=0,\dots,n-1$,
\begin{align}\label{vb_es_di6}
&\ns{\bar B\cdot \nabla \dt^j(\kappa\p_3b_h+\bar B_3 u_h)}_{0,n-j-1}\nonumber
\\&\quad\ls\ns{\dt^jb_h}_{1,n-j}+\ns{\dt^{j+1} b_h}_{0,n-j-1}+\ns{\dt^jG^3_h}_{0,n-j-1}+\ns{\dt^jG^4}_{0,n-j} \nonumber
\\&\quad \ls   \fdb{n} + \fd{N+4} \se{2N}.
\end{align}
By  \eqref{poincare_1}  and \eqref{poincare_2}  again,   it follows from \eqref{vb_es_di6},   \eqref{bbdd2}  and \eqref{p_G_e_02} that
\begin{align}\label{vb_es_di62}
&\ns{ \dt^j(\kappa\p_3b_{h}+\bar B_3 u_{h})}_{0,n-j-1}+\as{ \dt^j(\kappa\p_3b_{h}+\bar B_3 u_{h})}_{n-j-1}\nonumber
\\&\quad \ls \ns{\bar B\cdot \nabla \dt^j(\kappa\p_3b_{h}+\bar B_3 u_{h})}_{0,n-j-1}+\as{ \dt^j G^7_h}_{n-j-1} 
\nonumber
\\&\quad\ls   \fdb{n} + \fd{N+4} \se{2N}.
\end{align}
Hence,   by \eqref{vb_es_di62} and since  $\bar B_3\neq 0$ again, one has
\begin{align}\label{vb_es_di64}
\ns{ \dt^j  u_h }_{0,n-j-1}\ls  \ns{ \dt^j \p_3b_h }_{0,n-j-1}+\ns{ \dt^j(\kappa\p_3b_h+\bar B_3 u_h)}_{0,n-j-1}  \ls   \fdb{n} +\fd{N+4} \se{2N}.
\end{align}

Consequently, collecting the estimates \eqref{v_bd_es_3}, \eqref{v_bd_es_32} and \eqref{vb_es_di64} yields \eqref{low_diss2}.
\end{proof}

%%%%%%%%%%%%%%%%%%%%%%%%%%%%%%%%%%%%%%%%%%%%%%%%%%%%%%
\subsection{Estimates of $u$, $b$ and $\hb$}
%%%%%%%%%%%%%%%%%%%%%%%%%%%%%%%%%%%%%%%%%%%%%%%%%%%%%%

In this subsection we will complete the estimates of the velocity $u$ and the magnetic field $b$.

%%%%%%%%%%%%%%%%%%%%%%%%%%%%%%%%%%%%%%%%%%%%%%%%%%%%%%
\subsubsection{Estimates of $u$, $b$ and $\hb$ at the $2N$ level}\label{sec_53}
%%%%%%%%%%%%%%%%%%%%%%%%%%%%%%%%%%%%%%%%%%%%%%%%%%%%%%

We first derive the normal estimates of  $u$, $b$ and $\hb$ at the $2N$ level.
\begin{prop}\label{v_b_prop}
It holds that
\begin{align}\label{v_b_es}
&  \dtt  \ns{  (\curlv{u})_h }_{2N-1 } +\ns{  (\curlv{u})_h }_{2N-1 } +\sum_{j=0}^{2N}\norm{ \dt^j {u}}_{2N-j}^2+\sum_{j=0}^{2N}\norm{\dt^j b}_{2N-j+1}^2+\sum_{j=0}^{2N}\norm{\dt^j \hb}_{2N-j+1}^2\nonumber
\\&\quad\ls  \seb{2N}+\fdb{2N}+\fe{N+4}   \se{ 2N} 
\end{align}
and that
\begin{align}\label{vb_es_en}
&\sum_{j=0}^{2N} \norm{\dt^{j} u}_{2N-j}^2+	\sum_{j=0}^{2N-1}\norm{\dt^j b}_{2N-j+1 }^2+\norm{\dt^{2N} b}_{0}^2  +	\sum_{j=0}^{2N-1}\norm{\dt^j \hb}_{2N-j+1 }^2+\norm{\dt^{2N} \hb}_{0}^2  \nonumber
\\&\quad\ls   \seb{2N}+\ns{ (\curlv{u})_h }_{2N-1} +  \fe{N+4} \se{2N}.
\end{align}
\end{prop}
\begin{proof}
Fix $\ell=0,\dots, 2N-1$. Let $\al\in \mathbb{N}^{3}$ with $\abs{\al}\le 2N-1$ such that $\al_3\le  2N-1-\ell$. Taking the inner product of the equations \eqref{MHDcurlva} with $\pa^\al (\curlv{u})_h $ and integrating by parts over $\Omega_-$, by using  the second, eighth and eleventh equations in \eqref{MHDv}, one obtains
\begin{align}\label{eeqq1}
&\hal \dtt\int_{\Omega_-} \abs{\pa^\al (\curlv{u})_h}^2\,d\V  +\frac{\bar{B}_3^2}{\kappa}\int_{\Omega_-} \abs{\pa^\al (\curlv{u})_h}^2\,d\V\nonumber
\\&\quad\ls  \(\norm{\pa^\al L_h}_0^2+\norm{\Phi^{\al}_h }_0\)\norm{\pa^\al (\curlv{u})_h}_0  .
\end{align}
It then follows from \eqref{eeqq1}, Cauchy's inequality, \eqref{lheq}, \eqref{p_Phi_e_1} and \eqref{curl_e_1} that 
\begin{align}\label{v_b_es_1}
 &\dtt \ns{\pa^\al (\curlv{u})_h}_0 + \ns{\pa^\al (\curlv{u})_h}_0 + \ns{\pa^\al (\curl{u})_h}_0 \nonumber
\\ &\quad \ls  \norm{\pa^\al L_h}_0^2+\ns{\Phi^{\al}_h }_0+\ns{\pa^\al \(\curlv u-\curl  u\) _h}_0\nonumber
\\& \quad\ls  \norm{\pa^\al u }_{0,1}^2+\norm{\pa^\al b }_{1,1}^2+\norm{\pa^\al \dt b }_{0}^2 +\fe{N+4}  \se{ 2N}   .
\end{align}
Summing \eqref{v_b_es_1} over such $\al$ yields
\begin{align}\label{v_b_es_2}
&  \dtt  \ns{   (\curlv{u})_h }_{ 2N-1-\ell,\ell} +\ns{   (\curlv{u})_h }_{ 2N-1-\ell,\ell} +\ns{   (\curl {u})_h }_{ 2N-1-\ell,\ell}\nonumber
\\&\quad\ls  \norm{  {u}}_{ 2N-1-\ell,\ell+1}^2+\norm{  {b}}_{ 2N -\ell,\ell+1}^2+\ns{\dt  {b}}_{ 2N-1} +\fe{N+4}  \se{ 2N}     .
\end{align}
On the other hand, employing the Hodge-type  estimates \eqref{v_ell_th} of Lemma \ref{ell_hodge} with $r= 2N-\ell\ge 1$ and using the second equation in \eqref{MHDv_perb} and  \eqref{p_G_e_02},  one obtains
\begin{align}\label{v_b_es_4}
  \ns{    {u} }_{ 2N-\ell,\ell} \ls & \ns{     {u} }_{0,2N-\ell+\ell}+\ns{   (\curl {u})_h  }_{ 2N-1-\ell,\ell} +\ns{   \diverge {u} }_{ 2N-1-\ell,\ell}   \nonumber
\\\le & \ns{     {u} }_{0,2N}  + \ns{   (\curl {u})_h  }_{ 2N-1-\ell,\ell}  +\ns{ G^2 }_{ 2N-1 }  
\nonumber
\\\ls & \ns{     {u} }_{0,2N}  + \ns{   (\curl {u})_h  }_{ 2N-1-\ell,\ell}  +\fe{N+4}  \se{ 2N}   .
\end{align}
Then one deduces from \eqref{v_b_es_2} and \eqref{v_b_es_4} that
\begin{align}\label{v_b_es_45}
&  \dtt  \ns{   (\curlv{u})_h }_{ 2N-1-\ell,\ell} +  \ns{   (\curlv{u})_h }_{ 2N-1-\ell,\ell} +  \ns{    {u} }_{ 2N-\ell,\ell}\nonumber
\\&\quad\ls  \norm{  {u}}_{ 2N-1-\ell,\ell+1}^2+\norm{  {b}}_{ 2N -\ell,\ell+1}^2+\ns{\dt  {b}}_{ 2N-1} +\fe{N+4} \se{ 2N}   .
\end{align}
Noting \eqref{di_ev_ntr} in Remark \ref{evolution_11}, we consider the following elliptic problem (\eqref{bbequ}):
\beq \label{bj_eq}
\begin{cases}
-\kappa \Delta   b=\bar{B} \cdot \nabla   {u}-  \dt  {b}+ G^3+\kappa \nabla G^4& \text{in } \Omega_-
\\   b=  b & \text{on }   \Sigma\cup\Sigma_-.
\end{cases}
\eeq
It then follows from the standard $H^r$ elliptic estimates  with $r=2N-\ell+1\ge 2$,  \eqref{p_G_e_02}  and \eqref{di_ev_ntr} with $n=2N$ that
\begin{align}\label{v_b_es_4b}
\ns{   {b} }_{2N -\ell+1,\ell } & \ls \norm{  {u}}_{2N  -\ell,\ell }^2+  \norm{\dt  b}_{2N -\ell-1,\ell } ^2+ \norm{ G^3}_{2N -\ell-1,\ell }^2+\norm{ G^4}_{2N -\ell,\ell }^2+\as{   {b} }_{2N  +1/2}  \nonumber
\\& \ls \norm{  {u}}_{2N -\ell,\ell}^2+  \norm{\dt  b}_{2N-1}^2+\fdb{2N}  +\fe{N+4}  \se{ 2N}   .
\end{align}
Then it follows from \eqref{v_b_es_45} and \eqref{v_b_es_4b} that
\begin{align}\label{v_b_es_42}
&  \dtt  \ns{  (\curlv{u})_h }_{2N-1-\ell,\ell} +\ns{  (\curlv{u})_h }_{2N-1-\ell,\ell}+\ns{   {u} }_{2N-\ell,\ell} +\ns{   {b} }_{2N -\ell+1,\ell }\nonumber
\\&\quad\ls \norm{ {u}}_{2N-1-\ell,\ell+1}^2 +\norm{  {b}}_{ 2N -\ell,\ell+1}^2  +\ns{\dt  {b}}_{2N-1} +\fdb{2N} +\fe{N+4} \se{ 2N}   .
\end{align}
A suitable linear combination of   \eqref{v_b_es_42}  for  $\ell=0,\dots, 2N-1$ yields that, recalling the  conventional notation \eqref{ccon} and  the definition of $\fdb{2N}$, 
\begin{align}\label{v_b_es_43}
&  \dtt  \ns{  (\curlv{u})_h }_{2N-1 } + \ns{  (\curlv{u})_h }_{2N-1 } +\ns{   {u}}_{2N}+\ns{   {b} }_{2N +1 }\nonumber
\\&\quad\ls  \norm{ {u}}_{0,2N}^2   +\norm{  {b}}_{ 1,2N}^2  +\ns{\dt  {b}}_{2N-1} +\fdb{2N}+\fe{N+4}  \se{ 2N}  
\nonumber
\\&\quad\ls  \norm{ {u}}_{0,2N}^2    +\ns{\dt  {b}}_{2N-1} +\fdb{2N}+ \fe{N+4} \se{ 2N}  .
\end{align}

Next,  applying $\curl$ to the first equation in \eqref{MHDv_perb} yields
\beq\label{vor1}
 \dt \curl {u}    = \bar{B}   \cdot \nabla  \curl  {b} + \curl G^1.
\eeq
For $j=1,\dots,2N-1$, employing the Hodge-type  estimates \eqref{v_ell_th} of Lemma \ref{ell_hodge} with $r= 2N-j\ge 1$, by \eqref{vor1},  the second equation in \eqref{MHDv_perb} and \eqref{p_G_e_02},  one obtains
\begin{align}\label{v_b_es_56}
\norm{ \dt^j   {u} }_{2N-j}^2 &\ls   \norm{ \dt^j   {u} }_{0,2N-j}^2+\norm{ \dt^j  (\curl {u})_h  }_{2N-j-1}^2+  \ns{   \dt^j  \diverge {u}}_{2N-j-1 } \nonumber\\
&\ls   \norm{ \dt^j   {u} }_{0,2N-j}^2+ \norm{\dt^{j-1} b}_{2N-j+1}^2+ \norm{\dt^{j-1} G^1}_{2N-j}^2  +  \ns{   \dt^j  G^2}_{2N-j-1 } \nonumber
\\ &\ls   \norm{ \dt^j   {u} }_{0,2N-j}^2+  \norm{\dt^{j-1} b}_{2N-(j-1)}^2+ \fe{N+4} \se{ 2N} .
\end{align}
On the other hand, applying $\dt^j$, $j=1,\dots,2N-1$, to  the problem \eqref{bj_eq} and the standard $H^r$ elliptic estimates with $r=2N-j+1\ge 2$,  \eqref{p_G_e_02} and \eqref{di_ev_ntr} with $n=2N$ show that
\begin{align}\label{v_b_es_6}
\norm{\dt^j b}_{2N-j+1}^2& \ls \norm{ \dt^j {u}}_{2N-j}^2+  \norm{\dt^{j+1} b}_{2N-j-1}^2+ \norm{\dt^jG^3}_{2N-j-1}^2+ \as{\dt^j b}_{2N-j+1/2}   \nonumber
\\& \ls \norm{ \dt^j {u}}_{2N-j}^2+  \norm{\dt^{j+1} b}_{2N-(j+1)}^2+ \fdb{2N}+ \fe{N+4} \se{ 2N} \ .
\end{align}
Combining \eqref{v_b_es_56} and \eqref{v_b_es_6}  and then summing  over $j=1,\dots,2N-1$ yield that
\begin{align}\label{v_b_es_6112233}
&\sum_{j=1}^{2N-1}\norm{ \dt^j {u}}_{2N-j}^2+\sum_{j=1}^{2N-1}\norm{\dt^j b}_{2N-j+1}^2 \nonumber
\\ &\quad\ls \sum_{j=1}^{2N-1}\norm{ \dt^j   {u} }_{0,2N-j}^2+ \norm{ b}_{2N}^2 +  \sum_{j=2}^{2N} \norm{\dt^{j} b}_{2N-j}^2  +   \fdb{2N}+ \fe{N+4} \se{ 2N}  .
\end{align}

Now combining \eqref{v_b_es_43} and \eqref{v_b_es_6112233} leads to
\begin{align}
&  \dtt  \ns{  (\curlv{u})_h }_{2N-1 } +\ns{  (\curlv{u})_h }_{2N-1 } +\sum_{j=0}^{2N-1}\norm{ \dt^j {u}}_{2N-j}^2+\sum_{j=0}^{2N-1}\norm{\dt^j b}_{2N-j+1}^2 \nonumber
\\&\quad\ls  \sum_{j=0}^{2N-1}\norm{ \dt^j   {u} }_{0,2N-j}^2+   \sum_{j=0}^{2N} \norm{\dt^{j} b}_{2N-j}^2  +   \fdb{2N}+ \fe{N+4} \se{ 2N}  .
\end{align}
This together with the Sobolev interpolation implies that
\begin{align}\label{v_b_es_412321}
&  \dtt  \ns{  (\curlv{u})_h }_{2N-1 } +\ns{  (\curlv{u})_h }_{2N-1 } +\sum_{j=0}^{2N-1}\norm{ \dt^j {u}}_{2N-j}^2+\sum_{j=0}^{2N-1}\norm{\dt^j b}_{2N-j+1}^2 \nonumber
\\&\quad\ls  \sum_{j=0}^{2N-1}\norm{ \dt^j   {u} }_{0,2N-j}^2+    \sum_{j=0}^{2N} \norm{\dt^{j} b}_{0}^2  +   \fdb{2N}+ \fe{N+4}  \se{ 2N}   \nonumber
\\&\quad\ls  \sum_{j=0}^{2N-1}\norm{ \dt^j   {u} }_{0,2N-j}^2+\fdb{2N}+ \fe{N+4}  \se{ 2N}  .
\end{align}
Finally, since $ \sum_{j=0}^{2N-1}\norm{ \dt^j   {u} }_{0,2N-j}^2\ls \seb{2N}$, \eqref{v_b_es_412321} yields the estimate \eqref{v_b_es}  as   $\ns{\dt^{2N}u}_0\le \seb{2N}$ and $\ns{\dt^{2N}b}_1+\sum_{j=0}^{2N}\norm{\dt^j \hb}_{2N-j+1}^2 \le \fdb{2N}$.

We now prove \eqref{vb_es_en}. First, one recalls from \eqref{v_b_es_4} with $\ell=0$ that
\beq \label{vb_es_en2}
\ns{u}_{2N}  \ls  \ns{u}_{0,2N}  +\ns{ (\curlv{u})_h }_{2N-1} +\fe{N+4} \se{2N}
\eeq
and from \eqref{v_b_es_56} that  for $j=1,\dots,2N-1$, 
\beq\label{vb_es_en22}
\norm{ \dt^j   {u} }_{2N-j}^2 \ls   \norm{ \dt^j   {u} }_{0,2N-j}^2+  \norm{\dt^{j-1} b}_{2N-(j-1)}^2+ \fe{N+4} \se{2N} .
\eeq
Now consider  the following two-phase elliptic problem, which follows from \eqref{MHDv_perb},
\beq \label{bj_eq2} 
\begin{cases}
 \kappa \curl\curl  b=\bar{B}   \cdot \nabla {u}-\dt {b} +G^3& \text{in } \Omega_-
\\ \diverge b= G^4 & \text{in } \Omega_-
   \\\curl \hb=  \hat{G}^3,\quad\diverge \hb=  \hat{G}^4 &\text{in }\Omega_+
\\\jump{b}=0  & \text{on } \Sigma
\\  b_3=0,\quad   {\kappa\curl b }\times e_3=\(u\times \bar B\)\times e_3+G^7 & \text{on } \Sigma_{-}
\\   b\times e_3=0  & \text{on } \Sigma_{+}.
\end{cases}
\eeq
Applying $\dt^j$, $j=0,\dots,2N-1$ to   \eqref{bj_eq2} and using the Hodge-type estimates \eqref{v_ell_th03} of Proposition \ref{propel3} (with $\eta=0$)   with $r=2N-j+1\ge 2$, \eqref{p_G_e_02} and the trace theory, one can get
\begin{align}\label{vb_es_en5}
&\norm{\dt^j b}_{2N-j+1 }^2+\norm{\dt^j \hb}_{2N-j+1 }^2\nonumber
\\&\quad \ls   \norm{\dt^{j} u}_{2N-j}^2 + \norm{\dt^{j+1} b}_{2N-j-1}^2+\norm{\dt^{j} G^3}_{2N-j-1}^2+\norm{\dt^{j} G^4}_{2N-j}^2
\nonumber
\\&\qquad +\norm{\dt^{j} \hat{G}^3}_{2N-j}^2+\norm{\dt^{j} \hat{G}^4}_{2N-j}^2
+\abs{ \dt^j {u}}_{{2N-j-1/2}}^2+\abs{ \dt^j {G^7}}_{{2N-j-1/2}}^2 \nonumber
\\&\quad \ls \norm{\dt^{j} u}_{2N-j}^2+ \norm{\dt^{j+1} b}_{2N-(j+1)}^2+ \fe{N+4} \se{2N}.
\end{align}
Hence, collecting \eqref{vb_es_en2}, \eqref{vb_es_en22} for $j=1,\dots,2N-1$ and \eqref{vb_es_en5} for $j=0,\dots,2N-1$ leads to
\begin{align}   \label{vb_es_en6} 
&\sum_{j=0}^{2N-1}\norm{ \dt^j {u}}_{2N-j}^2+\sum_{j=0}^{2N-1}\norm{\dt^j b}_{2N-j+1}^2+\sum_{j=0}^{2N-1}\norm{\dt^j \hb}_{2N-j+1 }^2\nonumber
\\ &\quad\ls \sum_{j=0}^{2N-1}\norm{ \dt^j   {u} }_{0,2N-j}^2+   \sum_{j=0}^{2N} \norm{\dt^{j} b}_{2N-j}^2  +\ns{ (\curlv{u})_h }_{2N-1}+ \fe{N+4} \se{ 2N}
\nonumber
\\ &\quad\ls \seb{2N}+   \sum_{j=0}^{2N} \norm{\dt^{j} b}_{2N-j}^2  + \ns{ (\curlv{u})_h }_{2N-1}+ \fe{N+4} \se{ 2N}  .
\end{align}
This together with the Sobolev interpolation implies that
\begin{align}   \label{vb_es_en7}
&\sum_{j=0}^{2N-1} \norm{\dt^{j} u}_{2N-j}^2+	\sum_{j=0}^{2N-1}\norm{\dt^j b}_{2N-j+1 }^2 +\sum_{j=0}^{2N-1}\norm{\dt^j \hb}_{2N-j+1 }^2\nonumber
\\&\quad\ls   \seb{2N}+ \sum_{j=0}^{2N}\norm{\dt^{j} b}_{0}^2+\ns{ (\curlv{u})_h }_{2N-1} +  \fe{N+4} \se{2N}
\nonumber \\&  \quad\ls   \seb{2N}+\ns{ (\curlv{u})_h }_{2N-1} +  \fe{N+4} \se{2N}.
\end{align}
This yields the estimate \eqref{vb_es_en} as   $\ns{\dt^{2N}u}_0+\ns{\dt^{2N}b}_0+\ns{\dt^{2N}\hb}_0\le \seb{2N}$.
\end{proof}

%%%%%%%%%%%%%%%%%%%%%%%%%%%%%%%%%%%%%%%%%%%%%%%%%%%%%%
\subsubsection{Estimates at $N+4,\cdots,2N$ levels}
%%%%%%%%%%%%%%%%%%%%%%%%%%%%%%%%%%%%%%%%%%%%%%%%%%%%%%

We now derive the following energy-dissipation estimates of $u$, $b$ and $\hb$ at $N+4,\cdots,2N$ levels.
\begin{prop}  \label{prop_vb_di}
For $n=N+4,\dots, 2N$, it holds that  
\begin{align}\label{vb_es_di}
   &\dtt  \ns{  (\curlv{u})_h }_{n-2} +\sum_{j=0}^{n-1}\norm{ \dt^j {u}}_{n-j-1}^2+\sum_{j=0}^{n-2}\norm{\dt^j b}_{n-j}^2+\sum_{j=0}^{n}\norm{\dt^j b}_{1,n-j}^2+\sum_{j=0}^{n}\norm{\dt^j \hb}_{n-j+1}^2\nonumber
   \\&\quad  \ls  \fdb{n}+\fd{N+4} \se{2N}
\end{align}
and that
\begin{align}\label{vb_es_en_n+2}
& \norm{   u}_{n-1}^2  +\norm{  u}_{0,n}^2+ \sum_{j=1}^{n}\norm{\dt^j u}_{n-j}^2+ \norm{   b}_{n}^2  +\sum_{j=1}^{n-1}\norm{\dt^j b}_{n-j+1}^2+\norm{\dt^{n}b}_{0}^2\nonumber
\\&\quad + \norm{   \hb}_{n}^2  +\sum_{j=1}^{n-1}\norm{\dt^j \hb}_{n-j+1}^2+\norm{\dt^{n}\hb}_{0}^2\ls  \bar{\mathcal{E}}_n+ \ns{ (\curlv{u})_h }_{n-2} +\fe{N+4} \se{2N}.
\end{align}
\end{prop}
\begin{proof}
Assume that $n=N+4,\dots, 2N$.  It follows  similarly as the derivation of \eqref{v_b_es_412321} in the proof of Proposition \ref{v_b_prop}, with $2N$ replaced by $n-1$ in \eqref{v_b_es_43}, \eqref{v_b_es_56} and \eqref{v_b_es_6}, that
\begin{align}\label{vb_es_di2}
&  \dtt  \ns{  (\curlv{u})_h }_{n-2} +\sum_{j=0}^{n-2}\norm{ \dt^j {u}}_{n-j-1}^2+\sum_{j=0}^{n-2}\norm{\dt^j b}_{n-j}^2 \nonumber
\\&\quad\ls  \sum_{j=0}^{n-2}\norm{ \dt^j   {u} }_{0,n-j-1}^2+\fdb{n-1}+\fd{N+4} \se{2N}.
\end{align}
Different from the derivation of \eqref{v_b_es}, here one can estimate the term $\sum_{j=0}^{n-2}\norm{ \dt^j   {u} }_{0,n-j-1}^2 $ in the right hand side of  \eqref{vb_es_di2} by the dissipation rather than the energy. One thus appeals to \eqref{low_diss2} to obtain the estimate \eqref{vb_es_di}   from \eqref{vb_es_di2} by the definition of $\fdb{n}$.

We now prove \eqref{vb_es_en_n+2}. It follows similarly as the derivation of  \eqref{vb_es_en7}, with $2N$ replaced by $n-1$ in \eqref{vb_es_en2} and  \eqref{vb_es_en5} with $j=0$, and $2N$ replaced by $n$ in \eqref{vb_es_en22}  and \eqref{vb_es_en5} with $j=1,\dots,n-1$, that
\begin{align}   \label{vb_es_en77}
&\norm{   u}_{n-1}^2  +\sum_{j=1}^{n-1} \norm{\dt^{j} u}_{n-j}^2+	\norm{b}_{n}^2  +\sum_{j=1}^{n-1}\norm{\dt^j b}_{n-j+1 }^2+	\norm{\hb}_{n}^2  +\sum_{j=1}^{n-1}\norm{\dt^j \hb}_{n-j+1 }^2  \nonumber
\\&  \quad\ls   \seb{n}+\ns{ (\curlv{u})_h }_{n-2} +  \fe{N+4} \se{2N}.
\end{align}
This yields the estimate \eqref{vb_es_en_n+2} by noting that $   \norm{  u}_{0,n}^2+ \norm{\dt^{n} u}_{0}^2+ \norm{\dt^n b}_{0 }^2  + \norm{\dt^n \hb}_{0 }^2  \ls   \seb{n}$.
\end{proof}

%%%%%%%%%%%%%%%%%%%%%%%%%%%%%%%%%%%%%%%%%%%%%%%%%%%%%% 	
\subsection{Estimates of $p$ and $\eta$}
%%%%%%%%%%%%%%%%%%%%%%%%%%%%%%%%%%%%%%%%%%%%%%%%%%%%%%

In this subsection  we shall complete the estimates on the pressure $p$ and the free interface function $\eta$.	

%%%%%%%%%%%%%%%%%%%%%%%%%%%%%%%%%%%%%%%%%%%%%%%%%%%%%%
\subsubsection{Energy}
%%%%%%%%%%%%%%%%%%%%%%%%%%%%%%%%%%%%%%%%%%%%%%%%%%%%%%

We begin with the estimates in the energy.
\begin{prop}\label{prop_qh_en}
It holds that for $n=N+4,\dots, 2N$,
\begin{align}\label{qh_en_es}
&\sum_{j=0}^{n-1}\ns{  \dt^j{p}}_{n-j} +\sum_{j=0}^{n-1}\abs{  \dt^j \eta }_{n-j+3/2}^2+\abs{\dt^{n}\eta }_1^2  +  \abs{\dt^{n+1}\eta }_{-1/2}^2  \nonumber
\\&\quad\ls \seb{n}+\sum_{j=1}^n\norm{  \dt^{j}{u}}_{n-j}^2+\sum_{j=0}^{n-1}\norm{  \dt^j{b}}_{n-j}^2+\fe{N+4} \se{2N} .
\end{align}
\end{prop}
\begin{proof}
Assume that $n=N+4,\dots,2N$. Let $\x_n$ denote the two sums in the right hand side of \eqref{qh_en_es}. It follows from the first equation  in \eqref{MHDv_perb} and \eqref{p_G_e_02} that for $j=0,\dots,n-1$,
\beq\label{q_en_2}
\norm{  \nabla \dt^j{p}}_{n-j-1}^2    \ls  \norm{  \dt^{j+1}{u}}_{n-j-1}^2+\norm{  \dt^j{b}}_{n-j}^2+ \norm{    \dt^j G^1}_{n-j-1}^2  \ls    \x_n+ \fe{N+4} \se{2N} .
\eeq
By the eighth  equation in \eqref{MHDv_perb} and \eqref{p_G_e_02}, one obtains that for $j=0,\dots,n-1$,
\beq\label{q_en_222}
\abs{ { \dt^j{p}}}_{0}^2\ls   \abs{  \dt^j \eta }_{2}^2 + \abs{  \dt^j G^{6}}_{0}^2     \ls \seb{n}  + \fe{N+4} \se{2N} .
\eeq
It then follows from \eqref{q_en_2} and \eqref{q_en_222} that for $j=0,\dots,n-1$, by Poincar\'e's inequality,
\beq\label{q_en_3}
\ns{  \dt^j{p}}_{n-j}  \ls  \ns{ \nabla \dt^j{p}}_{n-j-1} + \as{ { \dt^j{p}}}_{0}    \ls  \seb{n}  +\x_n+\fe{N+4} \se{2N}.
\eeq

Now, we improve the estimates of $\eta$ by using the estimates of $p$ derived in \eqref{q_en_3}. The standard elliptic theory on the eighth equation in \eqref{MHDv_perb} yields
 that for $j=0,\dots,n-1$, by  the trace theory, \eqref{p_G_e_02} and \eqref{q_en_3},
\begin{align} \label{q_en_4}
\abs{  \dt^{j}{\eta}}_{n-j+3/2}^2  & \ls \abs{  \dt^{j}{\eta}}_{0}^2  +  \abs{  \dt^j{p}}_{n-j-1/2}^2   +  \as{  \dt^j{G^{6}}}_{n-j-1/2}\nonumber
\\&\ls   \seb{n}   +\ns{  \dt^j{p}}_{n-j}   + \fe{N+4} \se{2N} \ls \seb{n}  +\x_n + \fe{N+4} \se{2N}.
\end{align}

Finally, using the normal trace estimate \eqref{normal_es_3} of Lemma \ref{Le_normal}, by the second equation in \eqref{MHDva} and \eqref{p_F_e_01}, one obtains 
\begin{align}\label{v_bd_es_2}
\abs{ \dt^{n}   u\cdot \N}_{-1/2} \ls \ns{ \dt^{n}  u  }_{0}+\ns{ \diva \dt^{n}  u  }_{0} \ls  \seb{n} +\ns{F^{2,(n,0)}}_0\ls \seb{n}+ \fe{N+4}\se{2N} .
\end{align}
It then follows from the seventh equation in \eqref{MHDva}, \eqref{v_bd_es_2} and \eqref{p_F_e_01} that
\beq \label{q_en_5}
\abs{  \dt^{n+1}\eta }_{-1/2}^2   \ls  	 \abs{  \dt^{n}{u}\cdot \N}_{-1/2}^2   +  \as{ F^{5,(n,0)}}_{ -1/2}\ls \seb{n}   + \fe{N+4} \se{2N}.
\eeq

Hence, summing \eqref{q_en_3}, \eqref{q_en_4} over $j=0,\dots,n-1$  and \eqref{q_en_5} yields the estimate \eqref{qh_en_es} due to that $\abs{\dt^{n}\eta }_1^2 \le\seb{n}$. 
\end{proof}

%%%%%%%%%%%%%%%%%%%%%%%%%%%%%%%%%%%%%%%%%%%%%%%%%%%%%%
\subsubsection{Dissipation}
%%%%%%%%%%%%%%%%%%%%%%%%%%%%%%%%%%%%%%%%%%%%%%%%%%%%%%

Now we consider the estimates in the dissipation.
 \begin{prop}\label{prop_qh_diss}
For $n=N+4,\dots, 2N$, it holds that
\begin{align}\label{q_es_diss}
&\sum_{j=0}^{n-2}\norm{  \dt^{j} p }_{n-j-1}^2+\sum_{j=0}^{n-2}\abs{  \dt^{j} \eta }_{n-j+1/2}^2+\abs{\dt^{n-1}\eta }_{1}^2 +\abs{\dt^{n}\eta }_{0}^2\nonumber
\\&\quad\ls \fdb{n} +\sum_{j=1}^{n-1}\norm{  \dt^{j}{u}}_{n-j-1}^2+\sum_{j=0}^{n-2}\norm{  \dt^j{b}}_{n-j-1}^2+\fd{N+4}\se{2N} .
\end{align}
\end{prop}
\begin{proof}
Assume that $n=N+4,\dots,2N$. Let $\y_n$ denote the two sums in the right hand side of \eqref{q_es_diss}. It follows from the first equation  in \eqref{MHDv_perb} and \eqref{p_G_e_02} that for $j=0,\dots,n-2$,
 \beq\label{q_diss_1}
\norm{  \nabla \dt^j{p}}_{n-j-2}^2   \ls  \norm{  \dt^{j+1}{u}}_{n-j-2}^2+\norm{  \dt^j{b}}_{n-j-1}^2+ \norm{    \dt^j G^1}_{n-j-2}^2 \ls    \y_n  +\fd{N+4}\se{2N} .
\eeq
 To estimate ${\dt^{j} p} $ on the interface $\Sigma$, one may estimate $\eta$ first. Indeed,  by  the seventh equation in \eqref{MHDv_perb},   \eqref{low_diss2} and \eqref{p_G_e_02}, one obtains that for $j=1,\dots,n$,
 \begin{align}\label{h_diss}
 \abs{\dt^j  \eta }_{n-j}^2  &\ls \abs{\dt^{j-1} u_3}_{n-j}^2    +\as{\dt^{j-1}G^{5} }_{n-j} \equiv  \abs{\dt^{j-1} u_3}_{n-(j-1)-1}^2  +\abs{\dt^{j-1}G^{5} }_{n-(j-1)-1}^2 \nonumber
 \\& \ls \fdb{n} +\fd{N+4}\se{2N} .
 \end{align}
Hence by the eighth equation in \eqref{MHDv_perb}, \eqref{h_diss} and \eqref{p_G_e_02},
one has that for $j=1,\dots,n-2$,  
\beq\label{q_diss_2}
 \abs{{\dt^{j} p} }_{0}^2 \ls \abs{\dt^{j} \eta }_{2}^2  +\abs{\dt^{j} G^{6}}_{0}^2 \ls \fdb{n} +\fd{N+4}\se{2N} .
 \eeq
For $\as{ {p}}_0$,  a different argument is needed since one has not controlled $\as{\eta}_0$ yet. Note that by the trace theory, the estimate \eqref{q_diss_1} with $j=0$ implies in particular that
 \beq\label{q_diss_2o}
 \abs{\nabla_h p }_{n-5/2}^2 \ls \norm{\nabla_h p }_{n-2}^2    \ls \norm{  \dt {u}}_{n-2}^2+\norm{  {b}}_{n-1}^2  +\fd{N+4}\se{2N} .
 \eeq
 Then it follows from the eighth equation in \eqref{MHDv_perb}, \eqref{q_diss_2o} and \eqref{p_G_e_02} that
 \beq\label{q_diss_2o11}
 \abs{\nabla_h  \eta }_{n-1/2}^2 \ls \abs{\nabla_h {p} }_{n-5/2}^2+\abs{\nabla_h  G^{6}}_{n-5/2}^2 \ls  \fdb{n} +\y_n  +\fd{N+4}\se{2N} .
 \eeq
 Since $\int_{\Sigma } \eta =0$, so Poincare's inequality and \eqref{q_diss_2o11} yield
 \beq \label{h_diss_2}
 \abs{\eta}_{n+1/2}^2 \ls \abs{\nabla_h  \eta }_{n-1/2}^2\ls   \fdb{n} +\y_n  +\fd{N+4}\se{2N} ,
 \eeq
which in turn implies that, by  the eighth equation in \eqref{MHDv_perb} and \eqref{p_G_e_02} again,
 \beq \label{q_diss_3}
 \abs{{p}}_{0}^2 \ls \abs{\eta}_{2}^2+\abs{ G^{6}}_{0}^2\ls   \fdb{n} +\y_n  +\fd{N+4}\se{2N} .
 \eeq

Now by Poincar\'e's inequality \eqref{poincare_1}, one deduces from \eqref{q_diss_1},   \eqref{q_diss_2} and \eqref{q_diss_3} that for $j=0,\dots, n-2$,
 \beq  \label{q_diss_4}
 \ns{  \dt^{j} p }_{n-j-1} \ls  \ns{ \nabla \dt^{j} p }_{n-j-2}  +\as{  { \dt^{j} p} }_{0}  \ls \fdb{n} +\y_n+\fd{N+4}\se{2N} .
 \eeq
 This in turn, by the trace theory, the eighth equation in \eqref{MHDv_perb} and \eqref{p_G_e_02}, improves the estimates of $\dt^j \eta $ so that for $j=1,\dots, n-2$, it holds that
 \begin{align} \label{q_diss_5}
 \abs{  \dt^{j} \eta }_{n-j+1/2}^2 &\ls \abs{  \dt^{j} p }_{n-j-3/2}^2+\abs{  \dt^{j} G^{6}}_{n-j-3/2}^2\nonumber
 \nonumber\\&\ls\norm{  \dt^{j} p }_{n-j-1}^2+\fd{N+4}\se{2N} \ls \fdb{n} +\y_n  +\fd{N+4}\se{2N} .
 \end{align}

 Consequently, collecting \eqref{q_diss_4} with $j=0,\dots,n-2$, \eqref{h_diss_2}, \eqref{q_diss_5} with $j=1,\dots,n-2$ and \eqref{h_diss} with $j=n-1$ and $n$ yields the estimate \eqref{q_es_diss}.
\end{proof}

%%%%%%%%%%%%%%%%%%%%%%%%%%%%%%%%%%%%%%%%%%%%%%%%%%%%%%
\section{Global energy estimates}\label{sec_priori}
%%%%%%%%%%%%%%%%%%%%%%%%%%%%%%%%%%%%%%%%%%%%%%%%%%%%%%

In this section we will derive the global-in-time full energy estimates by making use of the estimates derived in Sections \ref{sec_5} and \ref{sec_6}. Set
\beq\label{ggdef}
\fe{N+4}^w(t):=\sup_{0\le s\le t}(1+s)^{N-5}\fe{N+4} (s).
\eeq

%%%%%%%%%%%%%%%%%%%%%%%%%%%%%%%%%%%%%%%%%%%%%%%%%%%%%%
\subsection{Boundedness estimates of $\se{2N}$ and $\fd{2N}$}
%%%%%%%%%%%%%%%%%%%%%%%%%%%%%%%%%%%%%%%%%%%%%%%%%%%%%%

We first show the the boundedness of $\se{2N}$ and $\fd{2N}$.
\begin{thm}\label{bdd_e2N}
Let $N\ge 8$. There exists a universal constant  $\delta>0$ such that if
\beq\label{a priori}
\se{2N}(t)\le \delta,\quad\forall t\in[0,T],
\eeq
then
\beq\label{full_en_es}
\se{ 2N}(t)+ \int_0^t \fd{2N}(s)\,ds \ls\se{2N}(0)+ \sqrt{\sup_{0\le s\le t}\se{ 2N}(s)}\fe{N+4}^w(t),\quad\forall t\in[0,T].
\eeq
\end{thm}
\begin{proof}
First, the estimate \eqref{en_ev_2N} and Cauchy's inequality imply that for $N\ge8$,
\begin{align}\label{vb_es_en21}
\seb{2N}(t) + \int_0^t \sdb{2N}(s)\,ds&\ls\se{2N}(0)+ (\se{ 2N}(t))^{3/2} +\sup_{0\le s\le t}\se{ 2N}(s)\sqrt{\fe{N+4}^w(t)} \int_0^t (1+s)^{-(N-5)/2}\,ds\nonumber
\\&\quad+\sqrt{\sup_{0\le s\le t}\se{ 2N}(s)}\int_0^t  \fd{2N}(s)\, ds
\\& \ls\se{2N}(0)+ \sqrt{\sup_{0\le s\le t}\se{ 2N}(s)}\(\sup_{0\le s\le t}\se{ 2N}(s)+\int_0^t  \fd{2N}(s)\, ds+\fe{N+4}^w(t)\).\nonumber
\end{align}
Next, it follows from the estimates \eqref{v_b_es} and \eqref{di_ev_n1} that
\beq\label{v_b_es110}
  \dtt  \ns{  (\curlv{u})_h }_{2N-1 } +\ns{  (\curlv{u})_h }_{2N-1 }   \ls  \seb{2N}+\sdb{2N}+\fe{N+4}  \(  \se{ 2N} +\fd{2N}\).
\eeq
A Gronwall type argument on \eqref{v_b_es110} yields
 \begin{align}\label{v_b_es11}
  \ns{ (\curlv{u})_h(t) }_{2N -1 }   & \ls\se{2N}(0)+  \sup_{0\le s\le t}\seb{2N}(s)  +\int_0^t \sdb{2N}(s)\,ds\nonumber
 \\&   \quad+ \sup_{0\le s\le t}\se{ 2N}(s)\(\sup_{0\le s\le t}\se{ 2N}(s)+\int_0^t  \fd{2N}(s)\, ds\).
\end{align}
Now combining the estimates \eqref{vb_es_en} and \eqref{qh_en_es} with $n=2N$ yields
\beq\label{ful_es_en1}
\se{2N}
 \ls  \bar{\mathcal{E}}_{2N}+ \ns{ (\curlv{u})_h}_{2N-1} +\fe{N+4}  \se{2N} .
\eeq
Hence, it follows from \eqref{vb_es_en21}, \eqref{v_b_es11} and \eqref{ful_es_en1}  that for $\se{ 2N}\le \delta$ small,
\begin{align}\label{vb_es_en2100}
\sup_{0\le s\le t}\se{ 2N}(s)+ \int_0^t \sdb{2N}(s)\,ds
& \ls\se{2N}(0)+ \sqrt{\sup_{0\le s\le t}\se{ 2N}(s)}\( \int_0^t  \fd{2N}(s)\, ds+\fe{N+4}^w(t)\). 
\end{align}

On the other hand, taking $n=2N$ in the estimates  \eqref{vb_es_di} and \eqref{q_es_diss} and using  the estimate  \eqref{di_ev_n1}, one deduces that for $\se{ 2N}$ small,
\beq\label{vb_es_ditt}
   \dtt  \ns{  (\curlv{u})_h }_{2N-2} +\fd{2N}\ls  \sdb{2N}+\fe{N+4}     \se{2N}  +\fd{N+4} \se{2N},
\eeq
which implies
\beq\label{vb_es_ditt2}
\dtt  \ns{  (\curlv{u})_h }_{2N-2} +\fd{2N}\ls  \sdb{2N}+\fe{N+4}    \se{2N}  .
\eeq
Integrating \eqref{vb_es_ditt2} in time gives in particular that  
\begin{align}\label{vb_es_ditt33}
\int_0^t \fd{2N}(s)\,ds  &\ls\se{2N}(0)  +\int_0^t \sdb{2N}(s)\,ds   +\int_0^t \fe{N+4} (s) \se{ 2N}(s) \,ds \nonumber
\\& \ls\se{2N}(0)+\int_0^t \sdb{2N}(s)\,ds+  {\sup_{0\le s\le t}\se{ 2N}(s)}\fe{N+4}^w(t).
\end{align}
Hence, one may improve \eqref{vb_es_en2100} to be, since $\se{ 2N}\le \delta$ is small,
\beq\label{vb_es_en210000}
\sup_{0\le s\le t}\se{ 2N}(s)+ \int_0^t \fd{2N}(s)\,ds \ls\se{2N}(0)+ \sqrt{\sup_{0\le s\le t}\se{ 2N}(s)}\fe{N+4}^w(t). 
\eeq
This implies \eqref{full_en_es}.
\end{proof}

%%%%%%%%%%%%%%%%%%%%%%%%%%%%%%%%%%%%%%%%%%%%%%%%%%%%%%
\subsection{Decay estimates of $\fe{n}$ and $\fd{n}$}
%%%%%%%%%%%%%%%%%%%%%%%%%%%%%%%%%%%%%%%%%%%%%%%%%%%%%%

Next we derive the energy-dissipation estimates with respect to $\fe{n}$ and $\fd{n}$ and show the decay estimates for $n=N+4,\dots,2N-2$.
\begin{thm}\label{decay_e2N}
Let $N\ge 8$. There exists a universal constant  $\delta>0$ such that if
\beq\label{a priorid}
\se{2N}(t)\le \delta,\quad\forall t\in[0,T],
\eeq
then
\begin{align}\label{decay_est}
&\sum_{j=0}^{N-6} (1+t)^{N-5-j}\fe{N+4+j}(t) + \sum_{j=0}^{N-6}\int_0^t(1+s)^{N-5-j}\fd{N+4+j}(s)\,ds\nonumber
\\&\quad\ls \se{2N}(0)  + \int_0^t\fd{2N}(s)\,ds,\quad\forall t\in[0,T].
\end{align}
\end{thm}
\begin{proof}
First, it follows from the estimates \eqref{vb_es_di}, \eqref{q_es_diss} and \eqref{di_ev_n2} that for $n=N+4,\dots,2N-2,$
\beq\label{decay_es1}
\dtt \norm{(\curlv u)_h}_{n-2}^2+\fd{n}\ls \sdb{n}+  \fd{N+4}{\se{2N}}.
\eeq
This together with  the estimate \eqref{en_ev_n} implies that, since $n\ge N+4$ and $\se{ 2N}\le \delta$ is small, 
\beq\label{eeeqqq}
\frac{d}{dt}\left(\bar{\mathcal{E}}_{n}+\mathcal{B}_{n}+\norm{(\curlv u)_h}_{n-2}^2\right)+\fd{n} \le   0.
\eeq
On the other hand, it follows from the estimates \eqref{vb_es_en_n+2} and \eqref{qh_en_es} that for $n=N+4,\dots,2N-2,$
\beq
\fe{n} \ls   \bar{\mathcal{E}}_n+\norm{(\curlv u)_h}_{n-2}^2+\fe{N+4} \se{2N}.
\eeq
This together with \eqref{Bn_es}  implies  that
\beq\label{decay_es3} 
\fe{n} \ls \bar{\mathcal{E}}_{n}+\mathcal{B}_{n}+\norm{(\curlv u)_h}_{n-2}^2\ls \fe{n}.
\eeq
Hence, combining \eqref{eeeqqq} and \eqref{decay_es3} yields that for $n=N+4,\dots,2N-2,$
\beq\label{decay_es4}
\dtt \fe{n}+\fd{n}\le0.
\eeq     	

Note that  $\fd{n}$ can not control $\fe{n}$, which can be seen by checking both the spatial and the temporal regularities in their definitions. This rules out not only the exponential decay of $\fe{n}$ but also prevents one from using the spatial Sobolev interpolation as \cite{GT2,GT3} to bound $\fe{n} \ls \se{2N}^{1-\theta} \fd{n}^{\theta},  0<\theta<1$ so as to derive the algebraic decay.  Observe that $\fe{\ell}\le \fd{\ell+1}$. Then we will employ a time weighted inductive argument here. To begin with, we may rewrite  \eqref{decay_es4} as that for $j=0,\dots,N-6$,
\beq\label{decay_es5}
\dtt \fe{N+4+j}+\fd{N+4+j}\le 0.
\eeq  	
Multiplying \eqref{decay_es5} by $(1+t)^{N-5-j} $, one has that, by using $\fe{N+4+j}\le \fd{N+5+j}$,
\begin{align}\label{decay_es6}
\dtt\( (1+t)^{N-5-j} \fe{N+4+j}\)+(1+t)^{N-5-j}\fd{N+4+j}&\le (N-5-j)(1+t)^{N-6-j}\fe{N+4+j}\nonumber
\\& \ls   (1+t)^{N-5-(j+1)}\fd{N+4+(j+1)}.
\end{align}
Integrating \eqref{decay_es6} in time directly, by a suitable linear combination of the resulting inequalities, one obtains
\begin{align}\label{v_b_es_119}
&\sum_{j=0}^{N-6} (1+t)^{N-5-j}\fe{N+4+j}(t) + \sum_{j=0}^{N-6}\int_0^t(1+s)^{N-5-j}\fd{N+4+j}(s)\,ds\nonumber
\\&\quad\ls \se{2N}(0)+ \int_0^t\fd{2N-1}(s)\,ds.
\end{align}
This implies \eqref{decay_est}.  	 	
\end{proof}

%%%%%%%%%%%%%%%%%%%%%%%%%%%%%%%%%%%%%%%%%%%%%%%%%%%%%%
\subsection{The a priori estimates}\label{sec_apr}
%%%%%%%%%%%%%%%%%%%%%%%%%%%%%%%%%%%%%%%%%%%%%%%%%%%%%%

Now we can arrive at the ultimate energy estimates.
\begin{thm}\label{apriorith}
Let $N\ge 8$. There exists a universal constant  $\tilde\delta>0$ such that if
\beq\label{arre}
\se{2N}(t)\le \tilde\delta,\quad\forall t\in[0,T],
\eeq
then
\beq\label{thm_en1en}
 \se{2N} (t) + \int_0^t \fd{2N} (s)\,ds \le \tilde C_1 \se{2N} (0),\quad\forall t\in[0,T]
\eeq
and
\beq\label{thm_en2en}
\sum_{j=0}^{N-6} (1+t)^{N-5-j}\fe{N+4+j}(t) + \sum_{j=0}^{N-6}\int_0^t(1+s)^{N-5-j}\fd{N+4+j}(s)\,ds\ls  \se{2N} (0),\ \forall t\in[0,T]  .
\eeq
\end{thm}
\begin{proof}
Let $\tilde\delta$ be smaller than those $\delta$ in Theorems  \ref{bdd_e2N} and \ref{decay_e2N}.
The estimate \eqref{decay_est} of Theorem  \ref{decay_e2N} implies in particular that
\beq\label{efefe}
\fe{N+4}^w(t)\ls \se{2N} (0) +\int_0^t \fd{2N} (s)\,ds.
\eeq 
Then combining the estimates \eqref{full_en_es} and \eqref{efefe} yields that
\beq\label{full_en_esed}
\se{ 2N}(t)+ \int_0^t \fd{2N}(s)\,ds
  \le C_1\se{2N}(0)+ C_1\tilde \delta^{1/2}\int_0^t \fd{2N}(s)\,ds,
\eeq
which implies \eqref{thm_en1en} if $C_1\tilde \delta^{1/2}\le 1/2$.  Finally, \eqref{thm_en2en} follows from \eqref{decay_est}  and  \eqref{thm_en1en}.
\end{proof}

%%%%%%%%%%%%%%%%%%%%%%%%%%%%%%%%%%%%%%%%%%%%%%%%%%%%%%
\section{Local well-posedness}\label{sec_lwp}
%%%%%%%%%%%%%%%%%%%%%%%%%%%%%%%%%%%%%%%%%%%%%%%%%%%%%%

In this section we will prove the local well-posedness of \eqref{MHDv}, as stated in Theorem \ref{main_thm2}. As mentioned already in Section \ref{sec_2}, despite the a priori energy estimates established in Section \ref{sec_priori}, the local well-posedness of \eqref{MHDv} is still highly nontrivial due to the structure of our energy functionals, even the set of initial data with the high order compatibility conditions has to be examined. Based on the a priori energy estimates in Section \ref{sec_priori}, we will use an iteration scheme to construct solutions to \eqref{MHDv} by solving a free-surface Euler equations with surface tension (the hydrodynamic part) and a  two-phase magnetic system in moving domains (the magnetic part). Note that although the hydrodynamic part can be handled easily, yet great cares are needed to construct solutions to the magnetic part due to the nonlocal boundary condition and the less regularity of the velocity. We will construct elaborate approximate solutions for the magnetic part by modifying the analysis of Padula and Solonnikov \cite{PS} and obtain the solution as the limit of these approximate solutions. To this end, our a priori energy estimates independent of the regularization parameter are crucial.

%%%%%%%%%%%%%%%%%%%%%%%%%%%%%%%%%%%%%%%%%%%%%%%%%%%%%%
\subsection{The initial data and the compatibility conditions}\label{sec71}
%%%%%%%%%%%%%%%%%%%%%%%%%%%%%%%%%%%%%%%%%%%%%%%%%%%%%%

Since our energy functional framework requires high order regularity, we have to specify explicitly the high order compatibility conditions of the initial data. Thus, for given initial data, $(u_0,b_0,\eta_0)$, one needs to construct the data $\partial_t^j \eta(0)$ for $j=1,\dotsc,2N+1$, $\partial_t^j u(0)$ and $\partial_t^j b(0)$ for $j=1,\dotsc,2N$,  $\partial_t^j  p(0)$ for $j =0,\dotsc, 2N-1$, and $\partial_t^j  \hb(0)$ for $j =0,\dots, 2N$, and state the $2N$-th order  compatibility conditions. The construction of the data can be given as follows. First, one writes $(\dt^0u(0),\dt^0 b(0),\dt^0\eta(0))=(u_0,b_0,\eta_0)$ and constructs $\dt \eta(0)=u_0\cdot\n_0$,  hereafter $\n_0=\n(0)$, etc. Now, suppose that $j \in [0,2N-1]$ and that $\dt^\ell u(0),\dt^\ell b(0)$ are known for $\ell=0,\dotsc,j$, $\dt^\ell p(0)$ are known for $\ell = 0,\dotsc,j-1$ (with the understanding that nothing is known of $p(0)$ when $j=0$) and  $\dt^\ell \eta(0)$ are known for $\ell = 0,\dotsc,j+1$, then  $ \dt^{j+1} u(0)$,   $\dt^{j+1} b(0)$, $\dt^{j} p(0)$ and $\dt^{j+2} \eta(0)$ can be obtained as follows. First, let $ \dt^{j} p(0)$ be the  solution to 
\beq \label{MHDvp} 
\begin{cases}
\Delta^{\varphi_0} \dt^{j} p(0) =- \Delta^{\varphi_0}\((\D_t)^{j}-\dt^{j}\)p(0) -(\D_t)^{j}\( \nabla^\varphi u: \nabla^\varphi u^t  \)(0) 
\\\qquad\quad\qquad\ \,\,\,+\diverge^{\varphi_0}(\D_t)^{j} (\curlv   b\times (\bar B+b)) (0)   &\text{in }\Omega_-
\\ \dt^{j} p(0)=  -\sigma \dt^{j} H(0)  & \text{on } \Sigma
\\ \p_3^{\varphi_0} \dt^j p(0)= \p_3^{\varphi_0}\((\D_t)^{j}-\dt^{j}\)p(0)+\dt^j \( {\curlv   b\times (\bar B+b) }\)(0) \cdot e_3& \text{on } \Sigma_-.
\end{cases}
\eeq 
Next, define $ \dt^{j+1} u(0)$ as
\beq
 \dt^{j+1} u(0)=  -\((\D_t)^{j+1}-\dt^{j+1}\) u (0) +(\D_t)^j\(  -   {u}\cdot \Dn  {u} -\Dn   p +  \curlv   b\times (\bar B+b)\) (0)
\eeq
and  $ \dt^{j+1} b(0)$ as
\beq\label{dii0}
 \dt^{j+1} b(0)      =   -\((\D_t)^{j+1}-\dt^{j+1}\)b(0) +\curlvz(\D_t)^j   E (0).
\eeq
Finally, set $\dt^{j+2} \eta(0)$ to be
\beq
 \dt^{j+2} \eta(0)      = \dt^{j+1}\(u\cdot\n\)(0).
\eeq
Note that now one has the data  $\partial_t^j \eta(0)$ for $j=0,\dotsc,2N+1$, $\partial_t^j u(0)$ and $\partial_t^j b(0)$ for $j=0,\dotsc,2N$ and  $\partial_t^j  p(0)$ for $j =0,\dotsc, 2N-1$,  one then can obtain  $\partial_t^j \hb(0)$ for $j=0,\dotsc,2N$  as the  solution to, iteratively,  
 \beq\label{MHDvb}
 \begin{cases}
\curlvz \partial_t^j \hb(0)=-\curlvz \((\D_t)^{j}-\dt^{j}\)\hb(0) &\text{in }\Omega_+
\\ \divaz \partial_t^j \hb(0)=-\divaz \((\D_t)^{j}-\dt^{j}\)\hb(0)&\text{in }\Omega_+
\\\dt^j  \hb(0)\cdot \n_0= -\[\dt^j, \n\]\cdot \hb(0)+\dt^j(b\cdot\n)(0)   & \text{on } \Sigma
 \\  \dt^j \hb(0)\times e_3=0  & \text{on } \Sigma_{+}.
\end{cases}
\eeq
Note that  $\curlvz ((\D_t)^{j}-\dt^{j})\hb(0)\cdot e_3=0$ on $\Sigma_+$, which follows by the fact that $e_3\cdot \curlv =e_3\cdot \curl$ and $(\D_t)^{j}=\dt^{j}$ on $\Sigma_+$, and thus, by Proposition \ref{propel1},  guarantees the solvability of \eqref{MHDvb}. The construction of the  data is thus completed. In order for  $(u_0,b_0,\eta_0)$ to be taken as the initial data for the local well-posedness of  \eqref{MHDv} in our energy functional framework,  these data constructed above need to satisfy the following  $2N$-th order  compatibility conditions:
\beq\label{compatibility}
\begin{cases}
   \diverge^{\varphi_0}  u_0  =\diverge^{\varphi_0}  b_0  =0\text{ in }\Omega_-, \ u_{0,3} =b_{0,3} =0\text{ on }\Sigma_-, 
   \\   \jump{\dt^j b(0)}\times \n_0 =0\text{ on }\Sigma \text{ and } \dt^j E(0)\times e_3=0\text{ on }\Sigma_- ,\ j=0,\dots,2N-1.
 \end{cases}
\eeq

We shall now show  that the set of the initial data satisfying the compatibility conditions \eqref{compatibility}  is not empty. In principle, this is highly technical since the problem  \eqref{MHDv} is nonlinear and  nonlocal. Our key observation here is that since the nonlinear problem is a small perturbation of the linearized one, so their compatibility conditions for the initial data should be close to each other. Then our idea is to  first construct the initial data for the linearized problem that satisfies the corresponding linear compatibility conditions, and then we obtain a family of initial data satisfying the compatibility conditions for the nonlinear problem, which are close to the initial data of the linearized problem, by a perturbation argument.  For the linearized problem of \eqref{MHDv} ($i.e.$, \eqref{MHDv_perb} with $G^i$ and $\hat G_i$ being zero), given the initial data $(u_0,b_0,\eta_0)$,  the construction of  the data $\partial_t^j \eta(0)$ for $j=1,\dotsc,2N+1$, $\partial_t^j u(0)$ and $\partial_t^j b(0)$ for $j=1,\dotsc,2N$,  $\partial_t^j  p(0)$ for $j =0,\dotsc, 2N-1$ and $\partial_t^j  \hb(0)$ for $j =0,\dotsc, 2N$  is similar to that of the nonlinear problem \eqref{MHDv}, and the  linear compatibility conditions  is the one obtained by setting $\eta=0$ (and $E=u\times  \bar B -\kappa\curl  b$) in \eqref{compatibility}. Note that  the set of the initial data satisfying the linear compatibility conditions is not empty; indeed, any triple of $u_0\in C_0^\infty(\Omega_-)$ with $\diverge u_0=0$, $b_0\in C_0^\infty(\Omega_-)$ with $\diverge b_0=0$ and $\eta_0=0$ satisfies the linear compatibility conditions, which follows from the fact that, from the construction, $\dt^j\eta(0)=0$ for $j=1,\dots,2N+1$, $\dt^j p(0)=-\bar B\cdot \dt^j b(0)$ for $j=0,\dots,2N-1$, $\dt^j u(0)=  -\nabla\dt^{j-1} p(0)+\curl   \dt^{j-1}  b(0)\times  \bar B$ for $j=1,\dots,2N$ and $\dt^j b(0)=  -\kappa \curl\curl \dt^{j-1}   b(0)+\bar{B}   \cdot \nabla  \dt^{j-1} {u}(0)$ for $j=1,\dots,2N$ all belong to $C_0^\infty(\Omega_-)$ and $\dt^j\hb(0)=0$ for $j=0,\dots,2N$. Now given any smooth initial data satisfying the linear compatibility conditions,  denoted by $(u_0^L, b_0^L,\eta_0^L)$,  one may then employ the abstract argument before Lemma 5.3 of \cite{JT} of using the implicit function theorem to show that there exist  a constant  $\iota_0>0$  and a family of smooth initial data of the form $
(u_0^\iota,b_0^\iota,\eta_0^\iota)
=\iota  (u_0^L, b_0^L,\eta_0^L)+\iota^2
(\tilde{u}(\iota), \tilde{b}(\iota),\tilde{\eta}(\iota))
$
for $\iota\in [0,\iota_0)$ and some smooth $(\tilde{u}(\iota), \tilde{b}(\iota),\tilde{\eta}(\iota))$  so that $(u_0^\iota,b_0^\iota,\eta_0^\iota)$ satisfies the nonlinear compatibility conditions \eqref{compatibility}.

%%%%%%%%%%%%%%%%%%%%%%%%%%%%%%%%%%%%%%%%%%%%%%%%%%%%%%
\subsection{The free-surface Euler equations with surface tension}
%%%%%%%%%%%%%%%%%%%%%%%%%%%%%%%%%%%%%%%%%%%%%%%%%%%%%%

In this subsection we consider the following free-surface Euler equations that for given $ F$,
\beq\label{MHDvk1}
\begin{cases}
\D_t   {u}  +    {u}\cdot \Dn  {u}  +\Dn   p =  F & \text{in } \Omega_-
\\ \diva    u=0  &\text{in }\Omega_-
\\ \partial_t \eta   = u\cdot\N & \text{on }\Sigma
\\  p =  - \sigma H  & \text{on } \Sigma
\\ u_3 =0 & \text{on } \Sigma_{-}
\\ (u,\eta)\mid_{t=0}= (u_0,\eta_0).
\end{cases}
\eeq

Given the initial data $(u_0,\eta_0)$, let the data  $\partial_t^j \eta(0)$ for $j=1,\dotsc,2N+1$, $\partial_t^j u(0)$ for $j=1,\dotsc,2N$ and  $\partial_t^j  p(0)$ for $j =0,\dotsc, 2N-1$  be constructed similarly as in Section \ref{sec71}. The initial data are required to  satisfy the following compatibility conditions
\beq\label{cc1}
   \diverge^{\varphi_0}  u_0  =0\text{ in }\Omega_-\text{ and }  u_{0,3} =0\text{ on }\Sigma_-.
\eeq

Recall the definition \eqref{high_en} of $\se{2N}$, and  denote the  $u$-parts of $\se{2N}$ by $\sen(u)$, etc. Set
\begin{align}\label{nof1}
 \mathfrak{F}_2^{2N}(F):=  \int_0^T \sum_{j=0}^{2N}\norm{ \dt^j  F }_{{2N}-j}^2,\quad \mathfrak{F}^{2N}_\infty(F):= \sup_{[0,T]} \sum_{j=0}^{2N-1} \norm{ \dt^j  F }_{{2N}-j-1}^2
\end{align}
and
\beq\label{nof3}
 \mathfrak{F}_0^{2N}(F):=  \sum_{j=0}^{2N-1} \norm{ \dt^j  F(0) }_{{2N}-j-1}^2.
\eeq
Note that 
\beq\label{feses}
\mathfrak{F}^{2N}_\infty(F)\ls \mathfrak{F}_0^{2N}(F)+T\mathfrak{F}_2^{2N}(F).
\eeq

Now the local well-posedness of \eqref{MHDvk1} can be stated  as follows.
\begin{prop}\label{lemma1}
Let $N\ge 4$ be an integer. Assume  that $\mathfrak{F}_0^{2N}(F)+\mathfrak{F}_2^{2N}(F)<\infty$ for any $0<T\le 1$, $u_0\in H^{2N}(\Omega_-)$ and $\eta_0\in H^{2N+3/2}(\Sigma )$ are given such that $\sen(u,p,\eta)(0)<\infty$ and the compatibility conditions \eqref{cc1} are satisfied.  There exist a universal constant $\delta_1>0$   such that  if  $ \sen(u,p,\eta)(0) +   \mathfrak{F}_0^{2N}(F)+\mathfrak{F}_2^{2N}(F)\le \delta_1$,  then  there exists a  unique solution $(u,p, \eta)$ to  \eqref{MHDvk1} on $[0,T]$ satisfying
\begin{align}
\label{bound222}
\sup_{[0,T]}\sen(u,p,\eta)    \ls    \sen(u,p,\eta)(0)  + \mathfrak{F}_0^{2N}(F)+T\mathfrak{F}_2^{2N}(F)  .
\end{align}  
\end{prop}
\begin{proof}
The problem \eqref{MHDvk1} could be solved similarly as in Coutand and Shkoller \cite{CS1}. In fact, the arguments here would be reasonably easier since the geometry of \eqref{MHDvk1} is simpler than those of \cite{CS1} which treats more general domains in Lagrangian coordinates. So we may omit the details, and focus only on the derivation of the a priori estimate \eqref{bound222}.
The proof follows similarly as that of the a priori estimates for \eqref{MHDv},  not involving the magnetic part, and thus we provide only the necessary modifications. 

Assume  that the solution $(u,p, \eta)$ to \eqref{MHDvk1} is given on the interval $[0,T]$ and satisfies 
\beq\label{apriori_11}
\se{2N}(t)\le \delta,\quad\forall t\in[0,T]
\eeq
for  sufficiently small   $\delta>0.$
First, one may modify the proof of Proposition \ref{evolution_2N} easily to deduce 
 \begin{align}\label{en_ev_2Ne}
\seb{2N}(u,\eta)(t)   
&  \ls        \sen(u,p,\eta)(0)+   (\sen(u,p,\eta)(t))^{3/2} +   \int_0^t   (\sen(u,p,\eta))^{3/2}   \nonumber \\& \quad  +\sum_{j=0}^{2N }\int_0^t \sqrt{\seb{2N}(u)}\norm{ \dt^j  F }_{{2N}-j} 
.
\end{align} 
By Cauchy's inequality and the Cauchy-Schwarz  inequality, one deduces that for $T\le 1$,
\begin{align}
 \label{en_ev_2Ne1}
 \sup_{[0,T]}\seb{2N}(u,\eta) &  \ls        \sen(u,p,\eta)(0)+     \sup_{[0,T]} \(\sen(u,p,\eta)\)^{3/2}  +T\mathfrak{F}_2^{2N}(F).
\end{align}  

Next, following a variant of the proof of Proposition \ref{v_b_prop} by using the vorticity equations
\begin{align}\label{curlv_eqe}
\D_t \curlv{u}  +    {u}\cdot \Dn  \curlv {u} =     \curlv{u}\cdot \Dn  {u} + \curlv F  , 
\end{align}
one obtains 
\begin{align}
 \label{en_ev_2Ne2}
 \sup_{[0,T]}\sen(u)     \ls    \sen(u,\eta)(0)+ \sup_{[0,T]}\seb{2N}(u)+ \sup_{[0,T]} \(\sen(u,p,\eta)\)^{3/2}   + T \mathfrak{F}_2^{2N}(F)+\mathfrak{F}^{2N}_\infty(F).
\end{align}  
Following the proof of Proposition \ref{prop_qh_en} leads to
\begin{align}
 \label{en_ev_2Ne5}
 \sen(p,\eta)  \ls \seb{2N}(\eta)+   \sen(u)+ \(\sen(u,p,\eta)\)^{2}   +\mathfrak{F}_\infty^{2N}(F). 
\end{align}

Now, collecting the estimates \eqref{en_ev_2Ne1}, \eqref{en_ev_2Ne2} and \eqref{en_ev_2Ne5} yields that, by \eqref{feses},
\begin{align}
 \label{en_ev_2Ne7}
\sup_{[0,T]}\sen(u,p,\eta)     \ls        \sen(u,p,\eta)(0)+ \sup_{[0,T]} \(\sen(u,p,\eta)\)^{3/2} + \mathfrak{F}_0^{2N}(F)+T\mathfrak{F}_2^{2N}(F) .
\end{align}  
This implies in particular that, for $\delta>0$ small,
\begin{align}
 \label{en_ev_2Ne8}
 \sup_{[0,T]}\sen(u,p,\eta)   \le C_1\(       \sen(u,p,\eta)(0)+ \mathfrak{F}_0^{2N}(F)+T\mathfrak{F}_2^{2N}(F)  \),
\end{align}  
which closes the a priori assumption  \eqref{apriori_11} if one has assumed  that
\beq
 \sen(u,p,\eta)(0) +   \mathfrak{F}_0^{2N}(F)+\mathfrak{F}_2^{2N}(F) \le \delta_1:=\delta /C_1.
\eeq
 This in turn implies that the solution $(u,p,\eta)$ exists on $[0,T]$ for any $0<T\le 1$ and the estimate \eqref{bound222} holds, and the proposition is thus proved.
\end{proof}

%%%%%%%%%%%%%%%%%%%%%%%%%%%%%%%%%%%%%%%%%%%%%%%%%%%%%%
\subsection{The two-phase magnetic system in moving domains}
%%%%%%%%%%%%%%%%%%%%%%%%%%%%%%%%%%%%%%%%%%%%%%%%%%%%%%

In this subsection we consider the following magnetic system that for given  $\eta$ and $G$,
\beq\label{MHDvk01}
\begin{cases}
\D_t  {b}   +\kappa\curlv\curlv  b     =  \curlv   {G} & \text{in } \Omega_-
\\ \diva  b = 0  &\text{in }\Omega_-
\\\curlv\hb=0,\quad\diva \hb=0 &\text{in }\Omega_+
\\ \jump{b}=0  & \text{on } \Sigma
\\ b_3 =0,\quad \kappa \curlv  b \times e_3 = {G}\times e_3 & \text{on } \Sigma_{-}
\\  \hb\times e_3=0  & \text{on } \Sigma_{+}
\\ b\mid_{t=0}= b_{0}.
\end{cases}
\eeq 

Given the initial data $b_0 $,  let the data $\partial_t^j b(0)$ for $j=1,\dotsc,2N$ and $\partial_t^j \hb(0)$ for $j=0,\dotsc,2N$ be constructed similarly as in Section \ref{sec71}. For the later use, we may reformulate the construction of these data and the corresponding $2N$-th order compatibility conditions as follows. For $j=0,\dots, 2N$, denote $P_j(f_0,\dots, f_{j-1} )$ (depending on $\varphi$) to be the corresponding expression $((\D_t)^j-\dt^j)f(0)$ with $\dt^\ell f(0)$ replaced by $f_\ell $  for $\ell=0,\dots,j-1$  (with the understanding that $P_0=0$ when $j=0$). For the given initial data $b(0)=b_0$, one can define  iteratively  that  for $j=0,\dots,2N-1$,
\begin{align}\label{MHDvbee00e1}
 \dt^{j+1} b(0)      &  = -P_{j+1}( b(0),\dots, \dt^{j} b(0))-\kappa\curlvz    \curlvz \( \dt^{j} b(0)+P_{j}( b (0),\dots, \dt^{j-1} b(0)) \)\nonumber
  \\&\quad +\curlvz   (\D_t)^{j}  G(0) 
\end{align}
and  for $j=0,\dotsc,2N$, $ \partial_t^j \hb(0)$ is the solution to 
 \beq\label{MHDvbee00e2}
 \begin{cases}
\curlvz \partial_t^j \hb(0)=-\curlvz P_{j}( \hb(0),\dots, \dt^{j-1} \hb(0)) &\text{in }\Omega_+
\\ \divaz \partial_t^j \hb(0)=-\divaz P_{j}( \hb(0),\dots, \dt^{j-1} \hb(0))&\text{in }\Omega_+
\\\dt^j  \hb(0)\cdot \n_0= -\[\dt^j, \n\]\cdot \hb(0)+\dt^j(b\cdot\n)(0)   & \text{on } \Sigma
 \\  \dt^j \hb(0)\times e_3=0  & \text{on } \Sigma_{+}.
\end{cases}
\eeq
The   $2N$-th order compatibility conditions  for  \eqref{MHDvk01}  are the following:\beq\label{compatibility0j}  
\begin{cases}
 \diverge^{\varphi_0}  b_0  =0\text{ in }\Omega_-,\   b_{0,3} =0\text{ on }\Sigma_-, \jump{\dt^j b(0)}\times\n_0  =0  \text{ on }\Sigma,\ j=0,\dots 2N-1,
 \\  \kappa \curlvz  \dt^j  b(0) \times e_3\! = \!\! \(\!\!-\kappa \curlvz P_{j}( b (0),\dots, \dt^{j-1} b(0))+ \dt^j G (0) \!\!\)\!\times e_3 \text{ on }\Sigma_-,\ j=0,\dots 2N-1.
 \end{cases}
\eeq	
Note that it follows from \eqref{MHDvbee00e1}, the last two lines in \eqref{compatibility0j}  and the third equation in \eqref{MHDvbee00e2} that
\beq\label{commmo}
\begin{cases}
\divaz \partial_t^j b(0)=-\divaz P_{j}( b(0),\dots, \dt^{j-1} b(0)) \text{ in }\Omega_+ ,\ j=1,\dots,2N,
\\\dt^j b_3(0)=0\text{ on }\Sigma_-,\ j=1,\dots,2N
\text{ and }
\jumps{\dt^j b(0)}\cdot\n_0=0\text{ on }\Sigma,\ j=0,\dots,2N.
 \end{cases}
\eeq

The problem \eqref{MHDvk01} was solved   in Padula and Solonnikov \cite{PS} in a slightly different setting by using the full parabolic regularity of the problem. However, it should be noted that one key subtle point in \cite{PS} is that the iteration scheme of constructing the solutions to the viscous and resistive plasma-vacuum interface problem therein requires high order regularities of $u$ and $\eta$ guaranteed by the viscosity, which unfortunately is not the case here for solving the inviscid and resistive plasma-vacuum interface problem \eqref{MHDv}. Our  way to get around this difficulty is to regularize \eqref{MHDvk01}, and to solve the regularized problem  by modifying the arguments of \cite{PS}. Then we derive the uniform estimates of the approximate solutions independent of the smoothing parameter, which enable us to take the limit to solve the original problem \eqref{MHDvk01}. To this end, we will make an important use of the corresponding regularized electric field in vacuum. As our energy functional is different from the parabolic one of \cite{PS} and is of high order,  we need to solve the regularized problem in a higher regularity counterpart of that of \cite{PS}, which requires us also to smooth out the initial data $b_0$. Such a smoothing procedure is highly technical  as it needs to guarantee the high order compatibility conditions for the regularized problem. It seems extremely difficult for us to apply directly the usual standard regularization technique through mollifiers to achieve this. Our idea here is to introduce the so-called correctors to the regularized problem. 

More precisely, we will regularize  \eqref{MHDvk01} as follows:
\beq\label{MHDvk2}
\begin{cases}
\De_t  {\be}   +\kappa\curlve\curlve  \be     =  \curlve \( G^\eep -\Psi^\eep\)& \text{in } \Omega_-
\\ \divae  \be = 0  &\text{in }\Omega_-
\\\curlve\hbe=0,\quad\divae \hbe=0 &\text{in }\Omega_+
\\ \jump{\be}=0  & \text{on } \Sigma
\\  b^\eep_3 =0,\quad\kappa \curlve  \be \times e_3 = {G^\eep}\times e_3 & \text{on } \Sigma_{-}
\\  \hbe\times e_3=0  & \text{on } \Sigma_{+}
\\ \be\mid_{t=0}= b_{0}^\eep. 
\end{cases}
\eeq
Here $\varphi^\eep= \varphi(\eta^\epp)$ with $\eta^\epp=\( \eta\)^\epp_{t,x_h}$ and $G^\eep=\( G\)^\epp_{t,x}$, where $\( \cdot\)^\epp_{t}$ is the usual smooth  approximation in time through a mollifier, etc., and the corrector $\Psi^\eep$ and the smooth data $b_0^\eep$ are constructed  simultaneously  as follows. For the given $\dt^0b(0)=b_0$, the initial date of the original problem \eqref{MHDvk01}, let $\dt^jb(0)$ for $j=1,\dots,2N$  and $\dt^j\hb(0)$ for $j=0,\dots,2N$ be constructed  by \eqref{MHDvbee00e1} and \eqref{MHDvbee00e2}, respectively.  For $j=0,\dots, 2N$, let $P_j^\epp$  be $P_j$ with $\varphi$ replaced  by $\varphi^\epp$. Note that $P_j^\epp=0$ on $\Sigma_\pm$. For $j=0,\dots, 2N-1$, we define a sequence of smooth functions $w_j^\eep$ and $\hat w_j^\eep$  as follows. Suppose  that $j \in [0,2N-1]$ and that $w_\ell^\eep$ and $\hat w_\ell^\eep$ are known for $\ell=0,\dotsc,j-1$ (with the understanding that nothing is known for $w_0^\eep$ and $\hat w_0^\eep$ when $j=0$), we define $w_j^\eep$ and $\hat w_j^\eep$ as the solution to
\beq\label{MHDvk21100}
\begin{cases}
\kappa\curlvez\curlvez  w_j^\eep     =-\kappa   \curlvez\curlvez  P_{j}^\epp(w_0^\eep,\dots, w_{j-1}^\eep)
\\\qquad\qquad\qquad\qquad\, +\kappa\curlvez\curlvez  \((\D_t)^j  b(0)\)^\epp_{x}-\curlvez\phi_j^\epp& \text{in } \Omega_-
\\ \divaez  w_j^\eep =-\divaez P_j^\epp(w_0^\eep,\dots, w_{j-1}^\eep)  &\text{in }\Omega_-
\\\curlvez\hat{w}_j^\eep=-\curlvez P_j^\epp(\hat w_0^\eep,\dots, \hat w_{j-1}^\eep)  &\text{in }\Omega_+
\\ \divaez  \hat w_j^\eep =-\divaez P_j^\epp(\hat w_0^\eep,\dots, \hat w_{j-1}^\eep) &\text{in }\Omega_+
\\ \jump{w_j^\eep}=0  & \text{on } \Sigma
\\  w_j^\eep\cdot e_3=0,\quad \kappa \curlvez  w_j^\eep \times e_3 =\( \!- \kappa \curlvez  P_{j}^\epp(w_0^\eep,\dots, w_{j-1}^\eep)+\dt^j G^\eep(0)\!\)\times e_3 & \text{on } \Sigma_{-}
\\  \hat w_j^\eep\times e_3=0  & \text{on } \Sigma_{+},
\end{cases}
\eeq
where $\phi_j^\epp$ is a sequence of correctors satisfying 
\beq\label{phiedef}
\phi^\epp_{j,3}=0\text{ and }\phi^\epp_{j,h}=\(\kappa \curlvez  \((\D_t)^j  b(0)\)^\epp_{x}-\dt^j G^\eep(0)\)_h \text{ on }\Sigma_-,\ j=0,\dots,2N-1,
\eeq
which can be constructed by the harmonic extension, similarly as Lemma \ref{p_poisson}.
It should be noted here that the introduction of the correctors $\phi_j^\epp$ is crucial in order to guarantee the solvability of \eqref{MHDvk21100} according to Proposition \ref{propel3}. Indeed, without $\phi_j^\epp$, the solvability of \eqref{MHDvk21100} would require  $\kappa \curlvez  \((\D_t)^j  b(0)\)^\epp_{x}\times e_3=  \dt^j G^\eep(0) \times e_3 $ on $\Sigma_-$, which is not valid in general even that the last line in the $2N$-th order compatibility conditions \eqref{compatibility0j} hold. Next, by the second, sixth and seventh equations in \eqref{MHDvk21100},  according to Proposition \ref{propel1}, for  $j=0,\dots,2N-2$ one can define $\psi_j^\eep$ as the solution to
\beq \label{MHDvk2110022}
\begin{cases}
\curlvez \psi_j^\epp=  -\kappa \curlvez \curlvez \( w_{j}^\eep+P_{j}^\eep(w_0^\eep,\dots, w_{j-1}^\eep)   \)  
\\ \qquad\qquad\quad +\curlvez  (\De_t)^j  G^\eep(0)   -w_{j+1}^\eep- P_{j+1}^\epp(w_0^\eep,\dots, w_{j}^\eep) &\text{in }\Omega_-
\\\divaez \psi_j^\eep=0&\text{in }\Omega_-
\\ \psi_j^\eep\cdot\n_0^\epp=0&\text{on }\Sigma
\\\psi_j^\eep\times e_3=0&\text{on }\Sigma_-,
\end{cases}
\eeq
where $\n^\epp=(-\nabla_h\bar\eta^\epp,1).$ Now we can set $b_0^\eep=w_0^\eep$ and   $\Psi^\eep(t)$, by the time extension similarly as Lemmas \ref{l_sobolev_extension}--\ref{l_sobolev_extension2}, such that 
\beq\label{MHDvk211002233}
(\De_t)^j  \Psi^\eep(0)=\psi_j^\eep, \ j=0,\dots,2N-2\text{ and } 
(\De_t)^{2N-1}  \Psi^\eep(0)=0.
\eeq 
It follows from the fact that $(\D_t)^{j}=\dt^{j}$ on $\Sigma_-$ and the last equation in \eqref{MHDvk2110022} that $ \dt^j \Psi^\eep(0)\times e_3=0 \text{ on }\Sigma_- $, $j=0,...2N-1$, hence one can further choose to have  that (see Chapter 4 in Lions and Magenes \cite{LM})
\beq\label{psibd}
 \Psi^\eep\times e_3=0 \text{ on }\Sigma_-,
\eeq
Note that \eqref{psibd} is required for the solvability of \eqref{MHDvk2}.

Now having constructed smooth $b^\epp(0)=b_0^\eep$ and $\Psi^\epp(t)$,  we can construct the data $\partial_t^j b^\eep(0)$ for $j=1,\dotsc,2N$ and $\partial_t^j \hb^\eep(0)$ for $j=0,\dotsc,2N$  inductively that
   for $j=0,\dots,2N-1$,
\begin{align}\label{MHDvbee00e111}
 \dt^{j+1} b^\eep(0)     & =  -P_{j+1}^\eep( b^\eep(0),\dots, \dt^{j} b^\eep(0))-\kappa\curlvez   \curlvez \( \dt^{j} b^\eep(0)+P_{j}^\eep( b^\eep(0),\dots, \dt^{j-1} b^\eep(0))\) \nonumber
 \\  &\quad +\curlvez   \((\De_t)^{j}  G^\eep(0)-(\De_t)^{j} \Psi^\eep(0)\) 
\end{align}
and that for $j=0,\dotsc,2N$,   
 \beq\label{MHDvbee00e212}
 \begin{cases}
\curlvez \partial_t^j \hb^\eep(0)=-\curlvez P^\eep_{j}( \hb^\eep(0),\dots, \dt^{j-1} \hb^\eep(0)) &\text{in }\Omega_+
\\ \divaez \partial_t^j \hb^\eep(0)=-\divaez P^\eep_{j}( \hb^\eep(0),\dots, \dt^{j-1} \hb^\eep(0))&\text{in }\Omega_+
\\\dt^j  \hb^\eep(0)\cdot \ne_0= -\[\dt^j, \ne\]\cdot \hb^\eep(0)+\dt^j(b^\eep\cdot\ne)(0)   & \text{on } \Sigma
 \\  \dt^j \hb^\eep(0)\times e_3=0  & \text{on } \Sigma_{+},
 \end{cases}
\eeq
where $\ne=(-\nabla_h\eta^\epp,1)$.

We now claim that $\dt^j b^\eep(0)=w_j^\epp$, $j=1,\dots,2N-1$ and $\dt^j \hb^\eep(0)=\hat w_j^\epp$, $j=0,\dots,2N-1$.
First, since $b^\epp(0)=b_0^\eep=w_0^\eep$, it follows from \eqref{MHDvk21100} with $j=0$ that  $\hat w_0^\eep$ solves \eqref{MHDvbee00e212} with $j=0$, and hence by the uniqueness one has   $\hb^\epp(0)=\hat w_0^\eep$. Now, suppose  that $j \in [0,2N-2]$ and that     $\dt^\ell b^\eep(0)=w_\ell^\eep$ and $\dt^\ell \hb^\eep(0)=\hat w_\ell^\eep$ have been verified for $\ell=0,\dotsc,j$, one finds that, by \eqref{MHDvbee00e111}, \eqref{MHDvk2110022} and the first equation in \eqref{MHDvk2110022} and \eqref{MHDvk211002233},
\begin{align}
 \dt^{j+1} b^\eep(0)     & = -P_{j+1}^\eep(w_0^\eep,\dots, w_{j}^\eep)-\kappa\curlvez    \curlvez \( w_{j}^\eep+P_{j}^\eep(w_0^\eep,\dots, w_{j-1}^\eep)\) \nonumber
 \\  &\quad +\curlvez    (\De_t)^{j}  G^\eep(0)-\curlvez(\De_t)^{j}  \Psi^\eep(0) \nonumber
  \\& =  w_{j+1}^\eep.
\end{align}
It then follows from \eqref{MHDvk21100} with $j$ replaced by $j+1$ and the induction assumption that $\hat w_{j+1}^\eep$ solves  \eqref{MHDvbee00e212} with $j$ replaced by $j+1$, and hence by the uniqueness  one has $\dt^{j+1}\hb^\eep(0)=\hat w_{j+1}^\eep$. The claim is thus proved. Note then that  by  \eqref{MHDvk21100}, one finds that the corresponding $2N$-th order compatibility conditions for \eqref{MHDvk2} are satisfied, $i.e.,$ 
\beq\label{compatibility0je}  
\begin{cases}
\divaez b_0^\epp  =0\text{ in }\Omega_-,\   b_{0,3}^\epp =0\text{ on }\Sigma_-\text{ and}
 \\\text{for } j=0,\dots 2N-1,
 \ 
 \jump{\dt^j b^\epp(0)}\times\ne_0  =0  \text{ on }\Sigma  \text{ and } 
   \\  \kappa \curlvez  \dt^j  b^\epp(0) \times e_3 =  \(-\kappa \curlvez P_{j}^\epp( b^\epp(0),\dots, \dt^{j-1} b^\epp(0))+{ \dt^j G^\epp (0) }\)\times e_3 \text{ on }\Sigma_- .
 \end{cases}
\eeq	

In general, $b_0^\epp$ constructed above does not converge to $b_0$ and $\Psi^\epp$  does not vanish as $\epp\rightarrow0$. To ensure such convergence, additional conditions are required as shown in the following lemma, where
\beq\label{FFdef}
\mathfrak{F}_2^{2N}(G):=  \int_0^T \sum_{j=0}^{2N}\norm{ \dt^j  G }_{{2N}-j}^2,\quad \mathfrak{F}_\infty^{2N}(G):= \sup_{[0,T]} \sum_{j=0}^{2N-1} \norm{ \dt^j  G }_{{2N}-j}^2
\eeq
and
\beq\label{Psidef}
\mathfrak{F}_2^{2N}(\Psi):=  \int_0^T \sum_{j=0}^{2N}\norm{ \dt^j  \Psi }_{{2N}-j+1/2}^2,\quad\mathfrak{F}_\infty^{2N}(\Psi):= \sup_{[0,T]} \sum_{j=0}^{2N} \norm{ \dt^j  \Psi }_{{2N}-j}^2.
\eeq 
\begin{lem}\label{covrem}
Suppose that $\sup_{[0,  T]} \se{2N}(\eta)<\infty$, $\mathfrak{F}_\infty^{2N}(G)+\mathfrak{F}_2^{2N}(G)<\infty$,  $\se{2N}(b,\hb)(0)<\infty$ and  the $2N$-th order compatibility conditions \eqref{compatibility0j} are satisfied. Then as $\epp\rightarrow 0$, $(b_0^\epp,\hb_0^\epp)  \rightarrow  (b_0,\hb_0)$ in the norms of $\se{2N}(b,\hb)(0)$  and $\Psi^\epp \rightarrow 0$  in the norms of $\mathfrak{F}_\infty^{2N}(\Psi)+\mathfrak{F}_2^{2N}(\Psi)$. 
\end{lem}
\begin{proof}
First, it follows from the usual properties of mollifiers that  if $\sup_{[0,  T]} \se{2N}(\eta)<\infty$, then $\sup_{[0,  T]} \se{2N}(\eta^\epp)\ls \sup_{[0,  T]} \se{2N}(\eta)<\infty$, $\eta^\epp\rightarrow \eta$  in the norms of $\int_0^T\se{2N}(\eta)$ and $\eta^\epp(0)\rightarrow \eta(0)$  in the norms of $\se{2N}(\eta)(0)$ as $\epp\rightarrow0$. Similarly, if $\mathfrak{F}_\infty^{2N}(G)+\mathfrak{F}_2^{2N}(G)<\infty$, then  $\mathfrak{F}_\infty^{2N}(G^\epp)+\mathfrak{F}_2^{2N}(G^\epp)\ls \mathfrak{F}_\infty^{2N}(G)+\mathfrak{F}_2^{2N}(G)<\infty$, $G^\epp\rightarrow G$  in the norms of $\mathfrak{F}_2^{2N}(G)$ and $\dt^jG^\epp(0)\rightarrow \dt^j G(0)$ in $H^{2N-j}(\Omega_-),\ j=0,\dots,2N-1$, as $\epp\rightarrow0$. 

Now suppose that $\se{2N}(b,\hb)(0)<\infty$ and \eqref{compatibility0j} holds and recall \eqref{commmo}. Then by \eqref{phiedef} and the last line in \eqref{compatibility0j}, according to  the trace theory and the estimates of the harmonic extension similarly as Lemma \ref{p_poisson}, one has that $\phi^\epp_j\rightarrow 0$ in $H^{2N-j}(\Omega_-),\ j=0,\dots,2N-1$,  as $\epp\rightarrow 0$. Hence, by Propositions \ref{propel3} and \ref{propel1}, it is then routine to check from \eqref{MHDvk21100} for $j=0,\dots, 2N-1$, \eqref{MHDvbee00e111} with $j=2N-1$,    \eqref{MHDvbee00e212} with $j=2N$ and \eqref{MHDvk2110022} for $j=0,\dots, 2N-2$ that as $\epp\rightarrow 0$,  $(b_0^\epp,\hb_0^\epp)  \rightarrow  (b_0,\hb_0)$ in the norms of $\se{2N}(b,\hb)(0)$ and $\psi^\epp_j\rightarrow 0$ in $H^{2N-j}(\Omega_-)$, $j=0,\dots,2N-2$. Finally, according to \eqref{MHDvk211002233} and the estimates of the time extension similarly as Lemmas \ref{l_sobolev_extension}--\ref{l_sobolev_extension2}, one has that $\Psi^\epp \rightarrow 0$  in the norm of $\mathfrak{F}_\infty^{2N}(\Psi)+\mathfrak{F}_2^{2N}(\Psi)$ as $\epp\rightarrow 0$.
\end{proof}

We now establish the well-posedness of the regularized problem  \eqref{MHDvk2}.  Recall  the $L^2$ anisotropic space-time Sobolev spaces  
\begin{align} \label{parabolicspace}
H^{r,r/2}((0,T)\times\Omega)=L^2(0,T; H^r(\Omega))\cap H^{r/2}(0,T; L^2(\Omega)), \ r\ge 0,
\end{align} 
etc., see  Lions and Magenes \cite{LM}. Define
\beq
\mathfrak{K}^{n}(b^\epp ,\hb^\epp   ):=\norm{b^\epp }_{H^{2n+\ell,n+\ell/2}((0,T)\times\Omega_-)} + \norm{\hb^\epp }_{H^{2n+\ell,n+\ell/2}((0,T)\times\Omega_+)}.
\eeq
\begin{lem}\label{lemma2}
Let $N\ge 4$ be an integer and $1/2<\ell<1$ or $1<\ell<3/2$.
 Assume that the smooth initial data $b_0^\epp $ satisfies the $2N$-th order compatibility conditions \eqref{compatibility0je}. There exists a universal constant $\delta_2>0$ such that if $\sup_{[0,  T]}\se{2N}(\eta)\le \delta_2 $  for any $0<T\le 1$, then  for any $\epp>0$  there exists a  unique strong solution $(b^\eep,\hb^\eep)$ to  \eqref{MHDvk2}  on $[0,  T]$ satisfying
\beq\label{kkee}
\mathfrak{K}^{2N}(b^\eep,\hb^\eep ) <\infty.
\eeq
\end{lem}
\begin{proof} 
Note that a similar problem as \eqref{MHDvk2} was solved in Padula and Solonnikov \cite{PS}, that is, the problem when an isolated plasma surrounded by a vacuum which are bounded from the outside by a perfectly conducting wall. But the main scheme in \cite{PS} can be modified slightly to the  case here. We utilize the  results  of \cite{PS} and repeat some main steps, but refer to \cite{PS} for the full details. 

As in Beale \cite{B2} and Padula and Solonnikov \cite{PS}, one may introduce the following change of unknowns:
\beq\label{btr}
\bb^\eep=J^\eep(\jj^\eep)^{-1} b^\eep \text{ in }\Omega_-\text{ and }\hhbb^\eep =J^\eep (\jj^\eep)^{-1} \hb^\eep \text{ in }\Omega_+,
\eeq
where $\jj^\eep=\nabla\Phi^\eep$ with $\Phi^\eep=\Phi(\eta^\eep)$ in \eqref{diff_def} and $J^\eep$ is its determinant.  The advantages of introducing $\bb^\eep$ and $\hhbb^\eep$ are that $\diverge \bb^\eep=0$ in $\Omega_-$, $\diverge \hhbb^\eep=0$ in $\Omega_+$ and that it also keeps the boundary conditions of $\bb^\eep $ and $\hhbb^\eep $ same as those of $b$ and $\hb$.   By \eqref{btr}, one may rewrite \eqref{MHDvk2} with  $\eta^\epp $ small  in terms of $(\bb^\eep ,\hhbb^\eep )$ in the following perturbed form:
\beq\label{MHDvk4}
\begin{cases}
\dt \bb^\eep   +\kappa\curl \curl \bb^\eep     =\mathcal{Q}^{1,\epp}    & \text{in } \Omega_-
\\ \diverge \bb^\eep  = 0  &\text{in }\Omega_-
\\ \curl\hhbb^\eep  =\curl  {\mathcal{Q}}^{2,\epp},\quad\diverge  \hhbb^\eep =0 &\text{in }\Omega_+
\\ \jump{\bb^\eep }=0  & \text{on } \Sigma
\\ \bb^\eep _3=0 ,\quad \kappa  \curl \bb^\eep  \times e_3 =\mathcal{Q}^{3,\epp}\times e_3 & \text{on } \Sigma_{-}
\\  \hhbb^\eep \times e_3=0  & \text{on } \Sigma_{+}
\\ \bb^\eep \mid_{t=0}= \bb^\eep _{0}\equiv J_0^\eep (\jj_0^\eep)^{-1} b_0^\epp,
\end{cases}
\eeq
where
\begin{align}\label{q1def}
\mathcal{Q}^{1,\epp}&=\curl\( (\jj^\eep)^T({G^\eep}-\Psi^\eep)+\kappa ( \curl \bb^\eep -(\jj^\eep)^T(J^\eep)^{-1}\jj^\eep\curl( (\jj^\eep)^T(J^\epp)^{-1}\jj^\eep  \bb^\eep  ))\) \nonumber
\\&\quad-J^\eep(\jj^\eep)^{-1}\dt((J^\epp)^{-1}\jj^\eep) \bb^\eep + (\jj^\eep)^{-1}\dt\bar{\eta^\eep}\p_3((J^\epp)^{-1}\jj^\epp \bb^\eep ),
\\\label{q2def}
  {\mathcal{Q}}^{2,\epp} &= \(I- (\jj^\eep)^T(J^\epp)^{-1}\jj^\epp\)  \hhbb^\eep   ,
\\\label{q3def}
\mathcal{Q}^{3,\epp}& = (\jj^\eep)^T{G^\epp}+\kappa ( \curl \bb^\eep -(\jj^\eep)^T(J^\epp)^{-1}\jj^\epp\curl( (\jj^\eep)^T(J^\epp)^{-1}\jj^\epp  \bb^\eep  ))
.
\end{align} 
It is straightforward to check that 
\beq\label{Qdenti}
\diverge \mathcal{Q}^{1,\epp} =0\text{  in }\Omega_- , \  {\mathcal{Q}}^{1,\epp}\cdot e_3=\diverge_h (\mathcal{Q}^{3,\epp}\times e_3)_h\text{  on }\Sigma_-\text{  and } {\mathcal{Q}}^{2,\epp}\times e_3=0\text{  on }\Sigma_+,
\eeq
where one has used \eqref{psibd}. The  $n$-th order compatibility conditions  for  \eqref{MHDvk4}    read as
\beq\label{compatibility0jt}  
 \begin{cases}
 \diverge  \bb^\eep _0  =0\text{ in }\Omega_-,\   \bb^\eep _{0,3} =0\text{ on }\Sigma_-, 
 \ 
 \jump{\dt^j \bb^\eep (0)}\times e_3  =0\text{ on } \Sigma,\  j=0,\dots n-1,
   \\   \kappa \curl   \dt^j  \bb^\eep (0) \times e_3 =     { \dt^j \mathcal{Q}^{3,\epp} } \times e_3 \text{ on }\Sigma_- ,\  j=0,\dots,n-1.
 \end{cases}
\eeq

The problem \eqref{MHDvk4} with $\mathcal{Q}^{2,\epp}=0, \mathcal{Q}^{3,\epp}=0$ and  $\mathcal{Q}^{1,\epp} $ given and satisfying \eqref{Qdenti} could be  solved  by employing the Galerkin method as in Ladyzhenskaya and Solonnikov \cite{LS1,LS}.
For the general  $\mathcal{Q}^{2,\epp}\neq 0$ or $\mathcal{Q}^{3,\epp}\neq 0$,  \eqref{MHDvk4}  could be solved as in Padula and Solonnikov \cite{PS} by making use of the full parabolic regularity of the problem,   which works in  the  anisotropic space-time Sobolev spaces; such spaces allow for the control of the resulting forcing terms when one adjusts $\mathcal{Q}^{2,\epp}$ and $\mathcal{Q}^{3,\epp}$ to be zero.
Indeed, by  \eqref{Qdenti}, according to Proposition \ref{propel3} (setting $\eta$=0), one can define $(\tilde\bb^\epp,\tilde\hhbb^\epp)$  as the solution to
 \beq\label{MHDvk411} 
\begin{cases}
\kappa    \curl \curl \tilde\bb^\epp   ={\mathcal{Q}}^{1,\epp} & \text{in } \Omega_-
\\ \diverge \tilde\bb^\epp = 0 & \text{in } \Omega_-
 \\\curl  \tilde\hhbb^\epp= \curl {\mathcal{Q}}^{2,\epp}  ,\quad\diverge \tilde\hhbb^\epp= 0 &\text{in }\Omega_+
\\\jump{\tilde\bb^\epp}=0  & \text{on } \Sigma
\\  \tilde\bb^\epp_3=0,\quad  \kappa  \curlv \tilde\bb^\epp \times e_3={\mathcal{Q}}^{3,\epp}  & \text{on } \Sigma_{-}
\\  \tilde\hhbb^\epp\times e_3=0  & \text{on } \Sigma_{+}.
\end{cases}
\eeq
Similarly as for the Corollary on page 151 of  \cite{PS}, in a higher regularity, one can show that
\begin{align}
 \label{mm22}
\mathfrak{K}^{n}(\tilde \bb^\eep,\tilde\hhbb^\epp)   \ls\mathfrak{K}^{n}(\mathcal{Q}^{1,\epp},\mathcal{Q}^{2,\epp},\mathcal{Q}^{3,\epp}), 
\end{align}
where
\begin{align}
\mathfrak{K}^{n}(\mathcal{Q}^{1,\epp},\mathcal{Q}^{2,\epp},\mathcal{Q}^{3,\epp})&:=\norm{\mathcal{Q}^{1,\epp}}_{H^{2n-2+\ell,n-1+\ell/2}((0,T)\times\Omega_-)} 
+ \norm{\mathcal{Q}^{2,\epp}}_{H^{2n+\ell,n+\ell/2}((0,T)\times\Omega_+)} 
\nonumber
\\&\,\quad+ \norm{\mathcal{Q}^{3,\epp}}_{H^{2n-3/2+\ell,n-3/4+\ell/2}((0,T)\times\Sigma_-)} . 
\end{align}  
Then one finds that $(\bar\bb^\eep, \bar\hhbb^\eep):=( \bb^\eep-\tilde \bb,  \hhbb^\eep-\tilde \bb)$ solves
\beq\label{MHDvk4123}
\begin{cases}
\dt \bar\bb^\eep   +\kappa\curl \curl \bar\bb^\eep    =\bar  {\mathcal{Q}}^{1,\epp}:=- \dt \tilde\bb^\eep      & \text{in } \Omega_-
\\ \diverge \bar\bb^\eep  = 0  &\text{in }\Omega_-
\\ \curl\bar\hhbb^\eep  =0,\quad\diverge  \hhbb^\eep =0 &\text{in }\Omega_+
\\ \jump{\bar\bb^\eep }=0  & \text{on } \Sigma
\\ \bar\bb^\eep _3=0 ,\quad \kappa  \curl \bar\bb^\eep  \times e_3 =0 & \text{on } \Sigma_{-}
\\  \bar\hhbb^\eep \times e_3=0  & \text{on } \Sigma_{+}
\\ \bar\bb^\eep \mid_{t=0}= \bar\bb^\eep _{0}\equiv \bb^\eep _{0}-\tilde\bb^\eep (0).
\end{cases}
\eeq
Note that by \eqref{mm22}, one has 
\begin{align}
 \label{mmm1nn33}
\mathfrak{K}^{n}(\bar  {\mathcal{Q}}^{1,\epp})&:=\norm{\bar  {\mathcal{Q}}^{1,\epp}}_{H^{2n-2+\ell,n-1+\ell/2}((0,T)\times\Omega_-)}\ls \norm{\tilde \bb^\epp }_{H^{2n+\ell,n+\ell/2}((0,T)\times\Omega_-)} \nonumber\\&\ls\mathfrak{K}^{n}(\mathcal{Q}^{1,\epp},\mathcal{Q}^{2,\epp},\mathcal{Q}^{3,\epp})
\end{align}  
and by the trace theory (see Chapter 4 in Lions and Magenes \cite{LM}),
\begin{align}
 \label{mmm1nn3}
  \mathfrak{K}^{n}_0(\bar\bb^\eep )&:= \norm{ \bar\bb^\eep _0}_{H^{2n-1+\ell}(\Omega_-)}\le   \mathfrak{K}^{n}_0( \bb^\eep )+  \mathfrak{K}^{n}_0( \tilde\bb^\eep (0) )
  \ls   \mathfrak{K}^{n}_0( \bb^\eep )+ \norm{\tilde \bb^\epp }_{H^{2n+\ell,n+\ell/2}((0,T)\times\Omega_-)} 
  \nonumber\\&\ls\mathfrak{K}^{n}_0( \bb^\eep )+\mathfrak{K}^{n}(\mathcal{Q}^{1,\epp},\mathcal{Q}^{2,\epp},\mathcal{Q}^{3,\epp}). 
\end{align}

We  may now apply the results of Theorem 4 in   \cite{PS}, with a slight modification, in a higher regularity context as follows. Assume that $\bb^\eep _0$ and $\mathcal{Q}^{1,\epp},\mathcal{Q}^{2,\epp}, \mathcal{Q}^{3,\epp}$ are given such that  $\mathfrak{K}_0^{n}(\bb^\eep )<\infty$, $\mathfrak{K}^{n}(\mathcal{Q}^{1,\epp},\mathcal{Q}^{2,\epp},\mathcal{Q}^{3,\epp})<\infty$ and the $n$-th compatibility conditions \eqref{compatibility0jt}   are satisfied. Then there exists a unique strong solution $(\bar\bb^\eep, \bar\hhbb^\eep) $ to \eqref{MHDvk4123}, and hence $(\bb^\eep ,\hhbb^\eep  )$ solves \eqref{MHDvk4} on $[0,T]$ satisfying, by \eqref{mm22}, \eqref{mmm1nn33} and \eqref{mmm1nn3},
\begin{align}
 \label{mmm1nn}
\mathfrak{K}^{n}(\bb^\eep ,\hhbb^\eep  )
  &\ls \mathfrak{K}^{n}(\bar\bb^\eep ,\bar\hhbb^\eep  )+\mathfrak{K}^{n}(\tilde \bb^\eep,\tilde\hhbb^\epp)\ls \mathfrak{K}^{n}_0(\bar\bb^\eep )+ \mathfrak{K}^{n}(\bar  {\mathcal{Q}}^{1,\epp})+\mathfrak{K}^{n}(\mathcal{Q}^{1,\epp},\mathcal{Q}^{2,\epp},\mathcal{Q}^{3,\epp})\nonumber
  \\ &  \ls \mathfrak{K}^{n}_0(\bb^\eep )+\mathfrak{K}^{n}(\mathcal{Q}^{1,\epp},\mathcal{Q}^{2,\epp},\mathcal{Q}^{3,\epp}). 
\end{align}  
Indeed, the case that when  $n=1$, $1/2<\ell<1$ and the first order compatibility conditions ($i.e.$, \eqref{compatibility0jt} with   $n=1$) are satisfied was proved in \cite{PS}; the restriction   $\ell<1$ can be relaxed to include the case of $1<\ell<3/2$,  see the last paragraph on page 578 of Solonnikov \cite{So2}. The restriction of $\ell$,  $1/2<\ell<3/2$ with $\ell\neq 1$, is required so that the trace operator of $H^{2+\ell,1+\ell/2}((0,T)\times\Omega_-)$ onto the set of the initial data in \eqref{MHDvk4123} satisfying the first order compatibility conditions has a bounded right inverse, see  Lemma 2.1 in Beale \cite{B1} or Chapter 4 in Lions and Magenes \cite{LM}.  This then allows one to adjust the initial data to be zero, see Theorem 4 in \cite{PS}. The general cases for $n\ge 1$ follow  by an induction argument under the assumption \eqref{compatibility0jt}.

We now construct solutions to  \eqref{MHDvk4} with $\mathcal{Q}^{1,\epp}=\mathcal{Q}^{1,\epp}(\bb^\eep ,\eta^\eep, G^\eep,\Psi^\eep)$, $\mathcal{Q}^{2,\epp}=\mathcal{Q}^{2,\epp}(\hhbb^\eep ,\eta^\eep)$ and  $\mathcal{Q}^{3,\epp}=\mathcal{Q}^{3,\epp}(\bb^\eep ,\eta^\eep,G^\eep)$ defined by \eqref{q1def}--\eqref{q3def}, respectively.  We will use an iteration argument by making use of  the smallness of $\se{2N}(\eta^\eep)$;  it is crucial for our later use to not assume the higher order norms  of $\eta^\eep$, say, $\norm{\eta^\eep}_{H^{4N-1/2+\ell,2N-1/4+\ell/2}((0,T)\times\Sigma)} $, to be small. Our key point here is to apply \eqref{mmm1nn} in two levels of regularity, $i.e.$, $n=3$ and $2N$, respectively. One may use the following well-known fact (see Lions and Magenes \cite{LM}) that for $l>(d+2)/2$ with $d$ the spatial dimension, 
\beq\label{ssbb}
\norm{fg}_{H^{r,r/2}}\ls \norm{f }_{H^{l,l/2}}\norm{g}_{H^{r,r/2}}+\norm{g }_{H^{l,l/2}}\norm{f}_{H^{r,r/2}}.
\eeq
Recall that $ \eta^\eep,G^\eep,\Psi^\epp$ and $\bb_0^\epp$ are smooth. By \eqref{ssbb}, it is direct to check that
\begin{align}\label{mmm11} 
\mathfrak{K}^{3}(\mathcal{Q}^{1,\epp},\mathcal{Q}^{2,\epp},\mathcal{Q}^{3,\epp})
  \ls C_\epp+ \sup_{[0,T]}\se{2N}(\eta)\mathfrak{K}^{3}(\bb^\eep ,\hhbb^\eep  )
\end{align} 
and 
\begin{align}\label{mmm12} 
\mathfrak{K}^{2N}(\mathcal{Q}^{1,\epp},\mathcal{Q}^{2,\epp},\mathcal{Q}^{3,\epp})
  \ls C_\epp+ \sup_{[0,T]}\se{2N}(\eta)\mathfrak{K}^{2N}(\bb^\eep ,\hhbb^\eep  )+  C_{\epp}\mathfrak{K}^{3}(\bb^\eep ,\hhbb^\eep  ),
\end{align}  
where $C_\epp$ is a positive constant depending on $ \eta^\epp,G^\epp,\Psi^\epp$ and $\bb_0^\epp $. The solution to \eqref{MHDvk4} is obtained  as the limit of a sequence of approximate solutions to be constructed below. 
We first extend the initial data $(\dt^j \bb^\eep (0),\dt^j\hhbb^\eep (0))$ to a time-dependent function $(\bb^{\epp,0},\hhbb^{\epp,0}) $ such that $(\dt^j \bb^{\epp,0}(0),\dt^j \hhbb^{\epp,0}(0)) =( \dt^j \bb^\eep(0),\dt^j \hhbb^\eep (0)),\ j= 0,\dots,2N-1$, see for instance  Lemmas \ref{l_sobolev_extension}--\ref{l_sobolev_extension2}. Next, we claim that  there exist two constants   $M_2> M_1>0$, independent of $m$,  such that for $m\ge 0$,  if $(\bb^{\epp,m},\hhbb^{\epp,m})$ satisfies 
\beq
(\dt^j \bb^{\epp,m}(0),\dt^j \hhbb^{\epp,m}(0)) =( \dt^j \bb^\eep(0),\dt^j \hhbb^\eep (0)),\ j= 0,\dots,2N-1
\eeq
and 
\beq\label{claimmbound}
\mathfrak{K}^{3}(\bb^{\epp,m},\hhbb^{\epp,m})\le M_1\text{ and }\mathfrak{K}^{2N}(\bb^{\epp,m},\hhbb^{\epp,m})\le M_2 ,
\eeq
then there exists a unique solution $(\bb^{\epp,m+1},\hhbb^{\epp,m+1})$ to \eqref{MHDvk4} with $\mathcal{Q}^{1,\epp}=\mathcal{Q}^{1,\epp}(\bb^{\epp,m},\eta^\epp, G^\epp,\Psi^\epp)$, $\mathcal{Q}^{2,\epp}=\mathcal{Q}^{2,\epp}(\hhbb^{\epp,m},\eta^\epp)$ and  $\mathcal{Q}^{3,\epp}=\mathcal{Q}^{3,\epp}(\bb^{\epp,m},\eta^\epp,G^\epp)$ satisfying
\beq\label{claimmboundm1}
\mathfrak{K}^{3}(\bb^{\epp,m+1},\hhbb^{\epp,m+1})\le M_1\text{ and }\mathfrak{K}^{2N}(\bb^{\epp,m+1},\hhbb^{\epp,m+1})\le M_2.
\eeq
To prove the claim, note first that  \eqref{mmm11}--\eqref{claimmbound}  imply  $ \mathfrak{K}^{2N}(\mathcal{Q}^{1,\epp},\mathcal{Q}^{2,\epp},\mathcal{Q}^{3,\epp})<\infty$  and that the corresponding $2N$-th order compatibility conditions are satisfied.  Hence, one has the existence of  $(\bb^{\epp,m+1},\hhbb^{\epp,m+1})$. Moreover,  by \eqref{mmm1nn} with $n=3$, \eqref{mmm11} and the first assumption in \eqref{claimmbound}, one obtains
\begin{align}\label{clm+11}
\mathfrak{K}^{3}(\bb^{\epp,m+1},\hhbb^{\epp,m+1}) \le C_2 \(C_\epp+ \delta_2M_1\).
\end{align} 
So if $\delta_2\le 1/(2C_2)$ and taking 
$
M_1= 2C_2 C_\epp,
$
 then one has
 \begin{align}
\mathfrak{K}^{3}(\bb^{\epp,m+1},\hhbb^{\epp,m+1}) \le  M_1/2  +M_1/2 =  M_1.
\end{align}
On the other hand, by \eqref{mmm1nn} with $n=2N$, \eqref{mmm12}  and the second assumption in \eqref{claimmbound}, one has
\begin{align}
\mathfrak{K}^{2N}(\bb^{\epp,m+1},\hhbb^{\epp,m+1}) \le C_2\(  C_\epp+\delta_2M_2+ C_\epp
M_1  \) .
\end{align}
Hence, taking 
$
M_2=2 C_2C_\epp\(  1 +  
M_1\) ,
$
one then gets
 \begin{align}
\mathfrak{K}^{2N}(\bb^{\epp,m+1},\hhbb^{\epp,m+1}) \le M_2/2  +M_2/2 =  M_2.
\end{align}
Thus the claim is proved. Note that if one has taken $M_1\ge \mathfrak{K}^{2N}(\bb^{\epp,0},\hhbb^{\epp,0}) $, then 
  \eqref{claimmbound} holds for $m=0$. Consequently, one can then iterate from $m=0$ to construct the sequence of approximate solutions $\{(\bb^{\epp,m},\hhbb^{\epp,m})\}_{m=1}^\infty$.
  
   The uniform estimates in \eqref{claimmbound} imply that as $m\rightarrow \infty$, up to a subsequence, the sequence $(\bb^{\epp,m},\hhbb^{\epp,m})$ converges to a limit $(\bb^\epp, \hhbb^\epp )$ in
the weak or weak-$\ast$ sense of the norms in   $\mathfrak{K}^{2N}$. Moreover, according to the weak lower semicontinuity, one has
 \beq\label{bbesq}
\mathfrak{K}^{2N}(\bb^\epp,\hhbb^\epp)\le M_2.
\eeq

 Now we prove the contraction of the approximate sequence $\{(\bb^{\epp,m},\hhbb^{\epp,m})\}_{m=0}^\infty$. For $m\ge 1$,   set $\BB^{\epp,m}=\bb^{\epp,m}-\bb^{m-1}$ and $\hat \BB^{\epp,m}=\hhbb^{\epp,m}-\hhbb^{m-1}$. Then $(\BB^{\epp,m+1},\hat\BB^{\epp,m+1})$ solves \eqref{MHDvk4} with $\mathcal{Q}^{1,\epp}=\mathcal{Q}^{1,\epp}(\BB^{\epp,m},\eta^\epp, 0,0)$, $\mathcal{Q}^{2,\epp}=\mathcal{Q}^{2,\epp}(\hat\BB^{\epp,m},\eta^\epp)$,  $\mathcal{Q}^{3,\epp}=\mathcal{Q}^{3,\epp}(\BB^{\epp,m},\eta^\epp,0)$ and $\BB^{\epp,m+1}(0)=0$. Hence, in the same way as for \eqref{clm+11}, one has 
\begin{align}
 \mathfrak{K}^{3}(\BB^{\epp,m+1},\hat\BB^{\epp,m+1})\le C_2\delta_2\mathfrak{K}^{3}(\BB^{m},\hat\BB^{m})\le \hal \mathfrak{K}^{3}(\BB^{\epp,m},\hat\BB^{\epp,m}).
\end{align}
This implies that the sequence $\{(\bb^{\epp,m},\hhbb^{\epp,m})\}_{m=0}^\infty$ is contractive in the norm $ { \mathfrak{K}^{3}}$ and then converges to the limit $(\bb^\epp,\hhbb^\epp)$, strongly in the norm  $\mathfrak{K}^{3}$, which is a strong solution to the original  problem \eqref{MHDvk4} on $[0,T]$ satisfying \eqref{bbesq}.
The uniqueness of solutions to \eqref{MHDvk4} satisfying \eqref{bbesq} can be obtained by a similar argument as for the contraction.

Note that with the $(\bb^\epp,\hhbb^\epp)$  in hand,   $(b^\epp,\hb^\epp)=(J^\epp)^{-1} \jj^\epp(\bb^\epp,\hhbb^\epp)$ is then the unique solution to \eqref{MHDvk2} on $[0,T]$ satisfying \eqref{kkee}, which follows from \eqref{bbesq}. 
\end{proof}

\begin{rem}\label{reinini}
Since $\eta^\epp$ and $G^\epp$ are smooth, one may employ a standard parabolic regularization argument (see for instance \cite{B2}) to show that  the solution $(b^\eep,\hb^\eep)$ to  \eqref{MHDvk2} constructed in Lemma  \ref{lemma2}  is indeed smooth  for any positive time $t>0$.
\end{rem}

Now we shall derive the uniform estimates of the approximate solutions, independent of the smoothing parameter $\epp>0$, to take the limit as $\epp\rightarrow 0$.
Recall   $\sen(b, \hb )$, and define 
\begin{align}\label{high_diss22}
\sdn(b, \hb ):=& \sum_{j=0}^{2N}\norm{ \dt^j {b}}_{2N-j+1 }^2+ \sum_{j=0}^{2N}\norm{ \dt^j {\hb}}_{2N-j+1 }^2.
\end{align} 
\begin{prop}\label{prop2}
Let $N\ge 4$ be an integer.  Assume that for $0<T\le 1$,  $\sup_{[0,T]}\se{2N}(\eta)<\infty$, $\mathfrak{F}_\infty^{2N}(G)+\mathfrak{F}_2^{2N}(G)<\infty$ and $b_0\in H^{2N+1}(\Omega_-)$ are given such that $\sen(b,\hb)(0)<\infty$ and the $2N$-th order compatibility conditions \eqref{compatibility0j} are satisfied. There exists a universal constant $\delta_3>0$ such that if 
$
\sup_{[0,T]}\se{2N}(\eta)\le \delta_3,
$
then there exists a  unique solution $(b, \hb ) $ to  \eqref{MHDvk01}  on $[0,T]$ satisfying
\begin{align}
 \label{en_ev_2N11}
 \sup_{[0,T]}\sen(b, \hb ) +\int_0^T \sdn(b, \hb )     \ls    \sen(b, \hb )(0)+ \mathfrak{F}_\infty^{2N}(G)+ \mathfrak{F}_2^{2N}(G) .
\end{align}  
\end{prop}
\begin{proof} 
For each $\epp>0$,   let $(\be,\hbe)$ be the   solution  to   the regularized problem  \eqref{MHDvk2} on $[0,  T]$  constructed in Lemma \ref{lemma2}.  As \eqref{MHDvk2} is solved on $[0,T]$,  similarly as Remark \ref{ssrem}, it follows   that there exists  solutions $\hee^\epp$   on $[0,T]$  to the   corresponding regularized electric system in vacuum:
\beq\label{MHDvk22}
\begin{cases}
\curlve\hee^\epp=\De_t\hb^\epp,\quad\divae \hee^\epp=0 &\text{in }\Omega_+
\\
\hee^\epp\times \ne=\(-\kappa\curlve  b^\epp+G^\epp-\Psi^\epp\)\times \ne  &\text{on }\Sigma\\
\hee^\epp_3=0 &\text{on }\Sigma_{+}.
\end{cases}
\eeq
Since $(b^\eep,\hb^\eep)$ is smooth  for $t>0$ (see Remark \ref{reinini}),  $\hee^\epp$ is indeed smooth  for $t>0$.

We now derive the $\epp$-independent estimates of $(\be,\hbe )$ (and $\hee^\epp$). The proof follows  similarly as that of the a priori estimates of \eqref{MHDv}, not involving to the hydrodynamic part, and we provide only the necessary modifications. 

 First,  one may follow the proof of  Proposition \ref{eprop} to deduce
\begin{align}
 \label{en_ev_2Ne3be}
\sen(\hee^\eep)    \ls    \sen(\be,\hbe) + \mathfrak{F}_\infty^{2N}(G^\epp)+ \mathfrak{F}_\infty^{2N}(\Psi^\epp).
\end{align}   
Next,  following the proof of Proposition \ref{evolution_2N}  yields
 \begin{align}\label{en_ev_2Neb}
& \sup_{[0,T]}\seb{2N}(\be,\hbe)(t) +  \int_0^T  \sdb{2N}  (\be)  \nonumber
 \\&\quad \ls       \sen(\be,\hbe)(0)+     \int_0^T \sqrt{ \sen(\eta) } \( \sdn(\be,\hbe )  +\sen(\hee^\eep)\)+ \mathfrak{F}_2^{2N}(G^\epp)+ \mathfrak{F}_2^{2N}(\Psi^\epp).
\end{align} 
Following a variant in the proof of Proposition \ref{v_b_prop} of applying the elliptic theory yields
\begin{align}
 \label{en_ev_2Ne3b}
\sen(\be,\hbe)    \ls     \seb{2N}(\be,\hbe)+  \sen (\eta) \sen(\be,\hbe)+ \mathfrak{F}_\infty^{2N}(G^\epp)+ \mathfrak{F}_\infty^{2N}(\Psi^\epp)
\end{align}  
and
\begin{align}
 \label{en_ev_2Ne4}
\int_0^T \sdn(\be,\hbe)     \ls     \int_0^T \sdb{2N} (\be) +\int_0^t\sen(\eta)  \sdn(\be,\hbe) + \mathfrak{F}_2^{2N}(G^\epp) + \mathfrak{F}_2^{2N}(\Psi^\epp),
\end{align}  
Since $ \delta_3$ is small,  combining \eqref{en_ev_2Ne3be}--\eqref{en_ev_2Ne4} yields
 \begin{align}\label{en_ev_2Neb22}
& \sup_{[0,T]}\sen(\be,\hbe )  +  \int_0^T  \sdn (\be,\hbe )
 \nonumber\\&\quad \ls       \sen(\be,\hbe)(0) + \mathfrak{F}_\infty^{2N}(G^\epp)+\mathfrak{F}_2^{2N}(G^\epp)+ \mathfrak{F}_\infty^{2N}(\Psi^\epp)+  \mathfrak{F}_2^{2N}(\Psi^\epp).
\end{align} 
It follows from Lemma \ref{covrem} and \eqref{en_ev_2Neb22} that
 \begin{align}\label{en_ev_2Neb2233}
\sup_{[0,T]}\sen(\be,\hbe ) +  \int_0^T  \sdn (\be,\hbe )
  \ls \sen(b, \hb )(0)+ \mathfrak{F}_\infty^{2N}(G)+ \mathfrak{F}_2^{2N}(G)+\eep.
\end{align}  
The estimate  \eqref{en_ev_2Neb2233} allows one to conclude that as $\eep\rightarrow 0$, up to extraction of a subsequence,   the sequence $(\be,\hbe )$ converges to a limit $(b, \hb )$ in
the weak or weak-$\ast$ sense of the norms in the left hand side of  \eqref{en_ev_2Neb2233}, which makes it possible to take  the limit in \eqref{MHDvk2}  to find  that $(b, \hb ) $ solves   \eqref{MHDvk01}. The estimate \eqref{en_ev_2N11} follows from the the weak   lower
semicontinuity of the left hand side of \eqref{en_ev_2Neb2233} and  passing to the limit in the right hand side. 
\end{proof}

\subsection{Sequence of approximate solutions}
The  solution  to  the problem \eqref{MHDv} will be obtained by the method of successive approximations.  The sequence of approximate solutions, $\{(u^m,p^m,\eta^m,b^m,\hb^m )\}_{m=0}^\infty$, is constructed as follows. First, one constructs $(u^0,b^0,\eta^0)$ achieving the initial data. Second,  assuming that $(u^m,b^m,\eta^m)$ for $m\ge 0$  achieves the initial data and satisfies suitable estimates to be specified later, we define $(u^{m+1},p^{m+1},\eta^{m+1},b^{m+1},\hb^{m+1})$ as the solution to
\beq\label{mkgeometric}
\begin{cases}
\Dml_t   {u} ^{m+1} +    {u}^{m+1}\cdot \Dnml  {u}^{m+1}+\Dnml   p^{m+1} =   \curlvm   b^{m}\times (\bar B+b^{m})  & \text{in } \Omega_- 
\\ \divaml    u^{m+1}=0  &\text{in }\Omega_- 
\\ \Dm_t  {b}^{m+1}+\kappa\curlvm  \curlvm   {b}^{m+1} =\curlvm (u^{m+1}\times (\bar B+b^m))& \text{in } \Omega_- 
\\ \divam b^{m+1} = 0  &\text{in }\Omega_- 
\\\curlvm\hb^{m+1} =0,\quad\divam \hb^{m+1} =0 &\text{in }\Omega_+ 
\\ \partial_t \eta^{m+1} = u^{m+1}\cdot\nml & \text{on }\Sigma
\\  p^{m+1}=  - \sigma H^{m+1},\quad \jump{b^{m+1}}=0  & \text{on } \Sigma
\\ u_3^{m+1} =0,\quad b_3^{m+1}=0,\quad \kappa \curlvm   {b}^{m+1}\times e_3=(u^{m+1}\times (\bar B+b^m))\times e_3 & \text{on } \Sigma_{-}
\\   b^{m+1}\times e_3=0 & \text{on } \Sigma_{+}
\\ (u^{m+1},b^{m+1},\eta^{m+1})\mid_{t=0}= (u_0,b_{0},\eta_0).
\end{cases}
\eeq  
Here $\varphi^m=\varphi(\eta^m)$, $\n^m=(-\nabla_h\eta^m,1)$ and $H^m=H(\eta^m)$ as in \eqref{HHdef}.

This construction and the corresponding estimates are recorded in the following proposition.

\begin{prop}\label{l_iteration}
There exist universal positive constants $\tilde \delta_1$ and  $T_1$ such that if $\sen (0) \le \tilde \delta_1$ and $0<T\le T_1$,  then   there exists a sequence $\{(u^m,p^m,\eta^m,b^m,\hb^m  )\}_{m=0}^\infty$ that  solves \eqref{mkgeometric} on $[0,T]$ and satisfies the following estimates
\begin{equation}\label{l_it_03}
\sup_{[0,T]}\sen(u^m,p^m,\eta^m,b^m,\hb^m  )+\int_0^{T}  \sdn(b^m,\hb^m )   \ls\sen(0).
\end{equation}
\end{prop}

\begin{proof}
First, extend the initial data $(\dt^j u(0),\dt^j b(0),\dt^j \eta(0)),\ j=0,\dots,2N-1,$ to the time-dependent functions $(u^0,b^0,\eta^0) $  so that $\(\dt^j u^0(0) ,\dt^j b^0(0),\dt^j \eta^0(0) \)= (\dt^j u(0),\dt^j b(0),\dt^j \eta(0))$. This can be done by applying Lemmas \ref{l_sobolev_extension}--\ref{l_sobolev_extension2}, and one has in particular that
\begin{equation}\label{l_iT_1}
\sup_{[0,\infty]}\sen(u^0,b^0,\eta^0)+\int_0^{\infty}    \sdn(b^0) \le C_1
\sen(0).
\end{equation}

Next,  we claim that there exist   $\gamma_1,\gamma_2>0$ and  $T>0$ such that   if $(u^m,b^m,\eta^m)$ achieves the initial data and satisfies 
\begin{equation}\label{claim0}
\sup_{[0,T]}\se{2N}(u^m,\eta^{m})\le \gamma_1 \sen(0) 
\eeq
and
\beq\label{claim00}
\sup_{[0,T]}\se{2N}(b^m)+\int_0^{T}    \sdn(b^{m})\le  \gamma_2\sen(0) ,
\eeq
then there exists a unique solution   $(u^{m+1},p^{m+1},\eta^{m+1},b^{m+1},\hb^{m+1} )$ to \eqref{mkgeometric} on $[0,T]$ satisfying
\begin{equation} \label{claim123} 
\sup_{[0,T ]}\sen(u^{m+1},p^{m+1}, \eta^{m+1})\le \gamma_1\sen(0) 
\end{equation} 
and
\begin{equation} \label{claim1234} 
\sup_{[0,T ]}\sen( b^{m+1} ,\hb^{m+1} )+\int_0^{T } \sdn(b^{m+1},\hb^{m+1} )   \le \gamma_2\sen(0).
\end{equation} 
To prove the claim, one may first use  $( b^{m},\eta^{m})$ to construct  $(u^{m+1},p^{m+1}, \eta^{m+1})$ as the solution to \eqref{MHDvk1} with 
$
 F  =\curl^{  \varphi^m} b^{m} \times (\bar B+  b^{m}).
$
Recall the notations \eqref{nof1} and \eqref{nof3}. Note that
\beq\label{claim0000}
 \mathfrak{F}_0^{2N}(F)\ls  \(1+\sen(b^m,\eta^m)(0)\)    \sen(b)(0)\ls \sen(0) ,
 \eeq
 and by \eqref{claim0} and \eqref{claim00},
 \beq\label{claim00000}
 \mathfrak{F}_2^{2N}(F)\ls  \int_0^{T} \(1+\se{2N}(\eta^m,b^{m})\)  \sdn(b^{m})\ls \(1+(\gamma_1+\gamma_2) \sen(0)\) \gamma_2 \sen(0).
 \eeq
Hence,
\beq \mathfrak{F}_0^{2N}(F)+   T\mathfrak{F}_2^{2N}(F)  \le C_2\(    1+T\gamma_2 \(1+(\gamma_1+\gamma_2) \sen(0)\) \) \sen(0) .
\eeq
If $\tilde \delta_1\le \min\{1/(\gamma_1+\gamma_2),\delta_1/\(1+C_2(1+2 \gamma_2)\)\}$ with $\delta_1$ given in Proposition \ref{lemma1}, then  
\begin{align}\label{claes}
  \se{2N}(0)+  \mathfrak{F}_0^{2N}(F)+   T\mathfrak{F}_2^{2N}(F)   
  &\le \(  1+C_2\(    1+T\gamma_2 \(1+(\gamma_1+\gamma_2) \sen(0)\) \)  \)\sen(0)\nonumber
  \\& \le  \(1+C_2\(   1  + 2T \gamma_2  \)\)\sen(0)\le \delta_1,\ \forall T\le 1.
\end{align}
Hence, Proposition \ref{lemma1} guarantees the existence of a unique $(u^{m+1},p^{m+1},\eta^{m+1})$ solving \eqref{MHDvk1}  on $[0,T]$ for any $ T\le 1$. Moreover, \eqref{bound222} and \eqref{claes} imply 
\begin{align} \label{en_ev_2N22m22}
 \sup_{[0,T]}\se{2N}(u^{m+1},p^{m+1},\eta^{m+1})
 \le C_3 C_2\(   1  + 2T \gamma_2  \)\sen(0).
\end{align}
If $\gamma_1\ge 2C_3C_2$ and $T\le 1/(2\gamma_2)$, then \eqref{en_ev_2N22m22} yields \eqref{claim123}.

Now one uses $(u^{m+1},b^m, \eta^m)$ to construct  $(b^{m+1},\hb^{m+1} )$  as the solution to \eqref{MHDvk01} with    $\varphi=\varphi^m$ and $G =u^{m+1}\times (\bar B+b^m)$. It follows from \eqref{claim123} and \eqref{claim00} that for $T>0$ in the above,
\begin{align}\label{claes2} 
\mathfrak{F}_\infty^{2N}(G)+\mathfrak{F}_2^{2N}(G)\ls (1+T)  \sup_{[0,T]}\se{2N}( u^{m+1})\(1+\se{2N}(b^m)\)\ls \gamma_1 \sen(0) \(1+\gamma_2 \sen(0)\).
\end{align}
If $\tilde \delta_1\le \delta_3/\gamma_1$ with $\delta_3$ given in Proposition \ref{prop2} and hence by \eqref{claim0}, $\se{2N}( \eta^m)\le\delta_3$,  then Proposition \ref{prop2} guarantees the existence of a unique $(b^{m+1},\hb^{m+1} )$ solving \eqref{MHDvk01} on $[0,T]$. In addition, \eqref{en_ev_2N11} and \eqref{claes2} imply
\begin{align}
 \label{en_ev_2N22m}
 \sup_{[0,T ]}\sen( b^{m+1} ,\hb^{m+1} )+\int_0^{T } \sdn(b^{m+1},\hb^{m+1} )       
 \le C_4\(    \sen(0)+\gamma_1 \sen(0) \(1+\gamma_2 \sen(0)\)\) .
\end{align} 
Hence if  $\tilde \delta_1\le 1/\gamma_2$ and $ \gamma_2\ge C_4(1+2\gamma_1)$, then \eqref{en_ev_2N22m} yields \eqref{claim1234}. 

Consequently,  the claim is proved. Note that  if  $ \gamma_1,\gamma_2\ge C_1$, then \eqref{l_iT_1} implies that \eqref{claim0} and \eqref{claim00}  hold for $m=0$. Hence, one can then iterate from $m=0$ to construct the sequence   $\{(u^m,p^m,\eta^m,b^m,\hb^m  )\}_{m=0}^\infty$ satisfying the conclusions.
\end{proof}

Now we prove the contraction of the sequence $\{(u^m,p^m,\eta^m,b^m,\hb^m )\}_{m=0}^\infty$. For $m\ge 1$, define
\beq 
\begin{split}
&U^m= u^m-u^{m-1},\ Q^m= p^m-p^{m-1},\ \zeta^m= \eta^m-\eta^{m-1},
\\
 & B^m= b^m-b^{m-1},\ \hat B^m= \hb^m-\hb^{m-1}. 
 \end{split}
\eeq
Then it follows from \eqref{mkgeometric} that 
\beq\label{MHDvkm1}
\begin{cases}
\Dm_t   {U}^{m+1}  +    {u}^{m}\cdot \Dnm  {U}^{m+1}+\Dnm   Q^{m+1}   = \F^{1,m} & \text{in } \Omega_- 
\\ \divam   U^{m+1} = \F^{2,m}  &\text{in }\Omega_- 
\\ \Dm_t  {B}^{m+1}  +\kappa\curlvm \curlvm B^{m+1}     = \F^{3,m}  & \text{in } \Omega_- 
\\ \divam  B^{m+1}  = \F^{4,m}    &\text{in }\Omega_- 
\\\curlvm\hat B^{m+1} =\hat{\F}^{3,m} ,\quad\divam \hat B^{m+1} =\hat{\F}^{4,m}  &\text{in }\Omega_+ 
\\ \partial_t \zeta^{m+1}  = U^{m+1}\cdot\nm + \F^{5,m} & \text{on }\Sigma
\\  Q^{m+1}= - \sigma \Delta_h\zeta^{m+1}+\F^{6,m} , \quad  \jump{B^{m+1} }=0   & \text{on } \Sigma
\\ U^{m+1}_3 =0,\quad   B^{m+1}_3=0,\quad  \curlvm  B^{m+1}\times e_3 =\F^{7,m} & \text{on } \Sigma_{-}
\\  \hat B^{m+1}\times e_3=0      & \text{on } \Sigma_{+}
\\ (U^{m+1},B^{m+1},\zeta^{m+1})\mid_{t=0}= (0,{0},0),
\end{cases}
\eeq

where 
\begin{align}
F^{1,m} =&  -(\Dml_t -\Dm_t)  {u}^{m+1} -    \({u}^{m+1}\cdot \Dnml-{u}^{m}\cdot \Dnm\)  {u}^{m+1} -\(\Dnml -\Dnm  \)p^{m+1}\nonumber 
\\&+\curlvm   b^{m}\times (\bar B+b^{m})-\curlvmll  b^{m-1}\times (\bar B+b^{m-1}) ,  
\\ \F^{2,m} =&-\(\divaml-\divam\)    u^{m+1} ,
\\ \F^{3,m}  = &
 -\(\Dm_t-\Dmll_t\)  {b}^{m+1} -\kappa \(\curlvm \curlvm -\curlvmll \curlvmll \)b^{m+1}  \nonumber
 \\&+ \curlvm u^{m+1}\times (\bar B+b^m) -\curlvmll u^{m}\times (\bar B+b^{m-1})  ,
\\ \F^{4,m} = &-\(\divam-\divamll\)  b^{m+1}   ,
\\ \hat{\F}^{3,m} = &-\(\curlvm-\curlvmll\)  \hb^{m+1}   ,
\\ \hat{\F}^{4,m} = &-\(\divam-\divamll\)  \hb^{m+1}   , 
\\ \F^{5,m} =& u^{m+1}\cdot\(\nml-\nm\) ,
\\  \F^{6,m} =  &-\sigma \diverge_h  \(\! \frac{\nabla_h\eta^{m+1}}{\sqrt{1+|\nabla_h\eta^{m+1}|^2}}-\nabla_h\eta^{m+1}\!\)+\sigma \diverge_h  \(\! \frac{\nabla_h\eta^{m}}{\sqrt{1+|\nabla_h\eta^{m}|^2}}-\nabla_h\eta^{m}\!\),  
\\   \F^{7,m}  =&\(\kappa(-   \curlvm   +  \curlvmll)  b^{m}+u^{m+1}\times (\bar B+b^m)\ -u^{m}\times (\bar B+b^{m-1}) \)\times e_3  .  
\end{align} 

The contraction of the sequence in the lower-order energy, say, $\se{N}$, is given as follows.

\begin{prop}\label{l_iteration22}
There exist  universal positive constants $\tilde \delta_2$ $(\le \tilde\delta_1)$ and $T_2$ $(\le T_1)$ such that if $\sen (0) \le \tilde \delta_2$ and $T\in (0, T_2]$, 
then it holds that
\begin{align} \label{l_it_033}
& \sup_{[0,T]}\se{N}(U^{m+1},Q^{m+1},\zeta^{m+1},B^{m+1},\hat B^{m+1} )+\int_0^{T}   \sdn(B^{m+1},\hat B^{m+1} )\nonumber
 \\&\quad\le   \hal\(\sup_{[0,T]}\se{N}(U^m,\zeta^m,B^m)+\int_0^{T}   \sdn(B^{m})\).
\end{align}
\end{prop}
\begin{proof}
By \eqref{l_it_03},  one has the uniform smallness that 
\beq
\se{2N}(u^m,p^m,\eta^m,b^m,\hb^m )\ls\sen(0)\le \tilde \delta_2.
\eeq
Then following the proof of Proposition \ref{lemma1} with slight modifications, one may conclude that
\begin{align} \label{boundmmm1}
 \sup_{[0,T]}\se{N}(U^{m+1},Q^{m+1},\zeta^{m+1})
& \ls  \tilde \delta_2 \sup_{[0,T]}\se{N}(U^{m+1},Q^{m+1},\zeta^{m+1},\zeta^m,B^m)
\\&\quad +   \sqrt{\sup_{[0,T]}\sen(U^{m+1},\zeta^{m+1})T\int_0^{T}   \mathbb{D}_N(B^{m})} +T\int_0^{T}   \mathbb{D}_N(B^{m}).\nonumber
\end{align} 
On the other hand, modifying the proof of Proposition \ref{prop2} slightly yields
\begin{align} \label{boundmmm2}
&\sup_{[0,T ]}\se{N}(B^{m+1},\hat B^{m+1} ) +\int_0^{T}   \mathbb{D}_N(B^{m+1},\hat B^{m+1} )
\\&\quad\ls \sup_{[0,T]} \se{N}( U^{m+1}) +\tilde \delta_2\sup_{[0,T]} \se{N}(U^{m+1},B^{m+1},\hat B^{m+1},\zeta^m,B^m) +\tilde \delta_2 \int_0^{T}   \mathbb{D}_N(B^{m+1},\hat B^{m+1},B^{m}). \nonumber
\end{align} 
Hence, combining \eqref{boundmmm1} and \eqref{boundmmm2} shows, by the smallness of $\tilde \delta_2$ and Cauchy's inequality,
\begin{align} \label{ccc1}
 &\sup_{[0,T]}\se{N}(U^{m+1},Q^{m+1},\zeta^{m+1},B^{m+1},\hat B^{m+1} )+\int_0^{T}   \mathbb{D}_N(B^{m+1},\hat B^{m+1} ) \nonumber
\\ &\quad\le  C_1 \(\tilde \delta_2 \sup_{[0,T]}\(\se{N}(U^m,\zeta^m,B^m)\)+ (\tilde \delta_2+T)\int_0^{T}   \mathbb{D}_N(B^{m})\).
\end{align}

Consequently, if $\tilde \delta_2\le 1/(4C_1)$ and $T\le 1/(4C_1)$, then \eqref{ccc1} implies in particular \eqref{l_it_033}.
\end{proof}

\subsection{Local well-posedness of \eqref{MHDv}} 
We are now ready to state the local well-posedness result for \eqref{MHDv}.
\begin{thm}\label{main_thm2}
Assume that $\kappa>0$ and $\sigma>0$ and let $N\ge 4$  be an integer. Assume that $u_0\in H^{2N}(\Omega)$, $b_0\in H^{2N+1}(\Omega)$  and $\eta_0\in H^{2N+3/2}(\Sigma )$ are given such that $\se{2N} (0)<\infty$ and that  the  compatibility conditions \eqref{compatibility}  are satisfied.
There exist universal positive constants $\delta_0$ and $T_0$ such that if $\se{2N} (0) \le \delta_0$ and $0<T\le T_0$,  then there exists a  unique solution $(u,p,\eta,b, \hb )$  to  \eqref{MHDv} on $[0,T]$ satisfying
\begin{equation}\label{bound110}
\sup_{[0,T]}\se{2N}+\int_0^{T}  \sdn(b,\hb)\le \tilde C_2\se{2N}(0).
\end{equation}
\end{thm}
\begin{proof}
Take  $\delta_0=\tilde \delta_2$ and $T_0=T_2$ as given in Proposition  \ref{l_iteration22}. Let $\{(u^m,p^m,\eta^m,b^m,\hb^m )\}_{m=1}^\infty$ on $[0,T]$ with $0<T\le T_0$ be  the sequence constructed in Proposition \ref{l_iteration} that solves \eqref{mkgeometric}.
The uniform estimate  \eqref{l_it_03} implies that as $m\rightarrow \infty$, up to extraction of a subsequence, the sequence $(u^m,p^m,\eta^m,b^m,\hb^m  )$ converges to a limit $(u,p,\eta,b, \hb )$ in the weak or weak-$\ast$ sense of the norms in  \eqref{l_it_03}. Then the weak lower semicontinuity shows that $(u,p,\eta,b, \hb )$ satisfies the estimate \eqref{bound110}. On the other hand, the contractive estimate \eqref{l_it_033} 
  shows that the whole sequence  $(u^m,p^m, \eta^m,b^m,\hb^m )$ converges strongly to the limit
$(u,p,\eta,b, \hb )$   in the  norms of $\se{N}$, which is sufficient for passing to the limit in  \eqref{mkgeometric}.  Then one finds that the limit $(u,p,\eta,b,\hb)$ is a
strong solution to \eqref{MHDv} on $[0,T]$. The
uniqueness of   solutions to \eqref{MHDv} 
satisfying \eqref{bound110} follows  similarly as in the proof of the contraction.
\end{proof}

\begin{rem}
It is possible to remove the smallness assumption of $u_0$ and $b_0$ in Theorem  \ref{main_thm2}  by restricting the local existence time of the solution to be smaller, depending on the initial data, see Padula and Solonnikov \cite{PS}. However, much more  work is required if one would like to relax the smallness of $\eta_0$, owing to the way of solving the magnetic part; yet the local well-posedness recorded in Theorem  \ref{main_thm2} is sufficient to be adapted in the proof of our main theorem of global well-posedness.
\end{rem}
%%%%%%%%%%%%%%%%%%%%%%%%%%%%%%%%%%%%%%%%%%%%%%%%%%%%%%

\section{Global well-posedness}\label{sec_gwp}
%%%%%%%%%%%%%%%%%%%%%%%%%%%%%%%%%%%%%%%%%%%%%%%%%%%%%%

In this section we   prove the global well-posedness of \eqref{MHDv} as follows.
\begin{proof}[Proof of Theorem \ref{main_thm}]
Recall the constants $\tilde\delta$ and $\tilde C_1$ in Theorem \ref{apriorith} and $\delta_0$, $T_0$ and $\tilde C_2$ in Theorem \ref{main_thm2}, and, without loss of generality,   assume $\tilde C_1,\tilde C_2\ge 1$.
Assume that $\se{2N} (0) \le \varepsilon_0$ for
\beq\label{Edef}
\varepsilon_0 =\min\left\{\frac{\delta_0} {\tilde C_1},\frac{\tilde\delta}{2\tilde C_1\tilde C_2}\right\}.
\eeq
Set
\beq\label{Tdef}
T^\ast:=\sup_T\left\{T>0\left|\begin{array}{lll} \text{There exists a unique solution  to  \eqref{MHDv} on }[0,T]\\\text{satisfying }\sup_{[0,T]}\se{2N}  \le 2\tilde C_1\tilde C_2 \varepsilon_0.
\end{array}\right.\!\!\right\}.
\eeq
It follows from Theorem \ref{main_thm2} that $T^\ast\ge T_0>0$. We will show that $T^\ast=\infty$ by contradiction. Assume that $T^\ast<\infty$. Then for any $0<T<T^\ast$, it follows from \eqref{Tdef} and \eqref{Edef} that
\beq
\sup_{[0,T]}\se{2N}  \le 2\tilde C_1\tilde C_2 \varepsilon_0\le \tilde \delta.
\eeq
Then Theorem \ref{apriorith} shows
\beq
\sup_{[0,T]}\se{2N}  \le  \tilde C_1\se{2N}(0) \le   \tilde C_1\varepsilon_0\le \delta_0.
\eeq
Now taking the $T$ above as the initial time, one can apply Theorem  \ref{main_thm2}  again  to find that there is a unique solution  to  \eqref{MHDv}  on $[T,T+T_0]$ satisfying
\beq
\sup_{[T,T+T_0]}\se{2N}  \le  \tilde C_2\se{2N}(T) \le   \tilde C_1\tilde C_2 \varepsilon_0.
\eeq
This contradicts the definition of $T^\ast$, and so $T^\ast=\infty $. Hence, the solution is global and the estimates \eqref{thm_en1} and \eqref{thm_en2} follow  from Theorem \ref{apriorith}. 
 \end{proof}

%%%%%%%%%%%%%%%%%%%%%%%%%%%%%%%%%%%%%%%%%%%%%%%%%%%%%%
 \section{ The  plasma-plasma interface problem}\label{sec_pp}
 %%%%%%%%%%%%%%%%%%%%%%%%%%%%%%%%%%%%%%%%%%%%%%%%%%%%%%

 In this section we will present a similar global well-posedness for  the plasma-plasma interface problem for the incompressible inviscid and resistive MHD, where the two immiscible plasmas occupy the two regions $\Omega_\pm(t)$ respectively. Assume that the velocities $u_\pm$, the pressures $p_\pm$ and the magnetic fields $B_\pm$ satisfy the following problem:
\beq\label{ppMHD}
\begin{cases}
\partial_t {u}_\pm  +    {u}_\pm\cdot \nabla  {u} _\pm+\nabla p_\pm =  \curl B_\pm\times B_\pm& \text{in } \Omega_\pm(t)
\\ \Div u_\pm=0 &\text{in }\Omega_\pm(t)
\\ \dt B _\pm       =  \curl   E_\pm,\quad  E_\pm= u_\pm\times B_\pm-\kappa_\pm\curl  B_\pm& \text{in } \Omega_\pm(t)
\\ \Div B_\pm=0 &\text{in }\Omega_\pm(t)
\\ \dt\eta=u_\pm\cdot\n & \text{on }\Sigma(t)
\\  p_+=p_-+\sigma H,\quad B_+=B_-,\quad E_+\times \n =E_-\times \n  & \text{on } \Sigma(t)
\\ u_+\cdot e_3 =0,\quad B_+\times e_3 =\bar B\times e_3 & \text{on } \Sigma_{+} 
\\ u_-\cdot e_3 =0,\quad B_-\cdot e_3=\bar B\cdot e_3,\quad E_-\times e_3 =0 & \text{on } \Sigma_{-}. 
\end{cases}
\eeq
Here  $E_\pm$  are the electric fields of the plasmas and  $\kappa_\pm>0$ are the magnetic diffusion coefficients. Note that we have set the problem to be around the traversal magnetic field $\bar B$ with $\bar B_3\neq 0$.
 
Use again the mapping $\Phi$ (cf. \eqref{diff_def}) to transform the problem \eqref{ppMHD}   to the one in  $\Omega_\pm$. Since the domains $\Omega_\pm$ are fixed, one may simplify notation by writing $f$ to refer to $f_\pm$. In the following,  an equation for $f$ on $\Omega$ means that  the equation holds for $f_\pm$ on $\Omega_\pm$,  and an equation involving $f$ on $\Sigma$ means that the equation holds for both  $f=f_+$ and $f=f_-$. For a quantity $f=f_\pm$,   the interfacial jump on $\Sigma$ is defined as
\begin{equation}
\jump{f} := f_+ \vert_{\Sigma} - f_- \vert_{\Sigma}.
\end{equation} 
Then the problem \eqref{ppMHD}  is equivalent to the following problem for $(u,p,\eta,b,\hb)$, with $b= B-\bar B$,
\beq\label{ppMHDv}
\begin{cases}
\D_t   {u}  +    {u}\cdot \Dn  {u} +\Dn   p =   \curlv   b\times (\bar B+b)  & \text{in } \Omega 
\\ \diva    u=0  &\text{in }\Omega 
\\ \D_t  {b}        =  \curlv   E ,\quad   E=  u\times (\bar B+b) -\kappa\curlv  b & \text{in } \Omega 
\\ \diva  b = 0  &\text{in }\Omega 
\\ \partial_t \eta  = u\cdot\N & \text{on }\Sigma
\\  \jump{p}=\sigma H,\quad \jump{b}=0,\quad \jump{E}\times \n =0  & \text{on } \Sigma 
\\  u_3 =0 ,\quad b \times e_3=0  & \text{on } \Sigma_{+}
\\  u_3 =0 ,\quad b_3 =0,\quad E\times e_3 =0 & \text{on } \Sigma_{-}
\\(u,b,\eta)\mid_{t=0}= (u_0,b_{0},\eta_0).
\end{cases}
\eeq
As the pressure $p$ is uniquely determined up to constants (constant in space), to guarantee the uniqueness of $p$ and without loss of generality, one may require $\int_\Omega p =0.$

Given the initial data $(u_0,b_0, \eta_0)$, one needs to use the equations \eqref{ppMHDv} to construct the data $\partial_t^j u(0)$ and $\partial_t^j b(0)$ for $j=1,\dotsc,2N$, $\partial_t^j  p(0)$ for $j =0,\dotsc, 2N-1$ and $\partial_t^j \eta(0)$ for $j=1,\dotsc,2N+1$. These data need to satisfy the  $2N$-th order  compatibility conditions:
\beq\label{ppcompatibility}
\begin{cases}
   \diverge^{\varphi_0}  u_0  =0\text{ in }\Omega, \quad\jump{ u_0}\cdot\n_0  =0\text{ on }\Sigma,\quad u_{0,3} =0\text{ on }\Sigma_\pm,
   \\   \diverge^{\varphi_0}  b_0  =0\text{ in }\Omega,\quad\jump{ b_0}  =0\text{ on }\Sigma, \quad b_{0}\times e_3 =0\text{ on }\Sigma_+,\quad b_{0,3} =0\text{ on }\Sigma_-,
   \\   \jump{\dt^j b(0)}\times \n_0 =0\text{ on }\Sigma \text{ and }  \dt^j b(0) \times e_3 =0\text{ on }\Sigma_+,\ j=1,\dots,2N-1,
      \\  \dt^j \( \jump{E }\times \n\)(0)=0\text{ on }\Sigma \text{ and } \dt^j E(0)\times e_3=0\text{ on }\Sigma_- ,\ j=0,\dots,2N-1.
 \end{cases}
\eeq	 

For $f =f_\pm$, denote $\ns{f}_k = \ns{f_+}_{H^k(\Omega_+)} + \ns{f_-}_{H^k(\Omega_-)} $ and $\as{f}_s = \ns{f_+}_{H^k(\Sigma)} + \ns{f_-}_{H^k(\Sigma)}$. For an integer $N\ge 4$, we define the high-order energy as
\begin{align}\label{pphigh_en}
\se{2N}:=&\sum_{j=0}^{2N}\norm{  \dt^j{u}}_{2N-j}^2 +\as{ \dt^{2N}{u}}_{-1/2}+\sum_{j=0}^{2N-1}\norm{ \dt^j {b}}_{2N-j+1 }^2+\ns{  \dt^{2N} b}_{0} \nonumber
\\&+\sum_{j=0}^{2N-1}\norm{  \dt^j p}_{2N-j}^2 + \sum_{j=0}^{{2N}-1}\abs{  \dt^j\eta}_{{2N}-j+3/2}^2+\abs{  \dt^{2N}\eta}_{1}^2+\abs{  \dt^{{2N}+1}\eta}_{-1/2}^2.
\end{align}
For  $n=N+4,\dots,2N$, we define a set of energies as
\begin{align} \label{pplow_en}
\fe{n}:=&\norm{   u}_{n-1}^2+\norm{   u}_{0,n}^2 +\sum_{j=1}^{n}\norm{\dt^j u}_{n-j}^2 +\norm{ b}_{n}^2+\sum_{j=1}^{n-1}\norm{ \dt^j {b}}_{n-j+1 }^2+ \norm{ \dt^n{b}}_{0}^2  \nonumber
\\&+\sum_{j=0}^{n-1}\norm{  \dt^j p}_{n-j}^2+ \sum_{j=0}^{{n}-1}\abs{  \dt^j\eta}_{{n}-j+3/2}^2+\abs{  \dt^{n}\eta}_{1}^2+\abs{  \dt^{{n}+1}\eta}_{-1/2}^2
\end{align}
and the corresponding dissipations as
\begin{align} \label{pplow_diss}
\fd{n}:=&  \sum_{j=0}^{n-1} \norm{\dt^j u}_{n-j-1}^2+  \sum_{j=0}^{n-2}\norm{\dt^j b}_{n-j }^2 +\sum_{j=0}^{n}\norm{\dt^j b}_{1,n-j}^2        \nonumber
\\& +\sum_{j=0}^{n-2}\norm{  \dt^{j} p }_{n-j-1}^2 +\sum_{j=0}^{n-2}\abs{  \dt^{j} \eta}_{n-j+1/2}^2+\abs{\dt^{n-1}\eta}_{1}^2 +\abs{\dt^{n}\eta}_{0}^2.
\end{align}

Now the main results of this section are stated as follows.
\begin{thm}\label{ppmain_thm}
Assume that $\kappa>0,\ \bar B_3\neq 0$ and $\sigma>0$ and let $N\ge 8$ be an integer. Let $u_0\in H^{2N}(\Omega)$,
 $b_0\in H^{2N+1}(\Omega)$  and $\eta_0\in H^{2N+3/2}(\Sigma )$ be given such that $\se{2N} (0)<\infty$ and    the compatibility conditions \eqref{ppcompatibility}  as well as the zero average condition \eqref{zero_0v} are satisfied.
There exists a universal constant $\varepsilon_0>0$ such that if $\se{2N} (0) \le \varepsilon_0$, then there exists a global unique solution $(u,p,\eta,b)$ to  \eqref{ppMHDv}. Moreover, for all $t\ge 0$,
\beq\label{ppthm_en1}
 \se{2N} (t) + \int_0^t \fd{2N} (s)\,ds \ls \se{2N} (0)
\eeq
and
\beq\label{ppthm_en2}
\sum_{j=0}^{N-6} (1+t)^{N-5-j}\fe{N+4+j}(t) + \sum_{j=0}^{N-6}\int_0^t(1+s)^{N-5-j}\fd{N+4+j}(s)\,ds\ls \se{2N} (0)  .
\eeq
\end{thm}
\begin{proof}
The main approach is similar to that of Theorem \ref{main_thm}. We will not repeat all the details of the proof but sketch only the main differences.

For the a priori estimates, the main difference here lies in deriving the energy evolution estimates of the highest order time derivative of the solution to  \eqref{ppMHDv}. Indeed, compared to the plasma-vacuum  interface problem, the new and most technical difficulty is the control of the following nonlinear terms for the hydrodynamic part: 
\begin{align}\label{ppintro_11}
&-\int_{\Sigma} \dt^{2N-1} p_- \dt^{2N+1}\jump{ u}\cdot   \N\nonumber
\\&\quad=\int_{\Sigma} \dt^{2N-1} p_-\(   (2N+1)\dt^{2N}\jump{ u}\cdot   \dt \N+  \jump{u}\cdot \dt^{2N+1} \N\)+{\sum}_{\mathcal{R}},
\end{align}
where one used the fact that $\jump{u}\cdot\N=0$ on $\Sigma$ by the fifth equation in  \eqref{ppMHDv}.
The first term in the right hand side of \eqref{ppintro_11} is controlled by  using $\as{\dt^{2N} u}_{-1/2}\ls \se{2N}$ (cf. \eqref{pphigh_en}), while the treatment of the second term is much more involved: one needs  to introduce the ``material" derivative  $D_t^-=\D_t+U_-\cdot \nav$, where $U_-$ is an extension of $u_-$ onto $\Omega$, and the key point lies in that when considering $D_t^-\dt^{2N-1}$ instead of $\dt^{2N} $, certain troublesome nonlinear terms will be canceled. One may refer to Cheng, Coutand and Shkoller \cite{CCS1} for the details. With this, the rest of the  estimates can be derived more or less similarly as those for the plasma-vacuum interface problem \eqref{MHDv}.

For the construction of local solutions, we also decompose  \eqref{ppMHDv} into the hydrodynamic part  and the magnetic part. The hydrodynamic part is the two-phase incompressible Euler equations with surface tension, which can be solved in spirit of Cheng, Coutand and Shkoller \cite{CCS1}. For the other part of the two-phase magnetic system in moving domains, we can employ an  argument different from the one for \eqref{MHDv}. Indeed, by using again the change of unknown:
\beq\label{ppbtr1}
\bb=J \jj^{-1} b,
\eeq
for $\eta$ small we can solve the transformed system of $b$ in our energy functional framework by using directly the Galerkin method (recall the reason that we resorted to a smoothing argument for solving \eqref{MHDv} is merely due to that $\curl \hhbb\neq 0$ in $\Omega_+$; the fact that $\curl \bb\times e_3\neq 0$ on $\Sigma_-$  is indeed not an obstacle, see Guo and Tice \cite{GT1} for the study of the incompressible Navier-Stokes equations). However, the validity of \eqref{ppbtr1} requires that $\eta$ is of half regularity index higher than $b$, which is not the case for the pair of $\dt^{2N}\eta$ and $\dt^{2N}b$ in our energy functional setting. Our way to overcome this difficulty is to  regularize the hydrodynamic part by considering the following approximate problem
\beq\label{ppMHDvk}
\begin{cases}
\D_t   {u}  +    {u}\cdot \Dn  {u}-(\epsilon\Lambda\bar\eta-\epsilon\Lambda \Psi)\D_3u+\Dn   p =   \curlv   b\times (\bar B+b)  & \text{in } \Omega 
\\ \diva    u=0  &\text{in }\Omega 
\\ \D_t  {b}        =  \curlv   E ,\quad   E=  u\times (\bar B+b) -\kappa\curlv  b & \text{in } \Omega 
\\ \diva  b = 0  &\text{in }\Omega 
\\ \partial_t \eta +\epsilon\Lambda\eta-\epsilon\Lambda\Psi = u\cdot\N & \text{on }\Sigma
\\  \jump{p}=   \sigma H,\quad \jump{b}=0,\quad \jump{E}\times \n =0  & \text{on } \Sigma
\\  u_3 =0 ,\quad b \times e_3=0  & \text{on } \Sigma_{+}
\\  u_3 =0 ,\quad b_3 =0,\quad E\times e_3 =0 & \text{on } \Sigma_{-}.
\end{cases}
\eeq
Here $\epsilon>0$ is the artificial viscosity coefficient, $\Lambda:=(-\Delta_h)^{1/2}$  and   $\Psi$ is the so-called compensator  satisfying $\dt^j\Psi(0)=\mathcal{P}(\dt^j \bar \eta(0)),\ j=0,\dots, 2N-2$, where the initial data $\dt^j\eta(0)$ are those from the original problem \eqref{ppMHDv}. By the introduction of such $\Psi$, at time $t = 0$, one essentially adds nothing on the equations and its time derivatives up to $2N-2$ order. This allows one to take the initial data  for the problem \eqref{ppMHDv} exactly as the one for the regularized problem  \eqref{ppMHDvk} (the compatibility conditions to \eqref{ppMHDvk} are same as \eqref{ppMHDv}). It is crucial that the regularized problem \eqref{ppMHDvk} is asymptotically consistent with the a priori estimates of \eqref{ppMHDv}. Indeed, compared to \eqref{ppMHDv}, the artificial viscosity term leads to the gain of regularities for $\eta$ through the following  (cf. the first term in the right hand side of \eqref{ffhh}): for $\al\in \mathbb{N}^{1+2}$ with $|\al|\le 2N$,
\begin{align} 
- \int_{\Sigma }   \jump{ \pa^\al  p}    \epsilon    \pa^\al   \Lambda \eta 
=- \int_{\Sigma }   \sigma  \pa^\al   H          \epsilon  \pa^\al  \Lambda \eta    \le-  \sigma\epsilon \int_{\Sigma }  \abs{\nabla_h \Lambda^{1/2} \pa^\al  \eta}^2+C\sqrt{ \sen } \epsilon \sdn(\eta),
\end{align}
where
\beq\label{high_dissk}
\sdn(\eta):=\sum_{j=0}^{2N }\as{ \dt^j {\eta}}_{{2N}-j+3/2}.
\eeq
On the other hand,  all the new nonlinear terms resulting from the artificial viscosity term can be controlled by ${\sum}_{\mathcal{R}}$ with an exception (cf. the second term in the right hand side of \eqref{ppintro_11}), which is estimated as follows: by the trace theory,
\beq
-\int_{\Sigma}    \dt^{2N-1}  p_-  \jump{  u_h}\cdot    \nabla_h \epsilon \dt^{2N}\Lambda \eta\ls \abs{\dt^{2N-1}  p_- }_{1/2}\sqrt{\sen}\abs{\nabla_h \epsilon \dt^{2N}\Lambda \eta}_{-1/2}\ls  {\sen}\epsilon\sqrt{\sdn(\eta)}.
\eeq
The uniform in $\epsilon$ a priori estimates of \eqref{ppMHDvk} can be thus closed for $\sen$ small. It is also then routine to check that the regularized hydrodynamic part of \eqref{ppMHDvk} can be solved as the original one of \eqref{ppMHDv}.
Note that for each fixed $\epsilon>0$, the gained regularity of $\eta$ (cf. \eqref{high_dissk}) in particular validates the change of unknown in \eqref{ppbtr1}. Hence, one can first construct the solutions to the nonlinear $\epsilon$-approximate problem \eqref{ppMHDvk} by the method of successive approximations basing on the solvability of  the hydrodynamic and magnetic parts, whose limit as  $\epsilon\rightarrow 0$ yields the desired solution to  \eqref{ppMHDv}.
\end{proof}

\appendix

%%%%%%%%%%%%%%%%%%%%%%%%%%%%%%%%%%%%%%%%%%%%%%%
\section{Analytic tools}\label{section_app}
%%%%%%%%%%%%%%%%%%%%%%%%%%%%%%%%%%%%%%%%%%%%%%%

In this appendix we will collect the analytic tools which are used throughout the paper.

%%%%%%%%%%%%%%%%%%%%%%%%%%%%%%%%%%%%%%%%%%%%%%%
\subsection{Harmonic extension}
%%%%%%%%%%%%%%%%%%%%%%%%%%%%%%%%%%%%%%%%%%%%%%%

Define the specialized Poisson sum in $ \mt  \times \mathbb{R}$  by (see \cite{WTK})
\beq\label{poisson_def}
\mathcal{P} f(x) :=
\left\{\begin{array}{lll} \dis\sum_{\xi \in   \mathbb{Z} ^2 }  e^{2\pi i \xi  \cdot x_h} \sum_{j=0}^m\alpha_j e^{- |\xi|\lambda_jx_3} \hat{f}(\xi), & x_3> 0 
\\ \dis  \sum_{\xi \in     \mathbb{Z} ^2  }  e^{2\pi i \xi  \cdot x_h} e^{2\pi \abs{\xi }x_3} \hat{f}(\xi),& x_3\le 0,
\end{array}\right.
\eeq
where 
\beq
\hat{f}(\xi) = \int_{\mt}  f(x_h)   e^{-2\pi i \xi  \cdot x_h},\quad \xi\in  \mathbb{Z}^2  .
\eeq
Here  $0<\lambda_0<\lambda_1<\cdots<\lambda_m<\infty$ for $m\in \mathbb{N}$, and  $\alpha=(\alpha_0,\alpha_1,\dots,\alpha_m)^T$ is the solution to
\begin{equation}\label{Veq}
V(\lambda_0,\lambda_1,\dots,\lambda_m)\,\alpha= (1,1,\dots,1)^T,
\end{equation} 
where $V$ is the $(m+1) \times (m+1)$ Vandermonde matrix. The Poisson sum \eqref{poisson_def} is specialized in that  $\mathcal{P} f$ is differentially continuous across $\Sigma$ up to any order as needed provided that $m$ is sufficiently large. Moreover, the following estimate holds.
\begin{lem}\label{p_poisson}
It holds that for all $s\in \mathbb{R}$,
\beq
\norm{ \mathcal{P}f }_{s} \ls \abs{f}_{s-1/2}.
\eeq
\end{lem}
\begin{proof}
One may refer to Lemma A.9 of \cite{GT1} for instance.
\end{proof}

%%%%%%%%%%%%%%%%%%%%%%%%%%%%%%%%%%%%%%%%%%%%%%%
\subsection{Time extension}
%%%%%%%%%%%%%%%%%%%%%%%%%%%%%%%%%%%%%%%%%%%%%%%

The following lemmas allow one to extend the initial data to be time-dependent functions, which are ``hyperbolic" versions of the ``parabolic" ones in \cite{GT1} with some minor modifications.

\begin{lem}\label{l_sobolev_extension}
Suppose that $\dt^j \eta(0) \in H^{2N-j+3/2}(\Sigma)$ for $j=0,\dotsc,2N-1$.  There exists an extension $\eta^0$ defined on $[0,\infty)$, achieving the initial data,  so that
\begin{align}\label{kk11}
\sum_{j=0}^{2N+1}\sup_{[0,\infty]}\as{\dt^j \eta^0}_{ {2N-j+3/2}}+\sum_{j=0}^{2N+2} \int_0^\infty\as{\dt^j \eta^0}_{ {2N-j +2}}
\ls \sum_{j=0}^{2N-1}   \as{\dt^j \eta(0)}_{  {2N-j+3/2 }}.
\end{align}
\end{lem}
\begin{proof}
For each $j=0,\dotsc,2N-1$, we denote $f_j  = \dt^j \eta(0)$ and let $\varphi_j \in C_0^\infty(\Rn{})$ be such that $\varphi_j^{(k)}(0) = \delta_{j,k}$ for $k=0,\dotsc,2N-1$. Then
$\eta^0$ is constructed as a sum $\eta^0 = \sum_{j=0}^{2N-1} F_j$, where    $F_j$ is defined via its Fourier coefficients:
\begin{equation}
 \hat{F}_j(\xi,t) = \varphi_j(t \br{\xi}) \hat{f}_j(\xi)\br{\xi}^{-j},\ j=0,\dots,2N-1,\nonumber
\end{equation}
where $\br{\xi} = \sqrt{1+ \abs{\xi}^2}$. It follows by modifying the proof of  Lemma A.5 of \cite{GT1} suitably that $\eta^0$ satisfies the conclusion.
\end{proof}

\begin{lem}\label{l_sobolev_extension3}
Suppose that $\dt^j u(0) \in H^{2N-j}(\Omega_-)$ for $j=0,\dotsc,2N-1$.  There exists an extension $u^0$ defined on $[0,\infty)$, achieving the initial data,  so that
\begin{align}\label{kk12}
\sum_{j=0}^{2N}\sup_{[0,\infty]}\ns{\dt^j u^0}_{ {2N-j}}+\sum_{j=0}^{2N} \int_0^\infty\ns{\dt^j u^0}_{ {2N-j +1/2}}
\ls \sum_{j=0}^{2N-1}   \ns{\dt^j u(0)}_{  {2N-j }}.
\end{align}
\end{lem}
\begin{proof}
It follows similarly as Lemma \ref{l_sobolev_extension}, by using additionally the usual theory of extensions and restrictions in Sobolev spaces between $H^k(\Omega_-)$ and $H^k(\mathbb{R}^3)$ for $k\ge 0$.
\end{proof}

\begin{lem}\label{l_sobolev_extension2}
Suppose that $\dt^j b(0) \in H^{2N-j+1}(\Omega_-)$ for $j=0,\dotsc,2N-1$.  There exists an extension $b^0$ defined on $[0,\infty)$, achieving the initial data,  so that
\begin{align}\label{kk13}
\sum_{j=0}^{2N+1}\sup_{[0,\infty]}\ns{\dt^j b^0}_{ {2N-j+1}}+\sum_{j=0}^{2N+1} \int_0^\infty\ns{\dt^j b^0}_{ {2N-j +3/2}}
\ls \sum_{j=0}^{2N-1}   \ns{\dt^j b(0)}_{  {2N-j+1 }}.
\end{align}
\end{lem} 
\begin{proof}
It follows in the same way as Lemma \ref{l_sobolev_extension3}.
\end{proof}

%%%%%%%%%%%%%%%%%%%%%%%%%%%%%%%%%%%%%%%%%%%%%%%
\subsection{Product estimates}
%%%%%%%%%%%%%%%%%%%%%%%%%%%%%%%%%%%%%%%%%%%%%%%

The following standard estimates in Sobolev spaces are needed.
\begin{lem}\label{sobolev}
Let the domains below be  either $\Omega_\pm$, $\Sigma $ or $\Sigma_\pm$, and $d$ be the dimension.
\begin{enumerate}
\item Let $0\le r \le s_1 \le s_2$ be such that  $s_1 > d/2$. Then
\beq
\norm{fg}_{H^r} \lesssim \norm{f}_{H^{s_1}} \norm{g}_{H^{s_2}}.
\eeq
\item Let $0\le r \le s_1 \le s_2$ be such that  $s_2 >r+ d/2$.    Then
\beq\label{i_s_p_02}
\norm{fg}_{H^r} \lesssim \norm{f}_{H^{s_1}} \norm{g}_{H^{s_2}}.
\eeq 
\end{enumerate}
\end{lem}
 
\begin{lem}\label{sobolev-1}
It holds that  for $s>5/2$,
\beq\label{i_s_p_03}
 \norm{fg}_{-1} \ls \norm{f}_{-1} \norm{g}_{{s}}.
\eeq  
\end{lem}

%%%%%%%%%%%%%%%%%%%%%%%%%%%%%%%%%%%%%%%%%%%%%%%
\subsection{Poincar\'e-type inequality}
%%%%%%%%%%%%%%%%%%%%%%%%%%%%%%%%%%%%%%%%%%%%%%%

The following Poincar\'e-type inequality related to $\nabb$ holds.
\begin{lem}\label{lempoi}
For any constant vector $\bar B \in \mathbb{R}^3$ with $\bar B_3\neq 0$, it holds that  \beq\label{poincare_1}
\ns{ f}_0  \le \frac{1}{\bar B_3^2}\ns{(\bar B\cdot\nabla)f}_0+\as{ f}_0
\eeq
and
\beq\label{poincare_2}
\as{ f}_0  \le \frac{1}{\bar B_3^2}\ns{(\bar B\cdot\nabla)f}_0+\ns{ f}_0 .
\eeq
\end{lem}
\begin{proof}
It follows by the fundamental theorem of calculus, see Lemma A.4 in \cite{W}.
\end{proof}

%%%%%%%%%%%%%%%%%%%%%%%%%%%%%%%%%%%%%%%%%%%%%%%
\subsection{Normal trace estimates}
%%%%%%%%%%%%%%%%%%%%%%%%%%%%%%%%%%%%%%%%%%%%%%%

The following $H^{-1/2}$ boundary estimate holds for functions satisfying $v\in L^2$ and $\diva v \in L^2$.
\begin{lem}\label{Le_normal}
Assume that    $\norm{ \nabla \varphi}_{L^\infty} \leq C$, then
\beq\label{normal_es_3}
\abs{v \cdot \N}_{-1/2}\ls\norm{v}_0+\norm{\diva v}_0  .
\eeq
\end{lem}
\begin{proof}
One may refer to Lemma 3.3 in \cite{GT1}.  %Let $\psi \in H^{1/2}(\Sigma )$, and let $\tilde{\psi} \in H^1(\Omega)$ be a bounded extension which vanishes near $\Sigma_{-}$. Then
%\begin{align}
%\int_{\Sigma } \psi v \cdot \N   &= \int_\Omega \diva (\tilde{\psi}v) d \V = \int_\Omega \(\nabla^\varphi \tilde{\psi}\cdot v+\tilde{\psi} \diva v   \) d \V \nonumber
%\\ &\ls\norm{\tilde{\psi}}\norm{\diva v}_0 +\norm{\nabla \tilde{\psi}}_0\norm{v}_0 \ls \abs{ \psi }_{1/2}\(\norm{v}_0+\norm{\diva v}_0  \).\label{normal_2}
%\end{align}
%Then  \eqref{normal_es_3} follows immediately from \eqref{normal_2}.
\end{proof}

%%%%%%%%%%%%%%%%%%%%%%%%%%%%%%%%%%%%%%%%%%%%%%%
\subsection{Hodge-type estimates}
%%%%%%%%%%%%%%%%%%%%%%%%%%%%%%%%%%%%%%%%%%%%%%%
The following Hodge-type estimate  holds, when the boundary conditions are not specified. Let the domain be either $\Omega_-$ or $\Omega+$ or $\Omega$.
\begin{lem}\label{ell_hodge}
Let $r\ge 1$ be an integer. Then it holds that
\beq\label{v_ell_th}
\norm{v}_r\ls \norm{v}_{0,r}+\norm{ (\curl v)_h }_{r-1}+\norm{\Div v}_{r-1}  .
\eeq
\end{lem}
\begin{proof}
Notice that \eqref{v_ell_th} follows easily for $r=1$. Now for $r\ge 2$, applying the previous estimate of $r=1$ gives that for $\ell=1,\dots, r$,
\begin{align}\label{v_ell_th21}
\norm{v}_{\ell,r-\ell}&\ls \norm{v}_{\ell-1,r-\ell+1}+\norm{ (\curl v)_h }_{\ell-1,r-\ell}+\norm{\Div v}_{\ell-1,r-\ell} \nonumber
\\&\ls \norm{v}_{\ell-1,r-\ell+1}+\norm{ (\curl v)_h }_{r-1}+\norm{\Div v}_{r-1}  .
\end{align}
By an induction argument on  $\ell=1,\dots, r$, one gets \eqref{v_ell_th}.
\end{proof}

 %%%%%%%%%%%%%%%%%%%%%%%%%%%%%%%%%%%%%%%%%%%%%%%%%%%%%%
\section*{Acknowledgement}
%%%%%%%%%%%%%%%%%%%%%%%%%%%%%%%%%%%%%%%%%%%%%%%%%%%%%%

This research was supported by Zheng Ge Ru Foundation, HongKong RGC Earmarked Research Grants CUHK14305315, CUHK14302819, CUHK14300917, CUHK14302917, and Basic and Applied Basic Research Foundation of Guangdong Province 2020B1515310002. 
Y. J. Wang was also supported by the National Natural Science Foundation of China (No. 11771360) and the Natural Science Foundation of Fujian Province of China (No. 2019J02003).

\end{document}